\documentclass[a4paper,12pt,draft]{article}

\usepackage{amsthm}
\usepackage{latexsym}
\usepackage{dsfont}
\usepackage{bbm}
\usepackage{amssymb}
\usepackage{amsmath}
\usepackage{graphicx}
\usepackage{bibentry}		% allows to cite papers in the text

\usepackage{color}			% allows to mark changes in red

\numberwithin{equation}{section}

\theoremstyle{plain}
\newtheorem{Thm}{Theorem}[section]
\newtheorem{Lemma}[Thm]{Lemma}
\newtheorem{Cor}[Thm]{Corollary}
\newtheorem{Prop}[Thm]{Proposition}
\theoremstyle{remark}
\newtheorem{Rem}[Thm]{Remark}
\theoremstyle{definition}

\newtheorem{Def}[Thm]{Definition}

\newcommand{\N}{\mathbb{N}}
\newcommand{\integers}{\mathbb{Z}}
\newcommand{\Q}{\mathbb{Q}}
\newcommand{\R}{\mathbb{R}}

\newcommand{\Rp}{\R_{>}}
\newcommand{\Sd}{\mathbb{S}^{\mathit{d\!-\!1}}}

\newcommand{\C}{\mathbb{C}}

\newcommand{\V}{\mathbb{V}}
\newcommand{\G}{\mathcal{G}}
\newcommand{\imag}{\mathrm{i}}

\newcommand{\Prob}{\mathbb{P}}

\newcommand{\E}{\mathbb{E}}
\newcommand{\A}{\mathcal{A}}
\newcommand{\F}{\mathcal{F}}
\newcommand{\1}{\mathbbm{1}}

\newcommand{\bC}{\mathbf{C}}

\newcommand{\bX}{\mathbf{X}}
\newcommand{\bL}{\mathbf{L}}

\newcommand{\W}{\mathbf{W}}
\newcommand{\Z}{\mathcal{Z}}

\newcommand{\bnull}{\mathbf{0}}

\newcommand{\ba}{\mathbf{a}}
\newcommand{\bb}{\mathbf{b}}
\newcommand{\bc}{\mathbf{c}}

\newcommand{\bs}{\mathbf{s}}
\newcommand{\bt}{\mathbf{t}}
\newcommand{\bx}{\mathbf{x}}

\newcommand{\Smooth}{\mathrm{T}_{\Sigma}}
\newcommand{\SM}{\mathcal{S}(\mathcal{M})}
\newcommand{\SB}{\mathcal{S}(\mathcal{B})}
\newcommand{\SF}{\mathcal{S}(\mathfrak{F})}

\DeclareMathOperator{\sign}{\mathrm{sign}}

\newcommand{\eqdist}{\stackrel{\mathrm{law}}{=}}

\newcommand{\transp}{\mathsf{T}}

\newcommand{\dx}{\mathrm{d} \mathit{x}}
\newcommand{\dbx}{\mathrm{d}\mathbf{x}}
\newcommand{\dy}{\mathrm{d} \mathit{y}}
\newcommand{\dr}{\mathrm{d} \mathit{r}}
\newcommand{\ds}{\mathrm{d} \mathit{s}}
\newcommand{\dbs}{\mathrm{d}\mathbf{s}}

\title{Fixed points of multivariate smoothing transforms with scalar weights}
\author{Alexander Iksanov and Matthias Meiners}

\begin{document}

\nobibliography*    %   For the bibentry package

%\tableofcontents

\thispagestyle{empty}
\maketitle

\begin{abstract}
Given a sequence $(C_1,\ldots,C_d,T_1,T_2,\ldots)$ of real-valued random variables with $N := \#\{j \geq 1: T_j \not = 0\} < \infty$ almost surely,
there is an associated smoothing transformation which maps a distribution $P$ on $\R^d$ to the distribution of
$\sum_{j \geq 1} T_j \bX^{(j)} + \bC$ where $\bC = (C_1,\ldots,C_d)$ and $(\bX_j)_{j \geq 1}$ is a sequence of independent random vectors with distribution $P$ independent of $(C_1,\ldots,C_d,T_1,T_2,\ldots)$.
We are interested in the fixed points of this mapping.
By improving on the techniques developed in [\bibentry{Alsmeyer+Biggins+Meiners:2012}] and [\bibentry{Alsmeyer+Meiners:2013}],
we determine the set of all fixed points under weak assumptions on $(C_1,\ldots,C_d,T_1,T_2,\ldots)$.
In contrast to earlier studies, this includes the most intricate case when the $T_j$ take both positive and negative values with positive probability.
In this case, in some situations, the set of fixed points is a subset of the corresponding set when the $T_j$ are replaced by their absolute values,
while in other situations, additional solutions arise.
\vspace{0,1cm}


% applications of BPs
\hphantom{2010 Mathematics Subject Classification:}
Secondary:	39B32,				% eqs. for complex functions
			60G52,				% stable processes
			60E10				% characteristic functions
\end{abstract}

\section{Introduction}  \label{sec:Intro}

For a given $d \in \N$ and a given sequence $(\bC,T) =
((C_1,\ldots,C_d),(T_j)_{j \geq 1})$ where $C_1,\ldots,
C_d$, $T_1,T_2,\ldots$ are real-valued random variables with $N
:= \#\{j \geq 1: T_j \not = 0\} < \infty$ almost surely, consider
the mapping $\Smooth$ on the set of probability measures on $\R^d$
that maps a distribution $P$ to the law of the random variable
$\sum_{j \geq 1} T_j \bX^{(j)} + \bC$ where $(\bX^{(j)})_{j \geq
1}$ is a sequence of independent random vectors with distribution
$P$ independent of $(\bC,T)$. Here, $P$ is a fixed point of
$\Smooth$ iff, with $\bX$ denoting a random variable with
distribution $P$,
\begin{equation}    \label{eq:generalized stable inhom}
\bX ~\eqdist~ \sum_{j \geq 1} T_j \bX^{(j)} + \bC.
\end{equation}
In this paper we identify all solutions to \eqref{eq:generalized
stable inhom} under suitable assumptions.

Due to the
appearance of the distributional fixed-point equation \eqref{eq:generalized stable inhom}
in various applications such as interacting particle systems
\cite{Durrett+Liggett:1983}, branching random walks
\cite{Biggins:1977,Biggins+Kyprianou:1997}, analysis of algorithms
\cite{Neininger+Rueschendorf:2004,Roesler:1991,Volkovich+Litvak:2010},
and kinetic gas theory \cite{Bassetti+Ladelli+Matthes:2011}, there
is a large body of papers dealing with it in different settings.

The articles \cite{Alsmeyer+Biggins+Meiners:2012,Alsmeyer+Meiners:2012,Alsmeyer+Meiners:2013,Biggins:1977,Biggins+Kyprianou:1997,Biggins+Kyprianou:2005,Durrett+Liggett:1983,Iksanov:2004,Liu:1998}
treat the case $d=1$ in which we rewrite \eqref{eq:generalized
stable inhom} as
\begin{equation}    \label{eq:generalized stable inhom d=1}
X ~\eqdist~ \sum_{j \geq 1} T_j X^{(j)} + C.
\end{equation}
In all these references it is assumed that $T_j \geq 0$ a.s.~for
all $j \geq 1$. The most comprehensive result is provided in
\cite{Alsmeyer+Meiners:2013}. There, under mild assumptions on the
sequence $(C,T_1,T_2,\ldots)$, which include the existence of an
$\alpha \in (0,2]$ such that $\E [\sum_{j \geq 1} T_j^{\alpha}] =
1$, it is shown that there exists a couple $(W^*,W)$ of random
variables on a specified probability space such that $W^*$ is a
particular (\emph{endogenous}\footnote{See Section
\ref{subsec:endogeny} for the definition of {\em endogeny}.})
solution to \eqref{eq:generalized stable inhom d=1} and $W$ is a
nonnegative solution to the tilted homogeneous equation $W \eqdist \sum_{j
\geq 1} T_j^{\alpha} W^{(j)}$ where the $W^{(j)}$ are
i.i.d.~copies of $W$ independent of $(C,T_1,T_2,\ldots)$.
Furthermore, a distribution $P$ on $\R$ is a solution to
\eqref{eq:generalized stable inhom d=1} if and only if it is the
law of a random variable of the form
\begin{equation}    \label{eq:representation of solutions d=1}
W^* + W^{1/\alpha} Y_{\alpha}
\end{equation}
where $Y_{\alpha}$ is a strictly $\alpha$-stable random variable independent of $(W^*,W)$.\!\footnote{
For convenience, random variables with degenerate laws are assumed strictly $1$-stable.}
This result constitutes an
almost complete solution of the fixed-point problem in dimension
one leaving open only the case when the $T_j$ take positive and
negative values with positive probability. In a setup including
the latter case, which will be called the case of weights with
mixed signs hereafter, we derive the analogue of
\eqref{eq:representation of solutions d=1} thereby completing
the picture in dimension one under mild assumptions. It is
worth pointing out here that while one could guess at first glance
that \eqref{eq:representation of solutions d=1} carries over to
the case of weights with mixed signs with the additional
restriction that $Y_{\alpha}$ should be symmetric $\alpha$-stable
rather than strictly $\alpha$-stable, an earlier work
\cite{Alsmeyer+Roesler:2006} dealing with \eqref{eq:generalized
stable inhom d=1} in the particular case of deterministic weights
$T_j$, $j \geq 1$ suggests that this is not always the case.
Indeed, if, for instance, $1<\alpha<2$, $\E[\sum_{j \geq 1}
|T_j|^{\alpha}]=1$ and $\sum_{j \geq 1} T_j = 1$ a.s., it is
readily checked that addition of a constant to any solution again
gives a solution which cannot be expressed as in
\eqref{eq:representation of solutions d=1}. In fact, this is not
the only situation in which additional solutions arise and these
are typically not constants but limits of certain martingales not
appearing in the case of nonnegative weights $T_j$, $j \geq 1$.

Our setup is a mixture of the one- and the multi-dimensional
setting in the sense that we consider probability distributions on
$\R^d$ while the weights $T_j$, $j \geq 1$ are scalars.
Among others, this allows us to deal with versions
of \eqref{eq:generalized stable inhom} for stochastic processes
which can be understood as generalized equations of stability for
stochastic processes. Earlier papers dealing with the multi-dimensional case are
\cite{Bassetti+Matthes:2014,Mentemeier:2013}. On the one hand, the
setup in these references is more general since there the $T_j$,
$j \geq 1$ are $d \times d$ matrices rather than scalars. On the
other hand, they are less general since they cover the case
$\alpha = 2$ only \cite{Bassetti+Matthes:2014} (where the
definition of $\alpha$ is a suitable extension of the definition
given above) or the case of matrices $T_j$ with nonnegative
entries and solutions $\bX$ with nonnegative components only
\cite{Mentemeier:2013}.

We continue the introduction with a more detailed description of two applications, namely, kinetic models and stable distributions.

\subsubsection*{Kinetic models}

Bassetti {\it et
al.}~\cite{Bassetti+Ladelli:2012,Bassetti+Ladelli+Matthes:2011}
investigate the following kinetic-type evolution equation for a
time-de\-pen\-dent probability distribution $\mu_t$ on $\R$
\begin{equation}    \label{eq:kinetic-type evolution}
\partial_t \mu_t + \mu_t = \Smooth(\mu_t)
\end{equation}
where the smoothing transformation in \eqref{eq:kinetic-type evolution}, also called {\em collisional gain operator} in this context,
is associated to a sequence $(C,T)$ with $C=0$ and $N$ being
a fixed integer $\geq 2$. \eqref{eq:kinetic-type evolution}
generalizes the classical Kac equation \cite{Kac:1956} in which $T_1 = \sin(\Theta)$
and $T_2 = \cos(\Theta)$ for a random %??? Bogenlänge
angle $\Theta$ which is uniformly distributed over $[0,2 \pi]$,
and further inelastic Kac models \cite{Pulvirenti+Toscani:2004} in
which $T_1 = \sin(\Theta) |\sin(\Theta)|^{\beta-1}$ and $T_2 = \cos(\Theta) |\cos(\Theta)|^{\beta-1}$, $\beta>1$.
Also, one can show that the isotropic solutions of the multi-dimensional
inelastic Boltzmann equation \cite{Bobylev+Cercignani:2003} are
solutions to \eqref{eq:kinetic-type evolution}.

The stationary solutions to \eqref{eq:kinetic-type evolution} are
precisely the fixed points of $\Smooth$. For the Kac equation,
this results in the distributional fixed-point equation
\begin{equation}    \label{eq:Kac caricature}
X   ~\eqdist~   \sin(\Theta) X^{(1)} + \cos(\Theta) X^{(2)},
\end{equation}
while for the inelastic Kac equations, it results in
\begin{equation}    \label{eq:dissipative Kac caricature}
X   ~\eqdist~   \sin(\Theta) |\sin(\Theta)|^{\beta-1} X^{(1)} + \cos(\Theta) |\cos(\Theta)|^{\beta-1} X^{(2)}.
\end{equation}
It should not go unmentioned that \eqref{eq:kinetic-type evolution} is also used as a model for the distribution of wealth,
see \cite{Matthes+Toscani:2008} and the references therein.

%\subsubsection*{Analysis of algorithms}
%
%
%{\em PageRank}:
%The PageRank of a website is a certain measure for its importance.
%It is based on the following idea. Interpret the World Wide Web as a directed graph
%where websites are nodes and links are edges.
%Then consider an `easily bored' surfer who starts surfing at a certain website and,
%on each website, either starts afresh on a website chosen according to some fixed {\em teleportation distribution} (with probability $c>0$),
%or follows an outgoing link which is chosen uniformly at random among all outgoing links.
%The page rank of a website is proportional to the stationary distribution of the surfer.
%
%Volkovich and Litvak \cite{Volkovich+Litvak:2010} modeled the PageRank of a randomly chosen website $v$
%as a solution to \eqref{eq:generalized stable inhom} in dimension $d=1$.
%In this model, the inhomogeneity term $C$ is thought of as the contribution that comes from the teleportation distribution,
%while $N$ is the in-degree of $v$. $X_j$ and $T_j$ are the PageRank and a discount factor,
%which depends on the number of outgoing links, respectively, of the $j$th website that has a link to $v$.

\subsubsection*{Generalized equations of stability}

A distribution $P$ on $\R^d$ is called {\em stable} iff
there exists an $\alpha \in (0,2]$ such that for every $n \in \N$ there is a $\bc_n \in \R^d$ with
\begin{equation}    \label{eq:stable}
\bX ~\eqdist~   n^{-1/\alpha} \sum_{j=1}^n \bX^{(j)} + \bc_n
\end{equation}
where $\bX$ has distribution $P$ and $\bX^{(1)}, \bX^{(2)}, \ldots$ is a sequence of i.i.d.~copies of $\bX$,
see \cite[Corollary 2.1.3]{Samorodnitsky+Taqqu:1994}. $\alpha$ is called the {\em index of stability} and $P$ is called {\em $\alpha$-stable}.

Clearly, stable distributions are fixed points of certain smoothing transforms.
For instance, given a random variable $\bX$ satisfying \eqref{eq:stable} for all $n \in \N$,
one can choose a random variable $N$ with support $\subseteq \N$
and then define $T_1 = \ldots = T_n = n^{-1/\alpha}$, $T_j = 0$ for $j > n$
and $\bC = \bc_n$ on $\{N=n\}$, $n \in \N$.
Then $\bX$ satisfies \eqref{eq:generalized stable inhom}.

Hence, fixed-point equations of smoothing transforms can be considered as generalized equations of stability;
some authors call fixed points of smoothing transforms ``stable by random weighted mean'' \cite{Liu:2001}.
It is worth pointing out that the form of (the characteristic functions of) strictly stable distributions
can be deduced from our main result, Theorem \ref{Thm:SF},
the proof of which can be considered as a generalization of the classical derivation of the form of stable law given by
Gnedenko and Kolmogorov \cite{Gnedenko+Kolmogorov:1968}.

\section{Main results}

\subsection{Assumptions}    \label{subsec:assumptions}

Without loss of generality for the results considered here, we assume that
\begin{equation}
N   ~=~ \sup\{j \geq 1: T_j \not = 0\}  ~=~ \sum_{j \geq 1} \1_{\{T_j \not = 0\}}.
\end{equation}
Also, we define the function
\begin{equation}    \label{eq:m}
m:[0,\infty)    \to [0,\infty],
\quad
\gamma \mapsto \E \bigg[\sum_{j=1}^N |T_j|^{\gamma}\bigg].
\end{equation}
Naturally, assumptions on $(\bC,T)$ are needed in order to solve \eqref{eq:generalized stable inhom}.
Throughout the paper, the following assumptions will be in force:
\begin{align}
\Prob(T_j \in \{0\} \cup \{\pm r^n: n \in \integers\} \text{ for all } j \geq 1) < 1    \quad   \text{for all } r \geq 1.   \tag{A1}    \label{eq:A1}   \\
m(0) = \E[N]    >1. \tag{A2}    \label{eq:A2}   \\
m(\alpha) = 1 \text{ for some }\alpha > 0 \text{ and }
m(\vartheta)
> 1 \text{ for all } \vartheta \in [0,\alpha).  \tag{A3} \label{eq:A3}
\end{align}
We briefly discuss the assumptions
\eqref{eq:A1}-\eqref{eq:A3}\footnote{
Although (A3) implies (A2), we use both to keep the presentation consistent with earlier works.}
beginning with \eqref{eq:A1}. With $\R^*$ denoting the multiplicative group $(\R\setminus\{0\},\times)$, let
\begin{align*}
\mathbb{G}(T)
~:=~ \bigcap\big\{G: \ & G \text{ is a closed multiplicative subgroup of } \R^* \\
& \text{satisfying } \Prob(T_j \in G \text{ for } j=1,\ldots,N) = 1\big\}.
\end{align*}
$\mathbb{G}(T)$ is the closed multiplicative subgroup $\subset \R^*$ generated by the nonzero $T_j$.
%Here, $\R^*$ denotes the multiplicative group $(\R\setminus\{0\},\cdot)$.
There are seven possibilities: (C1) $\mathbb{G}(T) = \R^*$, (C2)
$\mathbb{G}(T) = \Rp :=(0,\infty)$, (D1) $\mathbb{G}(T) =
r^{\integers} \cup -r^{\integers}$ for some $r > 1$, (D2)
$\mathbb{G}(T) = r^{\integers}$ for some $r > 1$, (D3)
$\mathbb{G}(T) = (-r)^{\integers}$ for some $r > 1$, (S1)
$\mathbb{G}(T) = \{1,-1\}$, and (S2) $\mathbb{G}(T) = \{1\}$.
\eqref{eq:A1} can be reformulated as: Either (C1) or (C2)
holds. For the results considered, the cases (S1) and (S2) are
simple and it is no restriction to rule them out (see
\cite[Proposition 3.1]{Alsmeyer+Roesler:2006} in case $d=1$;
the case $d \geq 2$ can be treated by considering marginals).
Although the cases (D1)-(D3) in which the $T_j$ generate a (non-trivial) discrete
group could have been treated along the lines of this paper, they
are ruled out for convenience since they create the need for
extensive notation and case distinction.
Caliebe \cite[Lemma 2]{Caliebe:2003} showed that only simple cases
are eliminated when assuming \eqref{eq:A2}. \eqref{eq:A3} is natural in view of
earlier studies of fixed points of the smoothing transform, see
\textit{e.g.}\ \cite[Proposition 5.1]{Alsmeyer+Roesler:2006} and
\cite[Theorem 6.1 and Example 6.4]{Alsmeyer+Meiners:2012}. We
refer to $\alpha$ as the \emph{characteristic index (of $T$)}.

Define
\begin{equation}    \label{eq:pq}
p := \E \bigg[\sum_{j=1}^N |T_j|^{\alpha} \1_{\{T_j > 0\}}\bigg]    \quad\text{and}\quad    q := \E \bigg[\sum_{j=1}^N |T_j|^{\alpha} \1_{\{T_j < 0\}}\bigg].
\end{equation}
\eqref{eq:A3} implies that $0 \leq p,q \leq 1$ and $p + q = 1$. At some places it will be necessary to distinguish the following cases:

\begin{equation}    \label{eq:CaseI-III}
\text{\underline{Case I}: }     p = 1, q = 0.   \quad
\text{\underline{Case II}: }        p = 0, q = 1.   \quad
\text{\underline{Case III}: }       0 < p,q < 1.
\end{equation}
Case I corresponds to $\mathbb{G}(T) = \Rp$. Cases II and
III correspond to $\mathbb{G}(T) = \R^*$. In dimension $d=1$,
Case I is covered by the results in \cite{Alsmeyer+Meiners:2013} while
Case II can be lifted from these results. Case III is genuinely new.

In our main results, we additionally assume the following condition to be satisfied:
\begin{equation}    \tag{A4}    \label{eq:A4}
\eqref{eq:A4a} \text{ or } \eqref{eq:A4b} \text{ holds},
\end{equation}
where
\begin{equation}    \tag{A4a}   \label{eq:A4a}
\E \bigg[ \sum_{j=1}^N |T_j|^{\alpha} \log(|T_j|) \bigg] \in
(-\infty,0) \text{ and } \E \bigg[\bigg( \! \sum_{j=1}^N
|T_j|^{\alpha} \! \bigg) \log^+ \! \bigg( \! \sum_{j=1}^N
|T_j|^{\alpha} \! \bigg)\bigg] < \infty;
\end{equation}
\begin{equation}    \tag{A4b}   \label{eq:A4b}
\text{there exists some } \theta \in [0,\alpha) \text{
satisfying } m(\theta) ~<~ \infty.
\end{equation}
Further, in Case III when $\alpha = 1$, we need the assumption
\begin{align}
\E\bigg[\sum_{j=1}^N |T_j|^{\alpha} \delta_{-\log (|T_j|)}(\cdot) \bigg]
\text{ is spread-out,}  \quad
\E \bigg[\sum_{j=1}^N |T_j|^{\alpha} (\log^-(|T_j|))^2\bigg] <  \infty, \notag  \\
\E \bigg[ \sum_{j=1}^N |T_j|^{\alpha} \log (|T_j|) \bigg] \in (-\infty,0)
\quad   \text{and}  \quad
\E \bigg[h_3 \bigg(\sum_{j=1}^N |T_j|^{\alpha} \! \bigg) \bigg] < \infty    \tag{A5}    \label{eq:A5}
\end{align}
for $h_3(x) := x (\log^+ (x))^3 \log^+(\log^+ (x))$. \eqref{eq:A5}
is stronger than \eqref{eq:A4a}. The last assumption that will
occasionally show up is
\begin{equation}    \tag{A6}    \label{eq:A6}
|T_j| < 1   \text{ a.s.\ for all } j \geq 1.
\end{equation}
However, \eqref{eq:A6} will not be assumed in the main
theorems since by a stopping line technique, the general case can
be reduced to cases in which \eqref{eq:A6} holds. It will be stated
explicitly whenever at least one of the conditions \eqref{eq:A4a},
\eqref{eq:A4b}, \eqref{eq:A5} or \eqref{eq:A6} is assumed to
hold.

\subsection{Notation and background}    \label{subsec:notation}

In order to state our results, we introduce the underlying probability space and some notation that comes with it.

Let $\V := \bigcup_{n \geq 0} \N^n$ denote the infinite Ulam-Harris tree where $\N^0 := \{\varnothing\}$.
We use the standard Ulam-Harris notation, which means that we abbreviate $v = (v_1,\ldots,v_n) \in \V$ by $v_1 \ldots v_n$.
$vw$ is short for $(v_1,\ldots,v_n,w_1,\ldots,w_m)$ when $w = (w_1,\ldots,w_m) \in \N^m$.
We make use of standard terminology from branching processes and call the $v \in \V$ \emph{(potential) individuals}
and say that $v$ is a \emph{member of the $n$th generation} if $v \in \N^n$.
We write $|v|=n$ if $v \in \N^n$ and define $v|_k$ to be the restriction of $v$ to its first $k$ components if $k \leq |v|$ and $v|_k=v$, otherwise.
In particular, $v|_0 = \varnothing$. $v|_k$ will be called the \emph{ancestor of $v$ in the $k$th generation}.

Assume a family $(\bC(v),T(v))_{v \in \V} =
(C_1(v),\ldots,C_d(v),T_1(v),T_2(v),\ldots)_{v \in \V}$ of i.i.d.\
copies of the sequence $(\bC,T) = (\bC,T_1,T_2,\ldots)$ is given
on a fixed probability space $(\Omega,\A,\Prob)$ that also carries all
further random variables we will be working with.
For notational convenience, we assume that
\begin{equation*}
(C_1(\varnothing),\ldots,C_d(\varnothing),T_1(\varnothing),T_2(\varnothing),\ldots) ~=~ (C_1,\ldots,C_d,T_1,T_2,\ldots).
\end{equation*}
Throughout the paper, we let
\begin{equation}    \label{eq:A_n}
\A_n    ~:=~    \sigma((\bC(v),T(v)): |v|<n),   \qquad  n \geq 0
\end{equation}
be the $\sigma$-algebra of all family histories before the $n$th generation and define $\A_{\infty} := \sigma(\A_n: n \geq 0)$.

Using the family $(\bC(v),T(v))_{v \in \V}$, we define a Galton-Watson branching process as follows.
Let $N(v) := \sup\{j \geq 1: T_j(v) \not = 0\}$ so
that the $N(v)$, $v \in \V$ are i.i.d.~copies of $N$. Put $\G_0 :=
\{\varnothing\}$ and, recursively,
\begin{equation}\label{eq:G_n}
\G_{n+1}    ~:=~    \{vj \in \N^{n+1}: v \in \G_n, 1 \leq j \leq
N(v)\},    \qquad  n \in \N_0.
\end{equation}
Let $\G := \bigcup_{n \geq 0} \G_n$ and $N_n := |\G_n|$, $n \geq 0$.
Then $(N_n)_{n \geq 0}$ is a Galton-Watson process. $\E[N] > 1$ guarantees supercriticality and hence $\Prob(S) > 0$ where
\begin{equation*}
S ~:=~ \{N_n > 0 \text{ for all } n \geq 0\}
\end{equation*}
is the survival set.
Further, we define multiplicative weights $L(v)$, $v \in \V$ as follows. For $v = v_1 \ldots v_n \in \V$, let
\begin{equation}    \label{eq:L}
L(v)    ~:=~    \prod_{k=1}^{n} T_{v_k}(v|_{k-1}).
\end{equation}
Then the family $\bL := (L(v))_{v \in \V}$ is called \emph{weighted branching process}.
It can be used to iterate \eqref{eq:generalized stable inhom}.
Let $(\bX^{(v)})_{v \in \V}$ be a family of i.i.d.\ random variables defined on $(\Omega,\A,\Prob)$
independent of the family $(\bC(v),T(v))_{v \in \V}$.
For convenience, let $\bX^{(\varnothing)} =: \bX$.
If the distribution of $\bX$ is a solution to \eqref{eq:generalized stable inhom}, then, for $n \in \N_0$,
\begin{equation}    \label{eq:generalized_stable_inhom_iterated}
\bX ~\eqdist~   \sum_{|v|=n} L(v) \bX^{(v)} + \sum_{|v|<n} L(v) \bC(v).
\end{equation}
An important special case of \eqref{eq:generalized stable inhom}
is the homogeneous equation
\begin{equation}    \label{eq:generalized stable}
\bX ~\eqdist~ \sum_{j \geq 1} T_j \bX^{(j)}
\end{equation}
in which $\bC = \bnull = (0,\ldots,0) \in \R^d$ a.s.
Iteration of \eqref{eq:generalized stable} leads to
\begin{equation}    \label{eq:generalized_stable_iterated}
\bX  ~\eqdist~  \sum_{|v|=n} L(v) \bX^{(v)}.
\end{equation}

Finally, for $u \in \V$ and a function $\Psi=\Psi((\bC(v),T(v))_{v \in \V})$ of the weighted branching process,
let $[\Psi]_u$ be defined as $\Psi((\bC(uv),T(uv))_{v \in \V})$, that is, the same function but applied to the weighted branching process rooted in $u$.
The $[\cdot]_u$, $u \in \V$ are called {\em shift-operators}.

\subsection{Existence of solutions to \eqref{eq:generalized stable inhom} and related equations}

Under certain assumptions on $(\bC,T)$, a solution to \eqref{eq:generalized stable inhom}
can be constructed as a function of the weighted branching process $(L(v))_{v \in \V}$.
Let $\W_0^*:=\bnull$ and
\begin{equation}    \label{eq:W_n^*}
\W_n^*  ~:=~    \sum_{|v| < n} L(v) \bC(v), \quad n \in \N.
\end{equation}
$\W_n^*$ is well defined since a.s.~$\{|v| < n\}$ has
only finitely many members $v$ with $L(v) \not = 0$.
Whenever $\W_n^*$ converges in probability to a finite limit as $n \to \infty$, we set
\begin{equation}    \label{eq:W^*}
\W^*    ~:=~    \lim_{n \to \infty} \W_n^*
\end{equation}
and note that $\W^*$ defines a solution to \eqref{eq:generalized stable inhom}.
Indeed, if $\W_n^* \to \W^*$ in probability as $n \to \infty$,
then also $[\W_n^*]_j \to [\W^*]_j$ in probability as $n \to
\infty$. By standard arguments, there is a (deterministic)
sequence $n_k \uparrow \infty$ such that $[\W_{n_k}^*]_j \to
[\W^*]_j$ a.s.\ for all $j \geq 1$. Since $N<\infty$ a.s.,
this yields
\begin{equation*}
\lim_{k \to \infty} \W_{n_k+1}^*
~=~ \lim_{k \to \infty} \sum_{j=1}^N T_j [\W_{n_k}^*]_j + \bC
~=~ \sum_{j=1}^N T_j [\W^*]_j +  \bC
\quad   \text{a.s.}
\end{equation*}
and hence, 
\begin{equation}    \label{eq:W* fixed point}
\W^*   ~=~ \sum_{j=1}^N T_j [\W^*]_j   + \bC   \quad   \text{a.s.}
\end{equation}
because $\W_{n_k+1}^* \to \W^*$ in probability.

The following proposition provides sufficient conditions for
$\W_n^*$ to converge in probability.

\begin{Prop}    \label{Prop:Jelenkovic+Olvera:2012b}
Assume that \eqref{eq:A1}-\eqref{eq:A3} hold.
The following conditions are sufficient for $\W_n^*$ to converge in probability.
\begin{itemize}
    \item[(i)]
        For some $0 < \beta \leq 1$, $m(\beta) < 1$ and $\E[|C_j|^{\beta}] < \infty$ for all $j=1,\ldots,d$.
    \item[(ii)]
    	For some $\beta > 1$, $\sup_{n \geq 0} \int |x|^\beta \, \Smooth^n(\delta_{\bnull})(\dx) < \infty$
	and $T_j \geq 0$ a.s.~for all $j \in \N$ or $\E[C]=0$.
    \item[(iii)]
        $0 < \alpha < 1$, $\E[\sum_{j \geq 1} |T_j|^{\alpha} \log(|T_j|)]$ exists and equals $0$,
        and, for some $\delta > 0$, $\E [|C_j|^{1+\delta}] < \infty$ for $j=1,\ldots,d$.
\end{itemize}
\end{Prop}
For the most part, the proposition is known. Details along
with the relevant references are given at the end of Section
\ref{subsec:endogeny}.

Of major importance in this paper are the solutions to the one-dimensional tilted homogeneous fixed-point equation
\begin{equation}    \label{eq:W}
W   ~\eqdist~ \sum_{j \geq 1} |T_j|^{\alpha} W^{(j)}
\end{equation}
where $W$ is a finite, nonnegative random variable and the $W^{(j)}$, $j \geq 1$ are i.i.d.\ copies of $W$ independent of the sequence $(T_1,T_2,\ldots)$.
Equation \eqref{eq:W} (for nonnegative random variables) is equivalent to the functional equation
\begin{equation}    \label{eq:FE W}
f(t)    ~=~ \E \bigg[\prod_{j \geq 1} f(|T_j|^{\alpha} t)\bigg] \quad   \text{for all } t \geq 0
\end{equation}
where $f$ denotes the Laplace transform of $W$.
\eqref{eq:W} and \eqref{eq:FE W} have been studied extensively in the literature
and the results that are important for the purposes of this paper are summarized in the following proposition.

\begin{Prop}    \label{Prop:W}
Assume that \eqref{eq:A1}--\eqref{eq:A4} hold. Then
\begin{itemize}
	\item[(a)]
		there is a Laplace transform $\varphi$ of a probability distribution on $[0,\infty)$
		such that $\varphi(1) < 1$ and $\varphi$ solves \eqref{eq:FE W};
	\item[(b)]
		every other Laplace transform $\hat{\varphi}$ of a probability distribution on $[0,\infty)$ solving \eqref{eq:FE W}
		is of the form $\hat{\varphi}(t) = \varphi(ct)$, $t \geq 0$ for some $c \geq 0$;
    \item[(c)]  $1-\varphi(t)$ is regularly varying of index $1$ at $0$;
    \item[(d)]  a (nonnegative, finite) random variable $W$ solving \eqref{eq:W}
    with Laplace transform $\varphi$ can be constructed explicitly on $(\Omega,\mathcal{A},\Prob)$ via
                \begin{equation}   \label{eq:W_definition}
                W   ~:=~    \lim_{n \to \infty} \sum_{|v|=n} 1-\varphi(|L(v)|^{\alpha}) \quad   \text{a.s.}
                \end{equation}
\end{itemize}
\end{Prop}
\begin{proof}[Source]
(a), (b) and (c) are known. A unified treatment and references are given in
\cite[Theorem 3.1]{Alsmeyer+Biggins+Meiners:2012}.
(d) is contained in \cite[Theorem 6.2(a)]{Alsmeyer+Biggins+Meiners:2012}.
\end{proof}

Throughout the paper, we denote by $\varphi$ the Laplace transform
introduced in Proposition \ref{Prop:W}(a) and by $W$ the random
variable defined in \eqref{eq:W_definition}.
By Proposition \ref{Prop:W}(c), $D(t):=t^{-1}(1-\varphi(t))$ is slowly varying at $0$. If $D$ has
a finite limit at $0$, then, by scaling, we assume this limit to
be $1$. Equivalently, if $W$ is integrable, we assume $\E [W] =
1$. In this case, $W$ is the limit of the additive martingale
(sometimes called Biggins' martingale) in the branching random
walk based on the point process $\sum_{j=1}^N \delta_{-\log
(|T_j|^{\alpha})}$, namely, $W = \lim_{n \to \infty} W_n$
a.s.~where
\begin{equation}    \label{eq:Biggins' martingale}
W_n ~=~ \sum_{|v|=n} |L(v)|^{\alpha},   \qquad  n \in \N_0.
\end{equation}

As indicated in the introduction, for certain parameter constellations, another random variable plays an important role here.
Define
\begin{equation}    \label{eq:Z_n}
Z_n ~:=~ \sum_{|v|=n} L(v)  ,   \quad   n \in \N_0.
\end{equation}
Let $Z := \lim_{n \to \infty} Z_n$ if the limit exists in the a.s.\ sense and $Z=0$, otherwise.
The question of when $(Z_n)_{n \geq 0}$ is a.s.\ convergent is nontrivial.

\begin{Thm} \label{Thm:Z}
Assume that \eqref{eq:A1}-\eqref{eq:A4} are true.
Then the following assertions hold.
\begin{itemize}
	\item[(a)]
		If $0 < \alpha < 1$, then $Z_n \to 0$ a.s.~as $n \to \infty$.
	\item[(b)]
		If $\alpha > 1$, then $Z_n$ converges a.s.~and $\Prob(\lim_{n\to\infty} Z_n = 0)<1$
		iff $\E[Z_1] = 1$ and $Z_n$ converges in $\mathcal{L}^{\beta}$ for some/all $1<\beta<\alpha$.
		Further, for these to be true $(Z_n)_{n\geq 0}$ must be a martingale.
	\item[(c)]
		If $\alpha = 2$ and \eqref{eq:A4a} holds or $\alpha > 2$,
		then $Z_n$ converges a.s.~iff $Z_1 = 1$ a.s.
    \end{itemize}
\end{Thm}

Here are simple sufficient conditions for part (b) of the
theorem: if $1 < \alpha < 2$, $\E[Z_1]=1$, $\E [|Z_1|^{\beta}] <
\infty$ and $m(\beta) < 1$ for some $\beta > \alpha$, then
$(Z_n)_{n \geq 0}$ is an $\mathcal{L}^{\beta}$-bounded martingale
which converges in $\mathcal{L}^\beta$ and a.s. The assertion
follows in the, by now, standard way via an application of the
Topchi{\u\i}-Vatutin inequality for martingales. We omit further
details which can be found on p.~182 in
\cite{Alsmeyer+Kuhlbusch:2010} and in \cite{Roesler+Topchii+Vatutin:2000}.

If $\alpha = 1$, the behaviour of $(Z_n)_{n \geq 0}$ is irrelevant for us.
However, for completeness, we mention that if $\E [Z_1]=1$ or $\E[Z_1]=-1$,
then $(Z_n)_{n \geq 0} = (W_n)_{n \geq 0}$ or $(Z_n)_{n \geq 0} = ((-1)^n W_n)_{n \geq 0}$, respectively.
Criteria for $(W_n)_{n \geq 0}$ to have a nontrivial limit can be found in \cite{Alsmeyer+Iksanov:2009,Biggins:1977,Lyons:1997}.
If $\E [Z_1] \in (-1,1)$, then, under suitable assumptions, $Z_n \to 0$ a.s. We refrain from providing any details.

Theorem \ref{Thm:Z} will be proved in Section \ref{subsec:endogeny proofs}.

\subsection{Multivariate fixed points}

Most of the analysis concerning the equations \eqref{eq:generalized stable inhom} and \eqref{eq:generalized stable}
will be carried out in terms of Fourier transforms of solutions.
Indeed, \eqref{eq:generalized stable inhom} and \eqref{eq:generalized stable} are equivalent to
\begin{equation}    \label{eq:FE generalized stable inhom}
\phi(\bt)   ~=~ \E \bigg[e^{\imag \langle \bt, \bC \rangle} \prod_{j \geq 1} \phi(T_j \bt)\bigg]    \quad   \text{for all } \bt \in \R^d,
\end{equation}
and
\begin{equation}    \label{eq:FE generalized stable}
\phi(\bt)   ~=~ \E \bigg[\prod_{j \geq 1} \phi(T_j \bt)\bigg]
\quad   \text{for all } \bt \in \R^d,
\end{equation}
respectively. Here, $\langle\cdot,\cdot\rangle$ denotes the standard scalar product in $\R^d$ and $\imag$ the imaginary unit.
Let $\mathfrak{F}$ denote the set of Fourier transforms of probability distributions on $\R^d$ and
\begin{equation}    \label{eq:SF(C)}
\SF(\bC)    ~:=~    \{\phi \in \mathfrak{F}: \phi \text{ solves } \eqref{eq:FE generalized stable inhom}\}.
\end{equation}
Further, let $\SF := \SF(\bnull)$, that is,
\begin{equation}    \label{eq:SF}
\SF ~:=~    \{\phi \in \mathfrak{F}: \phi \text{ solves } \eqref{eq:FE generalized stable}\}.
\end{equation}
The dependence of $\SF(\bC)$ on $\bC$ is made explicit in the notation since at some points we will compare $\SF(\bC)$ and $\SF(\bnull)$.
The dependence of $\SF(\bC)$ and $\SF$ on $T$ is not made explicit because $T$ is kept fix throughout.

Henceforth, let $\Sd = \{\bx \in \R^d: |\bx|=1\}$ denote the unit sphere $\subseteq \R^d$.

\begin{Thm} \label{Thm:SF}
Assume \eqref{eq:A1}-\eqref{eq:A4}
and that $\W^*_n \to \W^*$ in probability\footnote{
When $\bC = \bnull$ a.s., then $\W^*_n \to \W^* = \bnull$ a.s.~as $n \to \infty$.}
as $n \to \infty$.
\begin{itemize}
    \item[(a)]
        Let $0 < \alpha < 1$.
        \begin{itemize}
            \item[(a1)]  Let $\mathbb{G}(T)=\Rp$.
            Then $\SF$ consists of the $\phi$ of the form
            \begin{equation}    \label{eq:phi alpha<1 R>0}
            \phi(\bt)
            =
            \E \! \bigg[\! \exp \! \bigg(\! \imag \langle\W^*\!, \bt \rangle - W \!\! \int \!\! |\langle\bt,\bs\rangle|^{\alpha}
            \Big[1- \imag \sign (\!\langle \bt, \bs \rangle\!) \tan \! \Big( \frac{\pi \alpha}{2} \Big) \Big] \, \sigma(\dbs) \! \bigg)\!\bigg]
            \end{equation}
            where $\sigma$ is a finite measure on $\Sd$.
            \item[(a2)] Let $\mathbb{G}(T)=\R^*$.
            Then $\SF$ consists of the $\phi$ of the form
            \begin{equation}    \label{eq:phi alpha<1 R*}
            \phi(\bt)
            =
            \E \! \bigg[\! \exp \! \bigg(\! \imag \langle\W^*\!, \bt \rangle - W \!\! \int \!\! |\langle\bt,\bs\rangle|^{\alpha} \, \sigma(\dbs) \! \bigg)\!\bigg]
            \end{equation}
            where $\sigma$ is a symmetric finite measure on $\Sd$.
        \end{itemize}
    \item[(b)]
        Let $\alpha = 1$.
        \begin{itemize}
            \item[(b1)]
                Let $\mathbb{G}(T)=\Rp$
                and assume that $\E[\sum_{j \geq 1} |T_j| (\log^-(|T_j|))^2] < \infty$.
                Then $\SF$ consists of the $\phi$ of the form
                \begin{equation}    \label{eq:phi alpha=1 R>0}
                \phi(\bt)
                =   \E \!\bigg[\!\exp\!\left(\!\imag \langle\W^*\!+W\ba, \!\bt \rangle - W \!\! \int \!\! |\langle \bt,\bs \rangle| \sigma(\ds)
                - \imag W \frac{2}{\pi} \!\int \! \langle \bt, \bs \rangle \log (|\!\langle \bt, \bs \rangle\!|) \sigma(\dbs) \!\right)\!\!\bigg]
                \end{equation}
                where $\ba \in \R^d$ and $\sigma$ is a finite measure on $\Sd$ with $\int s_k \, \sigma(\dbs) = 0$, $k=1,\ldots,d$.
                \item[(b2)]
                Let $\mathbb{G}(T)=\R^*$
                and assume that $\E[\sum_{j \geq 1} |T_j| (\log^-(|T_j|))^2] < \infty$ holds in Case II and that \eqref{eq:A5} holds in Case III.
                Then $\SF$ consists of the $\phi$ of the form
                \begin{equation}    \label{eq:phi alpha=1 R*}
                \phi(\bt) = \E \bigg[\exp\left(\imag \langle\W^*\!, \bt \rangle - W \int |\langle \bt,\bs \rangle| \, \sigma(\ds) \right)\bigg]
                \end{equation}
                where $\sigma$ is a symmetric finite measure on $\Sd$.
        \end{itemize}
    \item[(c)]
        Let $1 < \alpha < 2$.
        \begin{itemize}
            \item[(c1)]
                Let $\mathbb{G}(T)=\Rp$.
                Then $\SF$ consists of the $\phi$ of the form
                \begin{align}
                \phi(\bt)
                = \E \! \bigg[ \! \exp \! \bigg( \! & \imag \langle \W^*\!, \bt\rangle - W \!\! \int \!\! | \langle \bt, \bs \rangle |^{\alpha}
                \Big(1 \!-\! \imag \sign (\!\langle \bt, \bs \rangle\!) \tan \! \Big( \frac{\pi \alpha}{2} \Big) \Big) \, \sigma(\dbs) \! \bigg) \! \bigg]
                \label{eq:phi 1<alpha<2 R>0}
                \end{align}
                where $\sigma$ is a finite measure on $\Sd$.
            \item[(c2)]
                Let $\mathbb{G}(T)=\R^*$.
                Then $\SF$ consists of the $\phi$ of the form
                \begin{align}
                \phi(\bt)
                = \E \! \bigg[ \! \exp \! \bigg( \! & \imag \langle \W^* \!+\! Z\ba, \bt\rangle - W \!\! \int \!\! | \langle \bt, \bs \rangle |^{\alpha}
                \, \sigma(\dbs) \! \bigg) \! \bigg]
                \label{eq:phi 1<alpha<2 R*}
                \end{align}
                where $\ba \in \R^d$, $\sigma$ is a symmetric finite measure on $\Sd$,
                and $Z := \lim_{n \to \infty} Z_n$ if this limit exists in the a.s.~sense, and $Z=0$, otherwise.
        \end{itemize}
    \item[(d)]
        Let $\alpha = 2$.
        Then $\SF$ consists of the $\phi$ of the form
        \begin{equation}    \label{eq:phi alpha=2}
        \phi(\bt) ~=~   \E \bigg[\exp\bigg(\imag \langle \W^* + Z \ba, \bt \rangle - W \frac{\bt \Sigma \bt^{\!\transp}}{2}\bigg)\bigg]
        \end{equation}
        where $\ba \in \R^d$, $\Sigma$ is a symmetric positive semi-definite (possibly zero) $d \times d$ matrix
        and $\bt^{\!\transp}$ is the transpose of $\bt=(t_1,\ldots,t_d)$,
        and $Z := \lim_{n \to \infty} Z_n$ if this limit exists in the a.s.~sense, and $Z=0$, otherwise.
    \item[(e)]
        Let $\alpha > 2$.
        Then $\SF$ consists of the $\phi$ of the form
        \begin{equation}    \label{eq:phi alpha>2}
        \phi(\bt) ~=~   \E[\exp(\imag \langle \W^* + \ba, \bt \rangle)],
        \end{equation}
        where $\ba \in \R^d$. Furthermore, $\ba = \bnull$ if $\Prob(Z_1=1) < 1$.
\end{itemize}
\end{Thm}

Theorem \ref{Thm:SF} can be restated as follows.
When the assumptions of the theorem hold, a distribution $P$ on $\R^d$
is a solution to \eqref{eq:generalized stable inhom}
if and only if it is the law of a random variable of the form
\begin{equation}    \label{eq:representation of solutions}
\W^* + Z \ba + W^{1/\alpha} \mathbf{Y}_{\alpha}
\end{equation}
where $\W^*$ is the special (endogenous\footnote{See Section
\ref{subsec:endogeny} for the definition of {\em endogeny}.})
solution to the inhomogeneous equation, $Z$ is a special
(endogenous) solution to the one-dimensional homogeneous equation
(which vanishes in most cases, but can be nontrivial when $\alpha
> 1$), $\ba \in \R^d$, $W$ is a special (endogenous) nonnegative
solution to the tilted equation \eqref{eq:W}, and
$\mathbf{Y}_{\alpha}$ is a strictly $\alpha$-stable (symmetric
$\alpha$-stable if $\mathbb{G}(T) = \R^*$) random vector
independent of $(\bC,T)$.\!\footnote{ For convenience, random variables with degenerate laws are assumed
strictly $1$-stable here.} Hence, the solutions are scale mixtures
of strictly (symmetric if $\mathbb{G}(T) = \R^*$) stable
distributions with a random shift.
Theorem \ref{Thm:SF} in particular provides a deep insight into the structure
of all fixed points since stable distributions (see
\textit{e.g.}~\cite{Samorodnitsky+Taqqu:1994} and the references
therein) and the random variables $\W^*$, $W$, and $Z$ are well
understood. For instance, the tail behavior of solutions of the
form \eqref{eq:representation of solutions} can be derived from
the tail behavior of $\W^*$, $W$, $Z$, and $\mathbf{Y}_{\alpha}$.
The tail behavior of stable random variables is known, the tail
behavior of $W$ has been intensively investigated over the last
decades, see \textit{e.g.}~\cite{
Alsmeyer+Iksanov:2009,Alsmeyer+Kuhlbusch:2010,Biggins:1979,Biggins+Kyprianou:2005,Buraczewski:2009,Durrett+Liggett:1983,Iksanov:2004,Iksanov+Polotskiy:2006,Iksanov+Roesler:2006,Jelenkovic+Olvera:2012a,Liang+Liu:2011,Liu:1998,Liu:2000}.
Since the $T_j$ are
scalars in this paper, the tail behavior of $\W^*$ can be reduced
to the tail behavior of its (one-dimensional) components.
The latter has been investigated by several authors in the recent past
\cite{Alsmeyer+Damek+Mentemeier:2013,Buraczewski+Kolesko:2014,Jelenkovic+Olvera:2012b,Jelenkovic+Olvera:2012a}.
The tail behavior of $Z$ has been analysed in \cite{Alsmeyer+Damek+Mentemeier:2013}.

\subsection{Univariate fixed points}

Corollary \ref{Cor:set_of_solutions d=1} given next,
together with Theorems 2.1 and 2.2 of \cite{Alsmeyer+Meiners:2013}, provides a reasonably full
description of the one-dimensional fixed points of the homogeneous smoothing transforms in the case $\mathbb{G}(T)=\R^*$.

\begin{Cor} \label{Cor:set_of_solutions d=1}
Let $d=1$, $C=0$ and $\mathbb{G}(T)=\R^*$.
Assume that \eqref{eq:A1}-\eqref{eq:A4} hold true.
If $\alpha = 1$, additionally assume that $\E[\sum_{j \geq 1} |T_j| (\log^- (|T_j|))^2] < \infty$ in Case II
and \eqref{eq:A5} in Case III.
Then $\SF$ is composed of the $\phi$ of the form
\begin{equation}    \label{eq:solutions_general_form_non-lattice_alpha_not=1}
\phi(t) =
    \begin{cases}
    \E [ \exp(-\sigma^{\alpha} W |t|^\alpha)],
    & 0 < \alpha < 1, \\
    \E [ \exp(-\sigma W |t|)],
    & \alpha = 1, \\
    \E [ \exp(\imag a Z t -\sigma^{\alpha} W |t|^\alpha)],
    & 1 < \alpha < 2, \\
    \E [\exp(\imag a Z t -\sigma^2 W t^2)],
    & \alpha = 2,
    \end{cases}
\end{equation}
where $Z=\lim_{n \to \infty} Z_n$ if the limit exists in the
a.s.\ sense, and $Z=0$, otherwise. $\SF$ is empty when $\alpha
> 2$ unless $Z_1 = 1$ a.s., in which case $\SF = \{t
\mapsto\exp(\imag a t): a \not = 0\}$.
If $\alpha \in (0,1]$ or if $Z=0$, then $\sigma$ ranges over $(0,\infty)$.
Otherwise $(a,\sigma) \in \R \times (0,\infty)$.
\end{Cor}

\subsubsection*{The Kac caricature revisited}

As an application of Corollary \ref{Cor:set_of_solutions d=1},
we discuss  equations \eqref{eq:Kac caricature} and \eqref{eq:dissipative Kac caricature}.
In this context $d=1$,
$C=0$ and $T_1 = \sin(\Theta) |\sin(\Theta)|^{\beta-1}$,
$T_2 = \cos(\Theta) |\cos(\Theta)|^{\beta-1}$ and $T_j = 0$ for all $j \geq 3$
where $\Theta$ is uniformly distributed on $[0,2\pi]$.
Further, $\beta=1$ in the case of \eqref{eq:Kac caricature} and $\beta>1$ in the case of
\eqref{eq:dissipative Kac caricature}. In order to apply Corollary
\ref{Cor:set_of_solutions d=1}, we have to check whether
\eqref{eq:A1}-\eqref{eq:A4} and, when $\alpha = 1$, \eqref{eq:A5}
hold (note that we are in Case III).

Since $\Theta$ has a continuous distribution, \eqref{eq:A1}
and the spread-out property in \eqref{eq:A5} hold.
%\eqref{eq:A2} holds since $N = 2$ a.s.~and hence $m(0) = \E[N] =
%2$.
Further, for $\alpha = 2/\beta$ and $\vartheta \in [0,\alpha)$,
\begin{equation*}
|T_1|^{\alpha}+|T_2|^{\alpha}   ~=~ |\sin(\Theta)|^2 + |\cos(\Theta)|^2 ~=~ 1
\quad	\text{and}	\quad
|T_1|^\vartheta+|T_2|^\vartheta ~>~1	\quad	\text{a.s.}
\end{equation*}
Therefore, \eqref{eq:A3} (hence \eqref{eq:A2}) holds with
$\alpha=2/\beta$ and $W=1$. The latter almost immediately
implies \eqref{eq:A4a}. Moreover, since $|\sin(\Theta)|<1$ and $|\cos(\Theta)|<1$ a.s.,
$m$ is finite and strictly decreasing on $[0,\infty)$, in
particular the second condition in \eqref{eq:A5} holds (since $m$
is the Laplace transform of a suitable finite measure on
$[0,\infty)$, it has finite second derivative everywhere on
$(0,\infty)$). Further, when $\alpha=1$ (\textit{i.e.}~$\beta =
2$), the last condition in \eqref{eq:A5} is trivially fulfilled since
$|T_1|+|T_2|=1$. Finally, observe that $\E[Z_1] = 0$ which allows us to conclude from Theorem \ref{Thm:Z}(b) that $Z=0$
whenever $\alpha\in (1,2]$.

Now Corollary \ref{Cor:set_of_solutions d=1} yields

\begin{Cor} \label{Cor:Kac}
The solutions to \eqref{eq:Kac caricature} are precisely the
centered normal distributions, while the solutions to
\eqref{eq:dissipative Kac caricature} are precisely the symmetric $2/\beta$-stable distributions.
\end{Cor}

\subsection{The functional equation of the smoothing transform} \label{subsec:Results_functional_eq}

For appropriate functions $f$, call
\begin{equation}    \label{eq:FE}
f(t)    ~=~ \E \bigg[\prod_{j \geq 1} f(T_j t)\bigg]        \qquad
\text{for all } t
\end{equation}
{\it the functional equation of the smoothing transform}.
Understanding its properties is the key to solving
\eqref{eq:generalized stable}. \eqref{eq:FE} has
been studied extensively in the literature especially when $f$ is
the Laplace transform of a probability distribution on
$[0,\infty)$. The latest reference is
\cite{Alsmeyer+Biggins+Meiners:2012} where $T_j\geq 0$ a.s., $j\in\N$,
and decreasing functions $f:[0,\infty) \to [0,1]$ are
considered. Necessitated by the fact that we permit the random
coefficients $T_j$, $j \in\N$ in the main equations to take
negative values with positive probability, we need a two-sided
version of this functional equation. We shall determine all
solutions to \eqref{eq:FE} within the class $\mathcal{M}$ of functions $f: \R \to [0,1]$ that
satisfy the following properties:
\begin{itemize}
    \item[(i)]
        $f(0) = 1$ and $f$ is continuous at $0$;
    \item[(ii)]
        $f$ is nondecreasing
        on $(-\infty,0]$ and nonincreasing
        on $[0,\infty)$.
\end{itemize}
A precise description of $\SM$ which is the set of members
of $\mathcal{M}$ that satisfy \eqref{eq:FE} is given in the following theorem.

\begin{Thm} \label{Thm:d=1_2sided_FE}
Assume that \eqref{eq:A1}--\eqref{eq:A4} hold true and let $d=1$.
Then the set $\SM$ is given by the functions of the form
\begin{equation}    \label{eq:d=1_2sided_FE}
f(t)    ~=~ \begin{cases}
            \E [\exp(-Wc_{1} t^{\alpha})]			&   \text{for } t \geq 0,       \\
            \E [\exp(-Wc_{-1} |t|^{\alpha})]		&   \text{for } t \leq 0
            \end{cases}
\end{equation}
where $c_1, c_2 \geq 0$ are constants and $c_1=c_{-1}$ if $\mathbb{G}(T) = \R^*$.
\end{Thm}

This theorem is Theorem 2.2 of
\cite{Alsmeyer+Biggins+Meiners:2012} in case that all $T_j$ are
nonnegative. In Section \ref{subsec:one-dimensional_FE},
we prove the extension to the case when the $T_j$ take negative
values with positive probability.

The rest of the paper is structured as follows.
The proof of our main result, Theorem \ref{Thm:SF} splits into two
parts, the direct part and the converse part. The direct
part is to verify that the Fourier transforms given in
\eqref{eq:phi alpha<1 R>0}-\eqref{eq:phi alpha=2} are actually
members of $\SF$; this is done in Section \ref{subsec:direct}. The
converse part is to show that any $\phi \in \SF$ is of the form as
stated in the theorem. This requires considerable efforts and
relies heavily on the properties of the weighted branching process
introduced in Section \ref{subsec:notation}.
The results on this branching process which we need in the proofs
of our main results are provided in Section \ref{sec:BP}. In
Section \ref{sec:proofs of main results}, we first solve the
functional equation of the smoothing transform in the case
$\mathbb{G}(T)=\R^*$ (Section \ref{subsec:one-dimensional_FE}).
Theorem \ref{Thm:Z} is proved in Section \ref{subsec:endogeny
proofs}. The homogeneous equation \eqref{eq:generalized stable} is
solved in Section \ref{subsec:solving homogeneous}, while the
converse part of Theorem \ref{Thm:SF} is proved in Section
\ref{subsec:solving inhomogeneous}.

The scheme of the proofs follows that in \cite{Alsmeyer+Biggins+Meiners:2012,Alsmeyer+Meiners:2012,Alsmeyer+Meiners:2013}.
Repetitions cannot be avoided entirely and short arguments from the cited sources are occasionally repeated to make the paper at hand more self-contained.
However, we omit proofs when identical arguments could have been given and provide only sketches of proofs when the degree of similarity is high.

\section{Branching processes}   \label{sec:BP}

In this section we provide all concepts and tools from the theory of branching processes
that will be needed in the proofs of our main results.

\subsection{\!Weighted branching and the branching random walk} \label{subsec:WBP and BRW}

Using the weighted branching process $(L(v))_{v \in \V}$ we
define a related \emph{branching random walk}
$(\Z_n)_{n \geq 0}$ by
\begin{equation}    \label{eq:M_n}
\Z_n    ~:=~    \sum_{v \in \G_n} \delta_{S(v)}
\end{equation}
where $S(v) := -\log (|L(v)|)$, $v \in \V$ and $\G_n$ is the set
of individuals residing in the $n$th generation, see \eqref{eq:G_n}.
By $\mu$ we denote the intensity measure of the point process $\Z := \Z_1$, \textit{i.e.}, $\mu(B) := \E [\Z(B)]$ for Borel sets $B \subseteq \R$. $m$ (defined in \eqref{eq:m}) is the Laplace transform of $\mu$, that is, for $\gamma \in \R$,
\begin{equation*}
m(\gamma) ~=~ \int e^{-\gamma x} \, \mu(\dx)    ~=~ \E \bigg[\sum_{j=1}^N e^{-\gamma S(v)} \bigg].
\end{equation*}
By nonnegativity, $m$ is well defined on $\R$ but may assume the
value $+\infty$. \eqref{eq:A3} guarantees $m(\alpha)=1$. This
enables us to use a classical exponential change of measure. To be
more precise, let $(S_n)_{n \geq 0}$ denote a zero-delayed random
walk with increment distribution $\Prob(S_1 \in \dx) :=
\mu_{\alpha}(\dx) := e^{-\alpha x} \mu(\dx)$. It is well known
(see \textit{e.g.}\ \cite[Lemma 4.1]{Biggins+Kyprianou:1997}) that
then, for any given $n \in \N_0$, the distribution of $S_n$ is
given by
\begin{equation}    \label{eq:spinal walk}
\Prob(S_n \in B)    ~=~ \E \bigg[\sum_{|v|=n} |L(v)|^\alpha \1_B(S(v))\bigg],   \quad   B \subset \R    \text{ Borel.}
\end{equation}

\subsection{Auxiliary facts about weighted branching processes} \label{subsec:basics_WBP}

\begin{Lemma}   \label{Lem:sup->0}
If \eqref{eq:A1}--\eqref{eq:A3} hold, then $\inf_{|v|=n} S(v) \to \infty$ a.s.\ on $S$ as $n \to \infty$.
Equivalently, $\sup_{|v|=n} |L(v)| \to 0$ a.s.\ as $n \to \infty$.
\end{Lemma}
\begin{proof}[Source]
This is \cite[Theorem 3]{Biggins:1998}.
\end{proof}

The following lemma will be used to reduce Case II to Case I.
\begin{Lemma}   \label{Lem:WBP_even}
Let the sequence $T$ satisfy \eqref{eq:A1}--\eqref{eq:A3}. Then so does the sequence $(L(v))_{|v|=2}$.
If \eqref{eq:A4a} or \eqref{eq:A4b} holds for $T$, then \eqref{eq:A4a} or \eqref{eq:A4b}, respectively, holds for $(L(v))_{|v|=2}$.
If, moreover, $\E[\sum_{j \geq 1} |T_j|^{\alpha} (\log^- (|T_j|))^2] < \infty$, then the same holds for the sequence $(L(v))_{|v|=2}$.
\end{Lemma}
\begin{proof}
Throughout the proof we assume that $(T_j)_{j \geq 1}$
satisfies $m(\alpha)=1$ which is the first part of \eqref{eq:A3}.
Then $\E[\sum_{|v|=2}L(v)^{\alpha}]=1$ which is the first part of
\eqref{eq:A3} for $(L(v))_{|v|=2}$. We shall use \eqref{eq:spinal
walk} to translate statements for $(T_j)_{j \geq 1}$ and
$(L(v))_{|v|=2}$ into equivalent but easier ones for $S_1$ and
$S_2$. \eqref{eq:A1} for $(T_j)_{j \geq 1}$ corresponds to $S_1$
being nonlattice. But if $S_1$ is nonlattice, so is $S_2$. The
second part of \eqref{eq:A3} for $(T_j)_{j \geq 1}$ corresponds to
$\E [e^{\vartheta S_1}]>1$ which implies $\E [e^{\vartheta S_2}]>1$.
The same argument applies to \eqref{eq:A4b}.
%It is easy to see that validity of \eqref{eq:A2} and \eqref{eq:A3}
%for $(T_j)_{j \geq 1}$ implies validity of \eqref{eq:A2} and
%\eqref{eq:A3} for $(L(v))_{|v|=2}$. For the rest of the proof,
%assume that $(T_j)_{j \geq 1}$ satisfies \eqref{eq:A3}. We will
%use \eqref{eq:spinal walk} to translate statements for $(T_j)_{j
%\geq 1}$ and $(L(v))_{|v|=2}$ into equivalent but easier ones for
%$S_1$ and $S_2$. \eqref{eq:A1} for $(T_j)_{j \geq 1}$ corresponds
%to $S_1$ being nonlattice. But if $S_1$ is nonlattice, so is
%$S_2$.
%   If $S_2$ is lattice with lattice span $1$, say, then $\psi(2\pi n)^2=1$ for all $n \in \integers$ while $\psi(t)^2 \not = 1$, otherwise,
%   for the Fourier transform $\psi$ of $S_1$. Then either $\psi(2 \pi)=1$ and $S_1$ is 1-lattice or $\psi(2 \pi)=-1$.
%   This implies that $\Prob(2 \pi S_1 \in \{\cos = -1\}) = \Prob(2 \pi S_1 \in \pi + 2 \pi \integers) = 1$, in particular, $S_1$ is $1/2$-lattice.
%\eqref{eq:A4b} for $(T_j)_{j \geq 1}$ corresponds to $\E
%[e^{(\alpha-\theta)S_1}] < \infty$ which implies $\E
%[e^{(\alpha-\theta)S_2}] < \infty$.
The first condition in \eqref{eq:A4a} for $(T_j)_{j \geq 1}$,
$m'(\alpha) \in (-\infty,0)$, translates into $\E [S_1] \in
(0,\infty)$. This implies $\E [S_2] = 2 \E [S_1] \in (0,\infty)$
which is the first condition in \eqref{eq:A4a} for
$(L(v))_{|v|=2}$. As to the second condition in \eqref{eq:A4a},
notice that validity of \eqref{eq:A4a} for $(T_j)_{j \geq 1}$ in
combination with Biggins' theorem \cite{Lyons:1997} implies that
$W_n \to W$ as $n \to \infty$ in mean. Then $\sum_{|v|=2n}
|L(v)|^{\alpha}$ also converges in mean to $W$. Using the converse
implication in Biggins' theorem gives that  $(L(v))_{|v|=2}$
satisfies the second condition in \eqref{eq:A4a} as well. Finally,
$\E[\sum_{j \geq 1} |T_j|^{\alpha} (\log^-(|T_j|))^2]<\infty$
translates via \eqref{eq:spinal walk} into $\E [(S_1^+)^2] <
\infty$. Then $\E [(S_2^+)^2] \leq \E [(S_1^+ + (S_2-S_1)^+)^2] < \infty$.
\end{proof}

\subsection{Multiplicative martingales and infinite divisibility}

We shall investigate
the functional equation
\begin{equation}    \label{eq:multi_FE}
f(\bt) ~=~ \E \bigg[\prod_{j \geq 1} f(T_j \bt)\bigg],      \qquad  \bt \in \R^d
\end{equation}
within the set $\mathfrak{F}$ of Fourier transforms of probability
distributions on $\R^d$ and, for technical reasons, for $d=1$
within the class $\mathcal{M}$ introduced in Section
\ref{subsec:Results_functional_eq}. In order to at one go include
the functions of $\mathfrak{F}$ and $\mathcal{M}$ in our analysis,
we introduce the class $\mathcal{B}$ of measurable functions $f:
\R^d \to \C$ satisfying $\sup_{\bt \in \R^d} |f(\bt)| \leq 1$ and
$f(0)=1$. Then $\mathfrak{F} \subseteq \mathcal{B}$ and, when
$d=1$, $\mathcal{M} \subseteq \mathcal{B}$. By $\SB$ we denote the
the class of $f \in \mathcal{B}$ satisfying \eqref{eq:multi_FE}.

For an $f \in \SB$, we define the corresponding multiplicative martingale
\begin{equation} \label{eq:disintegrated}
M_n(\bt)    ~:=~    M_n(\bt,\bL)    ~:=~ \prod_{|v|=n} f(L(v)\bt),
\qquad n \in \N_0, \ \bt \in \R^d.
\end{equation}
The notion \emph{multiplicative martingale} is justified by the following lemma.

\begin{Lemma}   \label{Lem:multiplicative_m'gale}
Let $f \in \SB$ and $\bt \in \R^d$. Then $(M_n(\bt))_{n \geq 0}$
is a bounded martingale w.r.t.\ $(\A_n)_{n \geq 0}$ and
thus converges a.s.~and in mean to a random variable
$M(\bt) := M(\bt,\bL)$ satisfying
\begin{equation}    \label{eq:Disintegration_integrated}
\E [M(\bt)] ~=~ f(\bt).
\end{equation}
\end{Lemma}
\begin{proof}[Source]
Minor modifications in the proof of \cite[Theorem 3.1]{Biggins+Kyprianou:1997} yield the result.
\end{proof}

\begin{Lemma}   \label{Lem:M's equation}
Given $f \in \SB$, let $M$ denote the limit of the associated multiplicative martingales.
Then, for every $\bt \in \R^d$,
\begin{equation}    \label{eq:M's equation}
M(\bt) ~=~ \prod_{|v|=n} [M]_v(L(v)\bt) \quad   \text{a.s.}
\end{equation}
The identity holds for all $\bt \in \R^d$ simultaneously a.s.\ if $f \in \SF$.
\end{Lemma}
\begin{proof}
For $n \in \N_0$, we have $|\{|v|=n\}|<\infty$ a.s., and hence
\begin{equation*}
M(\bt)
~=~	\lim_{k\to\infty} \prod_{|v|=n} \prod_{|w|=k} f(L(vw) \bt)
~=~	\prod_{|v|=n} \lim_{k\to\infty}  \prod_{|w|=k} f([L(w)]_v L(v) \bt)
~=~	\prod_{|v|=n} [M]_v(L(v)\bt)
\end{equation*}
for every $\bt \in \R^d$ a.s. For $f \in \SF$, by standard arguments, the identity holds for all $\bt \in \R^d$ simultaneously a.s.
%The arguments used in the proof of \cite[Lemma 4.3]{Alsmeyer+Meiners:2013} give the result.
\end{proof}

Before we state our next result, we remind the reader that a measure $\nu$ on the Borel sets of $\R^d$ is called a
\emph{L\'evy measure} if $\int (1 \wedge | \bx |^2) \, \nu(\dbx) < \infty$, see \textit{e.g.}\
\cite[p.\;290]{Kallenberg:2002}.
In particular, any L\'evy measure assigns finite mass to sets of the form $\{\bx \in \R^d: | \bx | \geq \varepsilon\}$, $\varepsilon > 0$.

\begin{Prop}    \label{Prop:Disintegration}
Let $ \phi \in \SF$ with associated multiplicative martingales $(\Phi_n(\bt))_{n \geq 0}$
and martingale limit $\Phi(\bt)$, $\bt \in \R^d$.
Then, a.s.\ as $n \to \infty$, $(\Phi_n)_{n \geq 0}$ converges pointwise to a random characteristic function $\Phi$ of the form $\Phi = \exp(\Psi)$ with
\begin{equation}    \label{eq:char_exponent}
\Psi(\bt)   ~=~ \imag \langle \W, \bt \rangle - \frac{\bt \mathbf{\Sigma} \bt^{\!\textsf{T}}}{2}
+ \int \left(e^{\imag \langle \bt,\bx \rangle} - 1 - \frac{\imag \langle \bt, \bx \rangle}{1+| \bx |^2} \right) \, \nu(\dbx),
\quad   \bt \in \R^d,
\end{equation}
where $\W$ is an $\R^d$ valued $\bL$-measurable random variable, $\mathbf{\Sigma}$ is a random, $\bL$-measurable positive semi-definite $d \times d$ matrix, and $\nu$ is a (random) L\'evy measure on $\R^d$ such that, for any $\bt \in \R^d$, $\nu([\bt,\infty))$ and $\nu((-\infty,-\bt])$ are $\bL$-measurable. Moreover,
\begin{equation}    \label{eq:disintegration_integrated}
\E [\Phi(\bt)] ~=~ \phi(\bt)    \quad   \text{for all } \bt \in \R^d.
\end{equation}
\end{Prop}
This proposition is the $d$-dimensional version of Theorem 1 in
\cite{Caliebe:2003} and can be proved analogously.\!\footnote{The
proof of Theorem 1 in \cite{Caliebe:2003} contains an
inaccuracy that needs to be corrected. Retaining the notation of
the cited paper, we think that it cannot be excluded that the set
of continuity points $\mathcal{C}$ of the function $F(l)$
appearing in the proof of Theorem 1 in \cite{Caliebe:2003} depends
on $l$. In the cited proof, this dependence is ignored when the
limit $\lim_{u \to \infty, u \in \mathcal{C}}$ appears outside the
expectation on p.\;386. However, this problem can be overcome by
using a slightly more careful argument.} Therefore, we refrain
from giving further details.

Now pick some $f \in \SB$.
%The proof of Lemma 8.5 in \cite{Alsmeyer+Biggins+Meiners:2012}
The proof of Lemma \ref{Lem:M's equation} applies and gives the counterpart
of \eqref{eq:M's equation}
\begin{equation}    \label{eq:M's equation along ladder lines}
M(\bt)  ~=~ \prod_{v \in \mathcal{T}_u} [M]_v(L(v) \bt)     \quad   \text{a.s.}
\end{equation}
for $\mathcal{T}_u := \{v \in \G: S(v) > u, \, S(v|_k) \leq u \text{ for } 0 < k < |v|\}$, $u \geq 0$.
Taking expectations reveals that $f$ also solves the functional equation with the weight
sequence $(L(v))_{v \in \mathcal{T}_u}$ instead of the sequence
$(T_j)_{j \geq 1}$. Further, when $f \in \SF$, the proofs of Lemmas
8.7(b) in \cite{Alsmeyer+Biggins+Meiners:2012} and 4.4 in
\cite{Alsmeyer+Meiners:2013} carry over to the present situation
and yield
\begin{equation}    \label{eq:M' as limit along ladder lines}
M(\bt)  ~=~ \lim_{u \to \infty }\prod_{v \in \mathcal{T}_u} f(L(v) \bt) ~=:~    \lim_{u \to \infty} M_{\mathcal{T}_u}(\bt)  \quad   \text{for all $\bt$ in $\R^d$ a.s.}
\end{equation}
This formula allows us to derive useful representations for the
random L\'evy triplet of the limit $\Phi$ of the multiplicative
martingale corresponding to a given $\phi \in \SF$. Denote
by $\overline{\R^d} = \R^d \cup \{\mathbf{\infty}\}$ the one-point
compactification of $\R^d$.

\begin{Lemma}   \label{Lem:explicit_representation_along_ladder_lines}
Let $\bX$ be a solution to \eqref{eq:generalized stable} with characteristic function $\phi$ and $d$-dimensional distribution (function) $F$.
Let further $(\W,\mathbf{\Sigma},\nu)$ be the random L\'evy triplet of the limit $\Phi$ of the multiplicative martingale corresponding to $\phi$,
see Proposition \ref{Prop:Disintegration}.
Then
\begin{equation}    \label{eq:nu_ladder}
\sum_{v \in \mathcal{T}_u} F(\cdot/L(v))    ~\stackrel{\mathrm{v}}{\to}~  \nu     \quad   \text{as $u \to \infty$ a.s.}
\end{equation}
where $\stackrel{\mathrm{v}}{\to}$ denotes vague convergence on
$\overline{\R^d}\setminus\{\bnull\}$.
Further, for any $h > 0$ with $\nu(\{|\bx| = h\}) = 0$ a.s., the limit
\begin{equation}    \label{eq:W(h)}
\W(h)   ~:=~    \lim_{t \to \infty} \sum_{v \in \mathcal{T}_t} L(v) \int_{\{|\bx| \leq h/|L(v)|\}} \bx \, F(\dbx)
\end{equation}
exists a.s.~and
\begin{equation}    \label{eq:bW ladder}
\W  ~=~ \W(h) + \!\!\! \underset{\{h < |\bx| \leq 1\}}{\int} \bx \, \nu(\dbx)
		+ \!\!\! \underset{\{|\bx| > 1\}}{\int} \frac{\bx}{1+|\bx|^2} \, \nu(\dbx)
		- \!\!\! \underset{\{|\bx| \leq 1\}}{\int} \frac{\bx |\bx|^2}{1+|\bx|^2} \, \nu(\dbx) \quad a.s.
\end{equation}
\end{Lemma}
\begin{proof}
First, notice that by \eqref{eq:M' as limit along ladder lines}, for fixed $\bt \in \R^d$, we have
\begin{equation*}
\Phi_{\mathcal{T}_u}(\bt)   ~:=~    \prod_{v \in \mathcal{T}_u} \phi(L(v)\bt)
~\to~   \Phi(\bt)   ~=~ \lim_{n \to \infty} \prod_{|v|=n} \phi(L(v)\bt) \quad   \text{a.s.}
\end{equation*}
along any fixed sequence $u \uparrow \infty$.
$\Phi_{\mathcal{T}_u}(\bt)$ is a uniformly integrable martingale
in $u$ with right-continuous paths and therefore the convergence
holds outside a $\Prob$-null set for all sequences $u \uparrow
\infty$. Using the a.s.~continuity of $\Phi$ on $\R^d$
(see Proposition \ref{Prop:Disintegration}), standard
arguments show that the convergence holds for all $\bt \in \R^d$
and all sequences $u \uparrow \infty$ on an event of probability one,
\textit{cf.}~the proof \cite[Lemma 4.4]{Alsmeyer+Meiners:2013}.
On this event, one can use the
theory of triangular arrays as in the proof of Proposition
\ref{Prop:Disintegration} to infer that $\Phi$ has a
representation $\Phi = \exp(\Psi)$ with $\Psi$ as in
\eqref{eq:char_exponent}. Additionally, Theorem 15.28(i) and (iii)
in \cite{Kallenberg:2002} give \eqref{eq:nu_ladder} and \eqref{eq:W(h)},
respectively. Note that the integrand of \eqref{eq:char_exponent}
being $(e^{\imag \langle \bt,\bx \rangle} - 1 - \imag \langle \bt,
\bx \rangle/(1+| \bx |^2))$ rather than $(e^{\imag \langle \bt,\bx
\rangle} - 1 - \imag \langle \bt, \bx \rangle \1_{\{|\bx| \leq
1\}})$ as it is in \cite{Kallenberg:2002} (see \textit{e.g.}\
Corollary 15.8 in the cited reference) does not affect $\nu$ but
it does influence $\W$.
The integrals ${\int}_{\{|\bx| > 1\}} \bx/(1+|\bx|^2) \nu(\dbx)$
and ${\int}_{\{|\bx| \leq 1\}} \bx |\bx|^2/(1+|\bx|^2) \nu(\dbx)$ appearing in \eqref{eq:bW ladder}
are the corresponding compensation.
\end{proof}

\subsection{The embedded BRW with positive steps only}  \label{subsec:embedded}

In this section, an embedding technique, invented
in \cite{Biggins+Kyprianou:2005}, is explained. This approach is used
to reduce cases in which \eqref{eq:A6} does not hold to cases where it does.

Let $\G^>_0 := \{\varnothing\}$, and, for $n\in\N$,
\begin{equation*}
\G_n^> := \{vw \in \G: v \in \G_{n-1}^>, S(vw) > S(v) \geq S(vw|_k) \text{ for all }|v| < k < |vw|\}.
\end{equation*}
For $n\in\N_0$, $\G_n^>$ is called the $n$th strictly increasing ladder line.
The sequence $(\G^>_{n})_{n \geq 0}$ contains precisely those individuals $v$ the positions of
which are strict records in the random walk $S(\varnothing),
S(v|_1), \ldots, S(v)$. Using the $\G^>_n$, we can define the
$n$th generation point process of the embedded BRW of strictly
increasing ladder heights by
\begin{equation}    \label{eq:Z^>_n}
\Z^>_n  ~:=~    \sum_{v \in \G^>_n} \delta_{S(v)}.
\end{equation}
$(\Z^>_n)_{n \geq 0}$ is a branching random walk with positive
steps only. Let $T^> := (L(v))_{v \in \G^>_1}$ and denote by $\mathbb{G}(T^>)$
the closed multiplicative subgroup generated by $T^>$. The
following result states that the point process $\Z^> := \Z^>_1$
inherits the assumptions \eqref{eq:A1}-\eqref{eq:A5} from $\Z$ and
that also the closed multiplicative groups generated by $T$ and
$T^>$ coincide. We write $\mu_{\alpha}^>$ for the measure defined
by
\begin{equation*}
\mu_{\alpha}(B) ~:=~    \E \bigg[\sum_{v \in \G^>_1} e^{-\alpha
S(v)} \delta_{S(v))}(B) \bigg], \qquad  B \subseteq \R_\geq
\  \text{ Borel.}
\end{equation*}

\begin{Prop}    \label{Prop:embedded_BRW}
Assume \eqref{eq:A1}-\eqref{eq:A3}. The following assertions
hold.
\begin{itemize}
	\item[(a)]
		$\Prob(|\mathcal{G}^>_1|< \infty) = 1$.
	\vspace{-0.15cm}
	\item[(b)]
		$\Z^>$ satisfies \eqref{eq:A1}-\eqref{eq:A3} where \eqref{eq:A3} holds with the same $\alpha$ as for $\Z$.
	\vspace{-0.15cm}
	\item[(c)]
		If $\Z$ further satisfies \eqref{eq:A4a} or \eqref{eq:A4b}, then the same holds true for $\Z^>$, respectively.
	\vspace{-0.15cm}
	\item[(d)]
		If $\Z$ satisfies \eqref{eq:A5}, then so does $\Z^>$.
	\vspace{-0.15cm}
	\item[(e)]
		Let $\mathbb{G}(\Z)$ be the minimal closed additive subgroup $G$ of $\R$
		such that $\Z(\R \setminus G)=0$ a.s.\ and define $\mathbb{G}(\Z^>)$ analogously in terms of $\Z^>$ instead of $\Z$.
		Then $\mathbb{G}(\Z^>) = \mathbb{G}(\Z) = \R$.
	\vspace{-0.75cm}
	\item[(f)]
		$\mathbb{G}(T^>) = \mathbb{G}(T)$.
\end{itemize}
\end{Prop}

\begin{Rem} \label{Rem:embedded_BRW}
Notice that assertion (f) in Proposition \ref{Prop:embedded_BRW} is the best one can get. For instance, one
cannot conclude that if $T$ has mixed signs (Case III), then so
has $T^>$. Indeed, if $T_1$ is a Bernoulli random variable with
success probability $p$ and $T_2 = -U$ for a random variable $U$
which is uniformly distributed on $(0,1)$, then all members of $\G^>_1$ have negative weights, that
is, $T^>$ has negative signs only (Case II).
\end{Rem}

\begin{proof}[Proof of Proposition \ref{Prop:embedded_BRW}]
Assertions (a), (b), (c) and (e) can be formulated in terms of the $|L(v)|$, $v \in \V$ only and,
therefore, follow from \cite[Lemma 9.1]{Alsmeyer+Biggins+Meiners:2012} and \cite[Proposition 3.2]{Alsmeyer+Meiners:2013}.

It remains to prove (d) and (f).
For the proof of (d) assume that $\Z$ satisfies \eqref{eq:A5}. By
(c), $\Z^>$ also satisfies \eqref{eq:A4a} and, in particular,
the third condition in \eqref{eq:A5}. Further, the first
condition in \eqref{eq:A5} says that $\mu_{\alpha}$, the
distribution of $S_1$, is spread-out. We have to check that then
$\mu_{\alpha}^>$ is also spread-out. It can be checked (see
\textit{e.g.}~\cite{Biggins+Kyprianou:2005}) that $\mu_{\alpha}^>$
is the distribution of $S_\sigma$ for $\sigma = \inf\{n \geq 0:
S_n > 0\}$. Hence Lemma 1 in \cite{Araman+Glynn:2006} (or Corollaries 1 and 3 of \cite{Alsmeyer:2002})
shows that the distribution of $S_{\sigma}$ is also spread-out.
That the second condition in \eqref{eq:A5} carries over is \cite[Proposition 3.2(d)]{Alsmeyer+Meiners:2013}.
The final condition of \eqref{eq:A5} is that $\E[h_3(W_1)] < \infty$ where $h_n(x) = x
(\log^+ (x))^n \log^+(\log^+ (x) )$, $n=2,3$. In view of the
validity of \eqref{eq:A4a}, Theorem 1.4 in
\cite{Alsmeyer+Iksanov:2009} yields $\E[h_2(W)] < \infty$. Now
notice that $W$ is not only the limit of the martingale $(W_n)_{n
\geq 0}$ but also of the martingale $W_n^> = \sum_{v \in \G_n^>}
|L(v)|^{\alpha}$, $n \in \N_0$, see \textit{e.g.}~Proposition 5.1
in \cite{Alsmeyer+Kuhlbusch:2010}. The converse implication of the
cited theorem then implies that $\E[h_3(W_1^>)] < \infty$.

Regarding the proof of (f) we infer from (e) that $-\log
(\mathbb{G}(|T|)) = \mathbb{G}(\Z)=\mathbb{G}(\Z^>) = -\log
(\mathbb{G}(|T^>|))$ where $|T|=(|T_j|)_{j \geq 1}$ and $|T^>| =
(|L(v)|)_{v \in \G_1}$. Thus, by \eqref{eq:A1}, $\mathbb{G}(|T^>|)
= \mathbb{G}(|T|) = \Rp$. If $\mathbb{G}(T) = \Rp$, then $T=|T|$
and $T^> = |T^>|$ a.s.\ and thus $\mathbb{G}(|T^>|) = \Rp$ as
well. It remains to show that if $\mathbb{G}(T) = \R^*$, then $\mathbb{G}(T^>) = \R^*$ as well.
To this end, it is enough to show that $\mathbb{G}(T^>) \cap
(-\infty,0) \not = \emptyset$. If $\Prob(T_j \in (-1,0)) > 0$ for
some $j \geq 1$, then $\Prob(j \in \G_1^> \text{ and } T_j < 0) >
0$. Assume now
$\Prob(T_j \in (-1,0)) = 0$ for all $j \geq 1$. Since
$\mathbb{G}(T) = \R^*$ there is an $x \geq 1$ such that $-x \in
\mathrm{supp}(T_j)$ for some $j \geq 1$ where $\mathrm{supp}(X)$
denotes the support (of the law) of a random variable $X$.
By \eqref{eq:A3}, we have $m(\alpha) = 1 < m(\beta)$ for all $\beta
\in [0,\alpha)$. This implies that for some $k \geq 1$,
$\Prob(|T_k| \in (0,1)) >0$ and, moreover, $\Prob(T_k \in (0,1)) >
0$ since $\Prob(T_k \in (-1,0))=0$. Thus, for some $y \in (0,1)$,
we have $y \in \mathrm{supp}(T_k)$. Let $m$ be the minimal
positive integer such that $xy^m<1$. Then $-xy^m \in
\mathbb{G}(T^>)$.
\end{proof}

\subsection{Endogenous fixed points}    \label{subsec:endogeny}

Important for the problems considered here is the concept of {\em endogeny},
which has been introduced in \cite[Definition 7]{AB2005}.
For the purposes of this paper, it is enough to study endogeny in dimension $d=1$.

Suppose that $W^{(v)}$, $v \in \V$ is a family of random variables such that the $W^{(v)}$, $|v|=n$ are i.i.d.\ and independent of $\A_n$
for each $n \in \N_0$.
Further suppose that
\begin{equation}    \label{eq:RTP}
W^{(v)} ~=~ \sum_{j \geq 1} T_j(v) W^{(vj)} \quad   \text{a.s.}
\end{equation}
for all $v \in \V$. Then the family $(W^{(v)})_{v \in \V}$ is called a \emph{recursive tree process},
the family $(T(v))_{v \in \V}$ \emph{innovations process} of the recursive tree process.
The recursive tree process $(W^{(v)})_{v \in \V}$ is called nonnegative if the $W^{(v)}$, $v \in \V$ are all nonnegative,
it is called \emph{invariant} if all its marginal distributions are identical.
There is a one-to-one correspondence between the solutions to \eqref{eq:generalized stable} (in dimension $d=1$)
and recursive tree processes $(W^{(v)})_{v \in \V}$ as above, see Lemma 6 in \cite{AB2005}.
An invariant recursive tree process $(W^{(v)})_{v \in \V}$ is \emph{endogenous} if $W^{(\varnothing)}$ is measurable w.r.t.~the innovations process $(T(v))_{v \in \V}$.

\begin{Def}[\textit{cf.}\ Definition 8.2 in \cite{Alsmeyer+Biggins+Meiners:2012}]
\begin{itemize}
    \item
        A distribution is called \emph{endogenous} (w.r.t.\ the sequence $(T_j)_{j \geq 1}$)
        if it is the marginal distribution of an endogenous recursive tree process with innovations process $(T(v))_{v \in \V}$.
    \vspace{-0.2cm}
    \item
        A random variable $W$ is called \emph{endogenous fixed point (w.r.t.~$(T_j)_{j \geq 1}$)}
        if there exists an endogenous recursive tree process with innovations process $(T(v))_{v \in \V}$ such that $W = W^{(\varnothing)}$ a.s.
\end{itemize}
\end{Def}

A random variable $W$ is called non-null when $\Prob(W \not= 0) > 0$.
$W=0$ is an endogenous fixed point.
Of course, the main interest is in non-null endogenous fixed points $W$.
An endogenous recursive tree process $(W^{(v)})_{v \in \V}$ will be called non-null when $W^{(\varnothing)}$ is non-null.

Endogenous fixed points have been introduced in a slightly different way in \cite[Definition 4.6]{Alsmeyer+Meiners:2013}.
Using the shift-operator notation, in \cite[Definition 4.6]{Alsmeyer+Meiners:2013}, a random variable $W$ (or its distribution)
is called endogenous (w.r.t.~to $(T(v))_{v \in \V}$) if $W$ is measurable w.r.t.~$(L(v))_{v \in \V}$ and if
\begin{equation}    \label{eq:purely_tree-based}
W   ~=~ \sum_{|v|=n} L(v) [W]_v \quad   \text{a.s.}
\end{equation}
for all $n\in\N_0$.
It is immediate that $([W]_v)_{v \in \V}$ then defines an endogenous recursive tree process.
Therefore, Definition 4.6 in \cite{Alsmeyer+Meiners:2013} is (seemingly) stronger than the original definition of endogeny.
The next lemma shows that the two definitions are equivalent.

\begin{Lemma}   \label{Lem:endogeny}
Let \eqref{eq:A1}-\eqref{eq:A4} hold and let $(W^{(v)})_{v \in \V}$ be an endogenous recursive tree process with innovations process $(T(v))_{v \in \V}$.
Then $W^{(v)} = [W^{(\varnothing)}]_v$ a.s.~for all $v \in \V$ and \eqref{eq:purely_tree-based} holds.
\end{Lemma}
The arguments in the following proof are basically contained in \cite[Proposition 6.4]{Alsmeyer+Biggins+Meiners:2012}.
\begin{proof}
For $u \in \V$ and $n \in \N_0$, \eqref{eq:RTP} implies $W^{(u)} = \sum_{|v|=n} [L(v)]_u W^{(uv)}$ a.s.,
which together with the martingale convergence theorem yields
\begin{eqnarray}
\exp(\imag t W^{(u)})
& = &
\lim_{n \to \infty} \E[\exp(\imag t W^{(u)}) | \F_{|u|+n}]
~=~
\lim_{n \to \infty} \E\bigg[\exp\bigg(\imag t \sum_{|v|=n} [L(v)]_u W^{(uv)}\bigg) | \F_{|u|+n}\bigg]	\notag	\\
& = &
\lim_{n \to \infty} \prod_{|v|=n} \phi([L(v)]_ut)   ~=~ [\Phi(t)]_u \quad   \text{a.s.}	    \label{eq:Eexp(itW_u)}
\end{eqnarray}
where $\phi$ denotes the Fourier transform of $W^{(\varnothing)}$ and $\Phi(t)$ denotes the a.s.\ limit of the multiplicative martingale
$\prod_{|v|=n} \phi(L(v)t)$ as $n \to \infty$.
The left-hand side in \eqref{eq:Eexp(itW_u)} is continuous in $t$.
The right-hand side is continuous in $t$ a.s.\ by Proposition \ref{Prop:Disintegration}.
Therefore, \eqref{eq:Eexp(itW_u)} holds simultaneously for all $t \in \R$ a.s.
In particular, $\exp(\imag t W^{(\varnothing)}) = \Phi(t)$ for all $t \in \R$ a.s.
Thus, $\exp(\imag t W^{(u)}) = [\Phi(t)]_u = \exp(\imag t [W^{(\varnothing)}]_u)$ for all $t \in \R$ a.s.
This implies $W^{(u)} = [W^{(\varnothing)}]_u$ a.s.\ by the uniqueness theorem for Fourier transforms.
\end{proof}

Justified by Lemma \ref{Lem:endogeny} we shall henceforth
use \eqref{eq:purely_tree-based} as the definition of endogeny.
Theorem 6.2 in \cite{Alsmeyer+Biggins+Meiners:2012} gives (almost)
complete information about nonnegative endogenous fixed points in
the case when the $T_j$, $j \geq 1$ are nonnegative. This result
has been generalized in \cite{Alsmeyer+Meiners:2013}, Theorems 4.12
and 4.13. Adapted to the present situation, all these findings are summarized in the following proposition.

\begin{Prop}    \label{Prop:endogeny}
Assume that \eqref{eq:A1}-\eqref{eq:A4} hold true. Then
\begin{itemize}
	\item[(a)]
		there is a nonnegative non-null endogenous fixed point $W$ w.r.t.~$(|T_j|^{\alpha})_{j \geq 1}$
		given by \eqref{eq:W_definition}.
		Any other nonnegative endogenous fixed point w.r.t.~$(|T_j|^{\alpha})_{j \geq 1}$ is of the form $cW$ for some $c \geq 0$.
	\item[(b)]
		$W$ defined in \eqref{eq:W_definition} further satisfies
		\begin{align}
		W	~& =~
		\lim_{t \to \infty} D\left(e^{-\alpha t}\right) \sum_{v \in \mathcal{T}_t} e^{-\alpha S(v)}
		~=~ \lim_{t \to \infty} \sum_{v \in \mathcal{T}_t} e^{-\alpha S(v)} D\left(e^{-\alpha S(v)}\right)  \label{eq:twolimits}    \\
		& =~ \lim_{t \to \infty} \sum_{v \in \mathcal{T}_t} e^{-\alpha S(v)} \int_{\{|x|<e^{\alpha S(v)}\}} x \, \Prob(W \in \dx)
		\quad   \text{a.s.} \label{eq:third limit}
		\end{align}
	\item[(c)]
		There are no non-null endogenous fixed points w.r.t.~$(|T_j|^{\beta})_{j \geq 1}$ for $\beta \not = \alpha$.
	\item[(d)]
		If, additionally, $\E[\sum_{j \geq 1} |T_j|^{\alpha} (\log^- (|T_j|))^2] < \infty$,
		then $W$ given by \eqref{eq:W_definition} is the unique endogenous fixed point w.r.t.~$(|T_j|^{\alpha})_{j \geq 1}$ up to scaling.
		In particular, every endogenous fixed point w.r.t.~$(|T_j|^{\alpha})_{j \geq 1}$ is nonnegative a.s.\ or nonpositive a.s.
\end{itemize}
\end{Prop}
\begin{proof}
(a) is Theorem 6.2(a) in \cite{Alsmeyer+Biggins+Meiners:2012} (and partially already stated in Proposition \ref{Prop:W}).
\eqref{eq:twolimits} is (11.8) in \cite{Alsmeyer+Biggins+Meiners:2012}, \eqref{eq:third limit} is (4.39) in \cite{Alsmeyer+Meiners:2013}.
(c) is Theorem 6.2(b) in \cite{Alsmeyer+Biggins+Meiners:2012} in the case of nonnegative (or nonpositive) recursive tree processes
and Theorem 4.12 in \cite{Alsmeyer+Meiners:2013} in the general case. (d) is Theorem 4.13 in \cite{Alsmeyer+Meiners:2013}.
\end{proof}

In the case of weights with mixed signs there may be
endogenous fixed points other than those described in Proposition
\ref{Prop:endogeny}. Theorem \ref{Thm:endogeny} given next states
that under \eqref{eq:A1}-\eqref{eq:A4} these fixed points are
always a deterministic constant times $Z$, the limit of $Z_n ~=~
\sum_{|v|=n} L(v)$, $n\in\N_0$.

\begin{Thm} \label{Thm:endogeny}
Suppose \eqref{eq:A1}-\eqref{eq:A4}.
Then the following assertions hold.
\begin{itemize}
    \item[(a)]
        If $\alpha < 1$, there are no non-null
        endogenous fixed points w.r.t.~$T$.
    \item[(b)]
        Let $\alpha = 1$
        and assume that $\E[\sum_{j \geq 1} |T_j| (\log^-(|T_j|))^2] < \infty$ holds in Cases I and II
        and \eqref{eq:A5} holds in Case III.
        Then the endogenous fixed points are precisely of the form $cW$ a.s., $c \in \R$ in Case I,
        while in Cases II and III, there are no non-null
        endogenous fixed points w.r.t.~$T$.
    \item[(c)]
        If $\alpha > 1$, the following assertions are equivalent.
        \begin{itemize}
        \item[(i)]
        There is a non-null endogenous fixed point w.r.t.~$T$.
        \vspace{-0.1cm}
        \item[(ii)]
        $Z_n$ converges a.s. and $\Prob(\lim_{n\to\infty} Z_n=0)<1$.
        \vspace{-0.1cm}
        \item[(iii)]
        $\E [Z_1] = 1$ and $(Z_n)_{n \geq 0}$ converges in $\mathcal{L}^{\beta}$ for some/all $1 < \beta < \alpha$.
        \end{itemize}
        If either of the conditions (i)-(iii) is satisfied, then $(Z_n)_{n \geq 0}$ is a uniformly integrable martingale.
        In particular, $Z_n$
        converges a.s.~and in mean to some random variable $Z$ with $\E[Z]=1$
        which is an endogenous fixed point w.r.t.~$T$.
        Any other endogenous fixed point is of the form $cZ$ for some $c \in \R$.
\end{itemize}
\end{Thm}

The proof of this result is postponed until Section \ref{subsec:endogeny proofs}.

We finish the section on endogenous fixed points with the proof of Proposition \ref{Prop:Jelenkovic+Olvera:2012b}
which establishes the existence of $\W^*$. If well-defined, the latter random variable can be viewed as an endogenous
inhomogeneous fixed point.

\begin{proof}[Proof of Proposition \ref{Prop:Jelenkovic+Olvera:2012b}:]
By using the Cram\'{e}r-Wold device we can and do assume that
$d=1$. We shall write $W_n^*$ for $\W_n^*$.

\noindent (i)
It suffices to show that the infinite series $\sum_{v \in \V}|L(v)| |C(v)|$ converges
a.s.~which is, of course, the case if the sum has a finite moment of order $\beta$.
The latter follows easily (see, for instance, \cite[Lemma 4.1]{Jelenkovic+Olvera:2012b} or \cite[Proposition 5.4]{Alsmeyer+Meiners:2013}).

\noindent (ii) The assumption entails
$\E[|C|^\beta]<\infty$ and thereupon $\E[C]\in (-\infty,\infty)$.
If $T_j \geq 0$ a.s.\ for all $j \geq 1$, then sufficiency follows from
\cite[Proposition 5.4]{Alsmeyer+Meiners:2013}. Thus, assume that
$\E [C] = 0$. Then $(W_n^*)_{n \geq 0}$ is a martingale w.r.t.\
$(\A_n)_{n \geq 0}$. Since $W_n^*$ has distribution
$\Smooth^n(\delta_{\mathbf{0}})$, $n \geq 0$, this martingale is
$\mathcal{L}^{\beta}$-bounded and, hence, a.s.\ convergent.

\noindent (iii)
This is the first part of Theorem 1.1 in \cite{Buraczewski+Kolesko:2014}.
\end{proof}

\subsection{A multitype branching process and homogeneous stopping lines}   \label{subsec:HSL}

In this section we assume that \eqref{eq:A1}--\eqref{eq:A4} and
\eqref{eq:A6} hold and that we are in the case of weights with
mixed signs (Case III). Because of the latter assumption,
when defining the branching random walk $(\Z_n)_{n \geq 0}$ from
$(L(v))_{v \in \V}$, information is partially lost since
each position $S(v)$ is defined in terms of the absolute value
$|L(v)|$ of the corresponding weight $L(v)$, $v \in \V$. This loss
of information can be compensated by keeping track of the sign of
$L(v)$. Define
\begin{equation*}
\tau(v) ~:=~    \begin{cases}
                1   &   \text{if }  L(v) > 0,   \\
                -1  &   \text{if }  L(v) < 0    \\
                \end{cases}
\end{equation*}
for $v \in \G$. For the sake of completeness, let $\tau(v) = 0$ when $L(v) = 0$.
The positions $S(v),$ $v \in \V$ together with the signs $\tau(v),$ $v \in \V$ define a multitype general branching process with type space $\{1,-1\}$.

Define $M(\gamma) := (\mu_{\gamma}^{k,\ell}(\R))_{k,\ell =1,-1}$ where
\begin{equation*}
\mu_{\gamma}^{k,\ell}(\cdot)
~:=~    \E\bigg[\sum_{j \geq 1: \sign(T_j) = k \ell} |T_j|^{\gamma} \delta_{S(j)}(\cdot)\bigg]
\end{equation*}
Then $p = \E\big[\sum_{j \geq 1} \! T_j^{\alpha}
\1_{\{T_j
> 0\}}\big] = \mu_{\alpha}^{1,1}(\R) = \mu_{\alpha}^{-1,-1}(\R)$
and $q = 1-p = \E \big[\sum_{j \geq 1} \! |T_j|^{\alpha} \1_{\{T_j
< 0\}}\big] = \mu_{\alpha}^{1,-1}(\R) = \mu_{\alpha}^{-1,1}(\R)$.
Therefore,
\begin{equation*}
M(\alpha) = \begin{pmatrix} p & q   \\  q & p       \end{pmatrix}
\end{equation*}
where $0<p,q<1$ since we are in Case III.
Next, we establish that the general branching process possesses
the following properties.
\begin{itemize}
    \item[(i)]
        For all $h > 0$ and all $h_1,h_{-1} \in [0,h)$ either
        $\mu_{\alpha}^{1,1}(\R \setminus h\integers) > 0$, $\mu_{\alpha}^{1,-1}(\R \setminus (h_{-1}-h_1 + h\integers)) > 0$
        or $\mu_{\alpha}^{-1,1}(\R \setminus (h_{1}-h_{-1} + h\integers)) > 0$.
        Further, $M(\alpha)$ is irreducible.
    \vspace{-0.2cm}
    \item[(ii)]
        Either $M(0)$ has finite entries only and Perron-Frobenius eigenvalue $\rho > 1$ or $M(0)$ has an infinite entry.
    \vspace{-0.2cm}
    \item[(iii)]
        $M(\alpha)$ has eigenvalues $1$ and $2p\!-\!1$ (with right eigenvectors $(1,1)^{\!\textsf{T}}$ and $(1,-1)^{\!\textsf{T}}$, respectively).
        $1$ is the Perron-Frobenius eigenvalue of $M(\alpha)$.
    \vspace{-0.2cm}
    \item[(iv)]
        The first moments of $\mu_{\alpha}^{k,\ell}$ are finite and positive for all $k,\ell \in \{1,-1\}$.
\end{itemize}
(i)--(iv) correspond to assumptions (A1)--(A4) in
\cite{Iksanov+Meiners:2014} and will justify the applications of
the limit theorems of the cited paper.
\begin{proof}[Proof of the validity of (i)--(iv)]
(i) $M(\alpha)$ is irreducible because all its entries are positive.
Now assume for a contradiction that for some $h > 0$
and some $h_1,h_{-1} \in [0,h)$, $\mu_{\alpha}^{1,1}(\R \setminus
h\integers) = \mu_{\alpha}^{1,-1}(\R \setminus (h_{-1}-h_1 +
h\integers)) = \mu_{\alpha}^{-1,1}(\R \setminus (h_{1}-h_{-1} +
h\integers)) = 0$. Since $\mu_{\alpha}^{1,-1} =
\mu_{\alpha}^{-1,1}$ is nonzero, this implies $(h_{-1}-h_1 +
h\integers) \cap (h_{1}-h_{-1} + h\integers) \not = \emptyset$.
Hence, there are $m,n \in \integers$ such that $h_{-1}-h_1 + hm =
h_{1}-h_{-1} + hn$, equivalently, $2(h_{-1}-h_1) = h(n-m)$. Thus,
$h_{-1}-h_1$ and $h_{-1}-h_1$ belong to the lattice $\frac{h}{2} \integers$.
This contradicts \eqref{eq:A1}. While (iii) can be verified by elementary
calculations, (ii) is an immediate consequence of (iii) for $M(0)$
has strictly larger entries than $M(\alpha)$ which has
Perron-Frobenius eigenvalue $1$. (iv) follows from \eqref{eq:A4}.
\end{proof}

Recall that $\mathcal{T}_t = \{v \in \G: S(v) > t \text{ but } S(v|_k) \leq t \text{ for all } 0 \leq k < |v|\}$, $t \geq 0$.

\begin{Prop}    \label{Prop:ratios}
Assume that \eqref{eq:A1}-\eqref{eq:A3} and \eqref{eq:A6} hold and that we are in Case III, \textit{i.e.}, $0 < p,q < 1$.
Further, let $h:[0,\infty) \to (0,\infty)$ be a c\`adl\`ag function
such that $h(t) \leq C t^{\gamma}$ for all sufficiently large $t$ and some $C>0$, $\gamma \geq 0$.
\begin{itemize}
    \item[(a)]
        Suppose that \eqref{eq:A4a} holds and that $\E \big[\sum_{j=1}^N |T_j|^{\alpha} S(j)^{1+\gamma}\big] < \infty$.
        Then, for $\beta = \alpha$, $j=1,-1$, any $\varepsilon>0$ and all sufficiently large $c$, the following convergence in probability holds as $t \to \infty$ on
        the survival set $S$
        \begin{align}       \label{eq:ratio_convergence}
        \frac{\sum_{v \in \mathcal{T}_t: S(v) \leq t+c,\tau(v)=j} e^{-\beta (S(v)-t)} h(S(v)-t)}
        {\sum_{v \in \mathcal{T}_t} e^{-\alpha (S(v)-t)} h(S(v)-t)}
        ~\to~   \frac{1}{2} - \varepsilon(c)
         ~\geq~ \frac{1}{2} - \varepsilon.
        \end{align}
    \item[(b)]
        Suppose that \eqref{eq:A4b} holds.
        Then the convergence in \eqref{eq:ratio_convergence} holds in the a.s.\ sense for all $\beta \geq
        \theta$ (with $\theta$ defined in \eqref{eq:A4b}) and sufficiently large $c$ that may depend on $\beta$.
\end{itemize}
If one chooses $c = \infty$ (\textit{i.e.}, if one drops the
condition $S(v) \leq t+c$), the result holds with $\varepsilon(\infty) = 0$.
\end{Prop}
\begin{proof}
(a) and (b) can be deduced from general results on convergence of
multi-type branching processes, namely, Theorems 2.1 and 2.4 in
\cite{Iksanov+Meiners:2014}. The basic assumptions (A1)--(A4)
of the cited article are fulfilled for these coincide with (i)-(iv) here.
Assumption (A5) and Condition 2.2 in \cite{Iksanov+Meiners:2014}
correspond to \eqref{eq:A4a} and \eqref{eq:A4b} here,
respectively. Further, for fixed $j \in \{1,-1\}$, the numerator
in \eqref{eq:ratio_convergence} is $Z^{\phi}(t) = \sum_{v \in \G}
[\phi]_v(t-S(v))$ for
\begin{equation*}
\phi(t) ~=~ \sum_{k=1}^{N(v)} e^{-\beta(S(k)-t)} h(S(k)\!-\!t) \1_{\{t < S(k) \leq t+c,\,\tau(k)= \tau(\varnothing)j\}},
\end{equation*}
while the denominator is of the form $Z^{\psi}(t)$ with
\begin{equation*}
\psi(t) ~=~ \sum_{k=1}^{N(v)} e^{-\alpha(S(k)-t)} h(S(k)\!-\!t) \1_{\{\tau(k)= \tau(\varnothing)j\}}.
\end{equation*}
The verification of the remaining conditions of Theorems 2.1 and 2.4 in \cite{Iksanov+Meiners:2014}
is routine and can be carried out as in the proof of Proposition 9.3 in \cite{Alsmeyer+Biggins+Meiners:2012}.

The last statement follows from the same proof if one replaces $c$ in the definition of $\phi$ by $+\infty$.
\end{proof}

The final result in this section is on the asymptotic behaviour of $\sum_{v \in \mathcal{T}_t} L(v)$ in Case III when $\alpha = 1$:

\begin{Lemma}   \label{Lem:sumL(v)->0}
Assume that \eqref{eq:A1}-\eqref{eq:A6} hold, that $\alpha = 1$ and that $0<p,q<1$.
Then
\begin{equation}    \label{eq:sumL(v)->0}
t \sum_{v \in \mathcal{T}_t} L(v) ~\to~ 0   \quad   \text{as } t \to \infty \text{ in probability.}
\end{equation}
\end{Lemma}
\begin{proof}
\eqref{eq:sumL(v)->0} follows from Theorem 6.1 in \cite{Iksanov+Meiners:2014} for $\delta=1$.
\end{proof}

\section{Proofs of the main results}    \label{sec:proofs of main results}

\subsection{Proof of the direct part of Theorem \ref{Thm:SF}}   \label{subsec:direct}

\begin{proof}[Proof of Theorem \ref{Thm:SF} (direct part)]
We only give (a sketch of) the proof in the case $\alpha=1$
and $\mathbb{G}(T)=\Rp$. The other cases can be treated
analogously. Let $\phi$ be as in \eqref{eq:phi alpha=1 R>0}, \textit{i.e.},
\begin{equation*}
\phi(\bt)
=   \E \!\bigg[\!\exp\!\left(\!\imag \langle\W^*\!+W\ba, \!\bt \rangle - W \!\! \int \! |\langle \bt,\bs \rangle| \, \sigma(\ds)
- \imag W \frac{2}{\pi} \!\int \! \langle \bt, \bs \rangle \log (|\langle \bt, \bs \rangle|) \, \sigma(\dbs) \!\right)\!\!\bigg]
\end{equation*}
for some $\ba \in \R^d$ and a finite measure $\sigma$ on $\Sd$
with $\int s_k \, \sigma(\dbs)=0$ for $k=1,\ldots,d$. Using
$\W^* =  \sum_{j \geq 1} T_j [\W^*]_j + \bC$ a.s.~and $W = \sum_{j \geq 1} T_j [W]_j$ a.s.,
see \eqref{eq:W* fixed point} and \eqref{eq:purely_tree-based}, we obtain
\begin{align}
\imag \langle\W^*\!&+W\ba, \!\bt \rangle - W \!\! \int \! |\langle \bt,\bs \rangle| \, \sigma(\ds)
- \imag W \frac{2}{\pi} \!\int \! \langle \bt, \bs \rangle \log (|\langle \bt, \bs \rangle|) \, \sigma(\dbs)        \notag  \\
&=~ \imag \langle \bC, \!\bt \rangle + \imag \sum_{j \geq 1} \langle [\W^*]_j+[W]_j \ba, T_j \bt \rangle    \notag  \\
&\hphantom{=}~- \sum_{j \geq 1} [W]_j \!\! \int \! |\langle T_j \bt,\bs \rangle| \, \sigma(\ds)
- \imag \sum_{j \geq 1} [W]_j \frac{2}{\pi} \!\int \! \langle T_j \bt, \bs \rangle \log (|\langle \bt, \bs \rangle|) \, \sigma(\dbs).       \label{eq:xlogx integral}
\end{align}
Further, since $\int s_k \, \sigma(\dbs)=0$ for $k=1,\ldots,d$, we have
\begin{equation*}
\int \! \langle T_j \bt, \bs \rangle \log (|\langle \bt, \bs \rangle|) \, \sigma(\dbs)
~=~ \int \! \langle T_j \bt, \bs \rangle \log(|\langle T_j \bt, \bs \rangle|) \, \sigma(\dbs)
\end{equation*}
for all $j \geq 1$ with $T_j > 0$. Substituting this in \eqref{eq:xlogx integral}, passing to exponential functions, taking expectations on both sides
and then using that the couples $([\W^*]_j,[W]_j)$, $j \geq 1$ are i.i.d.~copies of
$(\W^*,W)$ independent of $(\bC,T)$ one can check that $\phi$ satisfies \eqref{eq:FE generalized stable inhom}.
\end{proof}

\subsection{Solving the functional equation in $\mathcal{M}$}       \label{subsec:one-dimensional_FE}

\begin{Thm} \label{Thm:d=1,2sided_FE_disintegrated}
Assume that \eqref{eq:A1}--\eqref{eq:A4} hold true and let $d=1$.
Let $f \in \SM$ and denote the limit of the corresponding multiplicative martingale by $M$.
Then there are constants $c_{1}, c_{-1} \geq 0$ such that
\begin{equation}    \label{eq:d=1_2sided_FE_disintegrated}
M(t)    ~=~ \begin{cases}
                \exp(-Wc_{1} t^{\alpha})    &   \text{for } t \geq 0,       \\
                \exp(-Wc_{-1} |t|^{\alpha}) &   \text{for } t \leq 0
                \end{cases}
\quad   \text{a.s.}
\end{equation}
Furthermore, if $\mathbb{G}(T) = \R^*$, then $c_{1} = c_{-1}$.
\end{Thm}

We first prove Theorem \ref{Thm:d=1,2sided_FE_disintegrated}
in Cases I and II (see \eqref{eq:CaseI-III}). Case III needs some
preparatory work and will be settled at the end of this section.

\begin{proof}[Proof of Theorem \ref{Thm:d=1,2sided_FE_disintegrated}]
\underline{Case I:} The statement is a consequence of Theorem 8.3 in \cite{Alsmeyer+Biggins+Meiners:2012}.
\newline
\underline{Case II:}
For $f \in \SM$,
iteration of \eqref{eq:FE} in terms of the weighted branching model gives
\begin{equation}    \label{eq:FE_2iterated}
f(t)    ~=~ \E \bigg[\prod_{|v|=2} f(L(v)t)\bigg],  \qquad  t \in \R.
\end{equation}
By Lemma \ref{Lem:WBP_even}, $(L(v))_{|v|=2}$ satisfies
\eqref{eq:A1}--\eqref{eq:A4}. Further, the endogenous fixed point
$W$ is (by uniqueness) the endogenous fixed point for $(|L(v)|^{\alpha})_{|v|=2}$.
Since in Case II all $T_j$, $j\in\N$ are a.s.~nonpositive, all $L(v)$, $|v|=2$ are a.s.~nonnegative.
This allows us to invoke the conclusion of the already settled Case I to infer that \eqref{eq:d=1_2sided_FE_disintegrated} holds
with constants $c_1, c_{-1} \geq 0$.
Using \eqref{eq:M's equation} for $n=1$ and $t > 0$ we get
\begin{eqnarray*}
\exp(-Wc_{1} t^{\alpha})
& = &
M(t)
~=~ \prod_{j \geq 1} [M]_j(T_j t) ~=~ \prod_{j \geq 1} \exp(-[W]_j c_{-1} |T_j t|^{\alpha})     \\
& = &
\exp\bigg(- c_{-1} \sum_{j \geq 1} |T_j|^{\alpha} [W]_j t^{\alpha} \bigg)
~=~ \exp(-W c_{-1} t^{\alpha})
\quad   \text{a.s.}
\end{eqnarray*}
In particular, $c_1 = c_{-1}$.
\end{proof}

Assuming that Case III prevails, \textit{i.e.}, $0<p,q<1$,
we prove four lemmas. While Lemmas \ref{Lem:f(t)<1} and
\ref{Lem:BK's_trick} are principal and will be used in the proof
of (the remaining part of) Theorem \ref{Thm:d=1,2sided_FE_disintegrated},
Lemmas \ref{Lem:ratios_bounded_away_from_infty} and \ref{Lem:Selection}
are auxiliary and will be used in the proof of Lemma \ref{Lem:BK's_trick}.

\begin{Lemma}   \label{Lem:f(t)<1}
Let $f \in \SM$. If $f(t) = 1$ for some $t \not = 0$, then $f(u) = 1$ for all $u \in \R$.
\end{Lemma}
\begin{proof}
Let $t \not = 0$ with $f(t) = 1$, w.l.o.g.\ $t>0$.
We have
\begin{equation*}
1 ~=~ f(t) ~=~ \E \bigg[\prod_{j \geq 1} f(T_j t)\bigg].
\end{equation*}
Since all factors on the right-hand side of this equation are
bounded from above by $1$, they must all equal $1$ a.s. In
particular, since $\Prob(T_j < 0) > 0$ for some $j$
(see Proposition \ref{Prop:embedded_BRW}(e)), there is some $t' < 0$
with $f(t') = 1$. Let $s := \min\{t,|t'|\}$. Then, since $f$ is
nondecreasing on $(-\infty,0]$ and nonincreasing on $[0,\infty)$, we have $f(u) = 1$ for all
$|u| \leq s$. Now pick an arbitrary $u \in \R$, $|u| > s$ and let
$\tau := \inf\{n \geq 1: \sup_{|v|=n} |L(v)u| \leq s\}$. Then
$\tau < \infty$ a.s.\ by Lemma \ref{Lem:sup->0}. Since
$(\prod_{|v|=n} f(L(v)u))_{n \geq 0}$ is a bounded martingale, the
optional stopping theorem gives
\begin{equation*}
f(u) ~=~ \E \bigg[\prod_{|v|=\tau} f(L(v) u)\bigg]  ~=~ 1.
\end{equation*}
This completes the proof since $u$ was arbitrary with $|u|>s$.
\end{proof}

Let $D_{\alpha}(t) := \frac{1-f(t)}{|t|^{\alpha}}$ for $t \not = 0$ and
\begin{equation*}
K_{\mathrm{l}} ~:=~ \liminf_{t \to \infty} \frac{D_{\alpha}(e^{-t}) \vee D_{\alpha}(-e^{-t})}{D(e^{-\alpha t})},
\quad   \text{and}  \quad
K_{\mathrm{u}}^{\pm} ~:=~ \limsup_{t \to \infty} \frac{D_{\alpha}(\pm e^{-t})}{D(e^{-\alpha t})}.
\end{equation*}
Further, put $K_{\mathrm{u}} := K_{\mathrm{u}}^{+} \vee K_{\mathrm{u}}^{-}$.

\begin{Lemma}   \label{Lem:ratios_bounded_away_from_infty}
Assume that \eqref{eq:A1}-\eqref{eq:A4} and \eqref{eq:A6}
hold, and let $f \in \SM$ with $f(t) < 1$ for some (hence all) $t
\not = 0$. Then
\begin{equation*}
0 ~<~ K_\mathrm{l} ~\leq~ K_\mathrm{u} ~<~\infty.
\end{equation*}
\end{Lemma}
\begin{proof}[Proof of Lemma \ref{Lem:ratios_bounded_away_from_infty}]
The proof of this lemma is an extension of the proof of Lemma 11.5 in \cite{Alsmeyer+Biggins+Meiners:2012}.
Though the basic idea is identical, modifications are needed at several places.

Since $D$ is nonincreasing,
\begin{align*}
\sum_{v \in \mathcal{T}_t:\tau(v)=j}
e^{-\alpha S(v)}
D(e^{-\alpha S(v)}) \1_{\{S(v) \leq t+c\}}
~\geq~ \frac{\sum_{v \in \mathcal{T}_t: \tau(v)=j}
e^{-\alpha S(v)} \1_{\{S(v)-t \leq  c\}}}{\sum_{v \in
\mathcal{T}_t} e^{-\alpha S(v)}} D(e^{-\alpha t})
\sum_{v \in \mathcal{T}_t} e^{-\alpha S(v)}
\end{align*}
%\begin{align*}
%& \frac{\sum_{v \in \mathcal{T}_t:\tau(v)=j} e^{-\alpha S(v)} D(e^{-\alpha S(v)}) \1_{\{S(v) \leq t+c\}}}{
%\sum_{v \in \mathcal{T}_t} e^{-\alpha S(v)} D(e^{-\alpha S(v)})}    \\
%&\qquad ~\geq~ \frac{\sum_{v \in \mathcal{T}_t,\tau(v)=j} e^{-\alpha S(v)} \1_{\{S(v)-t \leq  c\}}}{\sum_{v \in \mathcal{T}_t} e^{-\alpha S(v)}}
%\frac{D\left(e^{-\alpha t}\right) \sum_{v \in \mathcal{T}_t} e^{-\alpha S(v)}}{\sum_{v \in \mathcal{T}_t} e^{-\alpha S(v)} D(e^{-\alpha S(v)})}.
%\end{align*}
for $j=1,-1$. By Proposition \ref{Prop:ratios} with $h=1$,
the first ratio tends to something $\geq \frac{1}{2}-\varepsilon$ in probability on $S$ for given $\varepsilon > 0$ when $c$ is chosen sufficiently large.
The second converges to $W$ a.s.~on $S$ by \eqref{eq:twolimits}.
Further,
\begin{align*}
\sum_{v \in \mathcal{T}_t} & e^{-\alpha S(v)} D_{\alpha}(L(v))
~\geq~
\sum_{v \in \mathcal{T}_t: \tau(v)=j} e^{-\alpha S(v)} D_{\alpha}(je^{-S(v)}) \1_{\{S(v) \leq t+c\}}        \\
& \geq~
e^{-\alpha c} D_{\alpha}(je^{-(t+c)}) \sum_{v \in \mathcal{T}_t: \tau(v) = j} e^{-\alpha S(v)} \1_{\{S(v) \leq t+c\}} \\
& \geq~
e^{-\alpha c} \frac{D_{\alpha}(je^{-(t+c)})}{D(e^{-\alpha(t+c)})} \sum_{v \in
\mathcal{T}_t: \tau(v) = j} e^{-\alpha S(v)} D(e^{-\alpha S(v)})
\1_{\{S(v) \leq t+c\}}.
\end{align*}
For $j=1$, passing to the limit $t \to \infty$ along an appropriate subsequence gives
\begin{equation*}
-\log (M(1))    ~\geq~
e^{-\alpha c} K_{\mathrm{u}}^{+}\Big(\frac{1}{2}-\varepsilon\Big) W    \quad
\text{a.s.}
\end{equation*}
where the convergence of the left-hand side follows from taking logarithms in \eqref{eq:M' as limit along ladder lines},
\textit{cf.}~\cite[Lemma 8.7(c)]{Alsmeyer+Biggins+Meiners:2012}.
Now one can argue literally as in the proof of Lemma 11.5 in \cite{Alsmeyer+Biggins+Meiners:2012} to conclude that $K_{\mathrm{u}}^{+} < \infty$.
$K_{\mathrm{u}}^{-} < \infty$ follows by choosing $j=-1$ in the argument above.

In order to conclude that $K_{\mathrm{l}}>0$, we derive an upper bound for $-\log (M(1))$
\begin{eqnarray*}
\sum_{v \in \mathcal{T}_t} e^{-\alpha S(v)} D_{\alpha}(L(v))
& \leq &
e^{\alpha c} \big(D_{\alpha}(e^{-t}) \vee D_{\alpha}(-e^{-t})\big) \sum_{v \in \mathcal{T}_t} e^{-\alpha S(v)} \1_{\{S(v) \leq t+c\}} \\
& &
+ \sum_{v \in \mathcal{T}_t} e^{-\alpha S(v)} D_{\alpha}(L(v)) \1_{\{S(v) > t+c\}}  \\
& \leq &
e^{\alpha c} \frac{D_{\alpha}(e^{-t}) \vee D_{\alpha}(-e^{-t})}{D(e^{-\alpha t})}
\sum_{v \in \mathcal{T}_t} e^{-\alpha S(v)} D(e^{-\alpha S(v)}) \1_{\{S(v) \leq t+c\}}  \\
& & + \sum_{v \in \mathcal{T}_t} e^{-\alpha S(v)} D_{\alpha}(L(v)) \1_{\{S(v) > t+c\}}.
\end{eqnarray*}
Now letting $t \to \infty$ along an appropriate subsequence and using Proposition \ref{Prop:ratios}, we obtain that
\begin{equation*}
-\log (M(1))    ~\leq~  e^{\alpha c} K_{\mathrm{l}} W + K_{\mathrm{u}} \varepsilon W.
\end{equation*}
Hence, $K_{\mathrm{l}}=0$ would imply $M(1)=1$ a.s., in
particular, $f(1) = \E[M(1)] = 1$ which is a contradiction by
Lemma \ref{Lem:f(t)<1}.
\end{proof}

\begin{Lemma}   \label{Lem:Selection}
Suppose that \eqref{eq:A1}--\eqref{eq:A4} and \eqref{eq:A6} hold,
and let $f \in \SM$ with $f(t) < 1$ for some $t \not = 0$.
Let $(t_n')_{n \geq 1}$ be a sequence of non-zero reals tending to
$0$. Then there are a subsequence $(t_{n_k}')_{k \geq 1}$ and a
function $g:[-1,1] \to [0,\infty)$ which is decreasing on
$[-1,0]$, increasing on $[0,1]$, and satisfies $g(0)=0$ and
$g(1)=1$ such that
\begin{equation}\label{eq:selection}
\frac{1\!-\!f(z t_k)}{1\!-\!f(t_k)} ~\underset{k \to
\infty}{\longrightarrow}~  g(z)    \quad   \text{for all }
z \in [-1,1]
\end{equation}
where $(t_k)_{k \geq 1} = (t_{n_k}')_{k \geq 1}$ or $(t_k)_{k \geq 1} = (-t_{n_k}')_{k \geq 1}$.
The sequence $(t_k)_{k \geq 1}$ can be chosen such that
\begin{equation}    \label{eq:property of t_k}
\liminf_{k \to \infty} (1\!-\!f(t_k))/(1\!-\!\varphi(|t_k|^{\alpha})) ~\geq~ K_{\mathrm{l}} ~>~ 0.
\end{equation}
\end{Lemma}
\begin{proof}
From Lemma \ref{Lem:f(t)<1} we infer that $1\!-\!f(t) > 0$ for all
$t \not =0$. Thus, the ratio in \eqref{eq:selection} is
well-defined. Recalling that $f(t)$ is nonincreasing for $t\geq 0$ and nondecreasing for $t<0$ we conclude that,
for $z \in [0,1]$, $(1\!-\!f(zt))/(1\!-\!f(t)) \leq 1$,
while for $z \in [-1,0]$, $(1\!-\!f(zt))/(1\!-\!f(t)) \leq(1\!-\!f(-t))/(1\!-\!f(t))$. The problem
here is that at this point we do not know whether the latter ratio is bounded as $t \to 0$.
However, according to Lemma \ref{Lem:ratios_bounded_away_from_infty}
\begin{equation*}
\liminf_{n \to \infty} \frac{(1\!-\!f(t_n')) \vee (1\!-\!f(-t_n'))}{1-\varphi(|t_n'|^{\alpha})} ~\geq~  K_{\mathrm{l}} ~>~ 0.
\end{equation*}
Hence, there is a subsequence of either $(t_n')_{n \geq 1}$ or $(-t_n')_{n \geq 1}$ which,
for convenience, we again denote by $(t_n')_{n \geq 1}$ such that
\begin{equation*}    \label{eq:alternatives}
\liminf_{n \to \infty} \frac{1\!-\!f(t_n')}{1-\varphi(|t_n'|^{\alpha})} ~\geq~  K_{\mathrm{l}} ~>~ 0.
\end{equation*}
Another appeal to Lemma \ref{Lem:ratios_bounded_away_from_infty} gives
\begin{equation*}    \label{eq:limsup<infty}
\limsup_{n \to \infty} \frac{1\!-\!f(-t_n')}{1\!-\!f(t_n')}
~=~ \limsup_{n \to \infty} \frac{1\!-\!f(-t_n')}{1-\varphi(|t_n'|^{\alpha})} \frac{1-\varphi(|t_n'|^{\alpha})}{1\!-\!f(t_n')}
~\leq~  \frac{K_{\mathrm{u}}}{K_{\mathrm{l}}}   ~<~ \infty.
\end{equation*}
Hence, the selection principle enables us to choose a subsequence $(t_n)_{n \geq 1}$ of $(t_n')_{n \geq 1}$
along which convergence in \eqref{eq:selection} holds for each $z \in [-1,1]$
(details of the selection argument can be found in \cite[Lemma 11.2]{Alsmeyer+Biggins+Meiners:2012}).
The resulting limit $g$ satisfies $g(0)=0$ and $g(1)=1$. From the construction,
it is clear that \eqref{eq:property of t_k} holds.
\end{proof}

\begin{Lemma}   \label{Lem:BK's_trick}
Suppose that \eqref{eq:A1}--\eqref{eq:A4} hold and let $f \in \SM$ with $f(t) < 1$ for some $t \not = 0$. Then
\begin{equation}    \label{eq:reg_var_2sided}
\lim_{t \to 0} \frac{1-f(z t)}{1-f(t)}  ~=~ |z|^{\alpha} \quad \text{for all } z \in \R.
\end{equation}
\end{Lemma}
\begin{proof}
Taking expectations in \eqref{eq:M's equation along ladder lines} for $u=0$
reveals that $f \in \SM$ satisfies \eqref{eq:FE} with $T$ replaced by
$T^>=(L(v))_{v\in\G_1^>}$. Furthermore, Proposition
\ref{Prop:embedded_BRW} ensures that the validity of
\eqref{eq:A1}-\eqref{eq:A4} for $T$ carries over to $T^>$ with the
same characteristic exponent $\alpha$. Since $|L(v)|<1$ a.s. for
all $v\in \G_1^>$ we can and do assume until the end of proof that
assumptions \eqref{eq:A1}-\eqref{eq:A4} and \eqref{eq:A6} hold.

As in
\cite{Alsmeyer+Biggins+Meiners:2012,Biggins+Kyprianou:1997,Iksanov:2004},
the basic equation is the following rearrangement of \eqref{eq:FE}
\begin{equation}    \label{eq:BK's_trick}
\frac{1\!-f\!(zt_n)}{|z|^{\alpha} (1\!-\!f(t_n))}
~=~ \E \bigg[\sum_{j \geq1} |T_j|^{\alpha} \frac{1\!-\!f(z T_j t_n)}{|zT_j|^{\alpha}(1\!-\!f(t_n))} \prod_{k<j} f(z T_k t_n)\bigg]
\end{equation}
for $z \in [-1,1]$ and $(t_n)_{n\geq 1}$ as in Lemma \ref{Lem:Selection}.
The idea is to take the limit as $n \to \infty$ and then interchange limit and expectation. To justify the
interchange, we use the dominated convergence theorem. To this
end, we need to bound the ratios
\begin{align}   \label{eq:3 ratios}
|T_j|^{\alpha} \frac{1\!-\!f(z T_j t_n)}{|zT_j|^{\alpha}(1\!-\!f(t_n))} ~=~ |T_j|^{\alpha}
\frac{1\!-\!f(z T_j t_n)}{1\!-\!\varphi(|z T_j t_n|^{\alpha})} \frac{1\!-\!\varphi(|z T_j t_n|^{\alpha})}{|zT_j|^{\alpha}(1\!-\!\varphi(|t_n|^{\alpha}))}
\frac{1\!-\!\varphi(|t_n|^{\alpha})}{1\!-\!f(t_n)}.
\end{align}
By Lemma \ref{Lem:ratios_bounded_away_from_infty}, for all sufficiently large $n$,
\begin{equation*}
\frac{1\!-\!f(z T_j t_n)}{1\!-\!\varphi(|z T_j
t_n|^{\alpha})} ~\leq~   K_{\mathrm{u}}+1    \quad   \text{for all } j \geq 1 \text{ and } z \in [-1,1] \text{ a.s.}
\end{equation*}
Since $(t_n)_{n \geq 1}$ is chosen such that \eqref{eq:property of t_k} holds, for all sufficiently large $n$,
\begin{equation*}
\frac{1\!-\!\varphi(|t_n|^{\alpha})}{1\!-\!f(t_n)} ~\leq~ K_{\mathrm{l}}^{-1}+1.
\end{equation*}
Finally, when \eqref{eq:A4a} holds, then $D(t) =
t^{-1}(1-\varphi(t)) \to 1$ as $t \to \infty$. This implies that
the second ratio on the right-hand side of \eqref{eq:3 ratios}
remains bounded uniformly in $z$ for all $j \geq 1$ a.s.\
as $n \to \infty$. If \eqref{eq:A4b} holds, $D(t)$ is slowly
varying at $0$ and, using a Potter bound \cite[Theorem
1.5.6(a)]{Bingham+Goldie+Teugels:1989}, one infers that, for
an appropriate constant $K>0$ and all sufficiently large $n$,
\begin{equation*}
\frac{1\!-\!\varphi(|z T_j t_n|^{\alpha})}{|zT_j|^{\alpha}(1\!-\!\varphi(|t_n|^{\alpha}))} ~=~ \frac{D(|zT_j t_n|^{\alpha})}{D(|t_n|^{\alpha})}
~\leq~ K |zT_j|^{\theta-\alpha} \quad   \text{for all } j \geq 1, \
z \in [-1,1] \text{ a.s.}
\end{equation*}
where $\theta$ comes from \eqref{eq:A4b}.
Consequently, the dominated convergence theorem applies and letting $n \to \infty$ in \eqref{eq:BK's_trick} gives
\begin{equation*}
g(z)/|z|^{\alpha}   ~=~ \E \bigg[ \sum_{j \geq 1}
|T_j|^{\alpha} \frac{g(zT_j)}{|zT_j|^{\alpha}}
\bigg], \qquad z \in [-1,1]
\end{equation*}
with $g$ defined in \eqref{eq:selection}. For $x \geq 0$,
define $h_1(x) := e^{\alpha x} g(e^{-x})$ and $h_{-1}(x) :=
e^{\alpha x} g(-e^{-x})$. $h_1$ and $h_{-1}$ satisfy the following
system of Choquet-Deny type functional equations
\begin{eqnarray}
h_1(x)      &=& \int h_1(x+y) \mu_{\alpha}^+(\dy) + \int h_{-1}(x+y) \, \mu_{\alpha}^-(\dy),	\quad x\geq 0,		\label{eq:CDE1} \\
h_{-1}(x)   &=& \int h_{-1}(x+y) \mu_{\alpha}^+(\dy) + \int h_1(x+y) \, \mu_{\alpha}^-(\dy),	\quad x\geq 0,
\label{eq:CDE2}
\end{eqnarray}
where
\begin{eqnarray*}
\mu_{\alpha}^{\pm}(B)   &=& \E \bigg[\sum_{j \geq 1} \1_{\{\pm T_j > 0\}} |T_j|^{\alpha} \1_{\{S(j) \in B\}}\bigg]  ,   \quad   B \subseteq [0,\infty) \text{ Borel.}
\end{eqnarray*}
By \eqref{eq:A6}, $\mu_{\alpha}^+$ and $\mu_{\alpha}^-$ are
concentrated on $\Rp$ and $\mu_{\alpha}^+(\Rp)+\mu_{\alpha}^-(\Rp)
= 1$. By Lemma \ref{Lem:Selection}, $g$ is bounded and, hence,
$h_1$ and $h_{-1}$ are locally bounded on $[0,\infty)$. Now use
that $1 = h_1(0)$ in \eqref{eq:CDE1} to obtain that $h_j(y_0) \geq
1$ for some $j \in \{1,-1\}$ and some $y_0 > 0$. Then, since $g$
is nonincreasing on $[-1,0]$ and nondecreasing on $[0,1]$, $h_j(y) > 0$ for all $y \in [0,y_0]$.

The desired conclusions can be drawn from \cite[Theorem 1]{Ramachandran+Lau+Gu:1988}, but it requires less additional
arguments to invoke the general Corollary 4.2.3 in \cite{Rao+Shanbhag:1994}.
Unfortunately, we do not know at this point that the functions $h_j$, $j=1,-1$ are
continuous which is one of the assumptions of Chapter 4 in
\cite{Rao+Shanbhag:1994}. On the other hand, as pointed out right
after (3.1.1) in \cite{Rao+Shanbhag:1994}, this problem can be
overcome by considering
\begin{equation*}
H_j^{(k)}(x) ~=~ k \int_0^{1/k} h_j(x+y) \, \dy,    \qquad  j=-1,1,\ k \in \N.
\end{equation*}
Since the $h_j$ are nonnegative, so are the $H_j^{(k)}$. Further,
since one of the $h_j$ are strictly positive on $[0,y_0]$, the corresponding $H_j^{(k)}$ is strictly positive on $[0,y_0)$ as well.
Local boundedness of the $h_j$ implies continuity of the $H_j^{(k)}$.
For fixed $k \in \N$ and $j=1,-1$, using the definition of
$H_j^{(k)}$, \eqref{eq:CDE1} or \eqref{eq:CDE2}, respectively, and
Fubini's theorem, one can conclude that
\begin{align*}
H_j^{(k)}&(x)
~=~
%k \int_0^{1/k} h_j(x+t) \dt    \\
%& =~
%k \int_0^{1/k} \left(\int h_j(x+t+y) \mu_{\alpha}^+(\dy) + \int h_{-j}(x+t+y) \mu_{\alpha}^-(\dy) \right) \dt  \\
%& =~
%\int k \int_0^{1/k} \!\!\!\!\! h_j(x+t+y) \dt \mu_{\alpha}^+(\dy) + \int k \int_0^{1/k} \!\!\!\!\! h_{-j}(x+t+y) \dt \mu_{\alpha}^-(\dy)       \\
%& =~
\int H_j^{(k)}(x+y) \mu_{\alpha}^+(\dy) + \int H_{-j}^{(k)}(x+y) \mu_{\alpha}^-(\dy).
\end{align*}
Thus, for fixed $k$, $H_1^{(k)}$ and $H_{-1}^{(k)}$ satisfy the same system of equations \eqref{eq:CDE1} and \eqref{eq:CDE2}.
From Corollary 4.2.3 in \cite{Rao+Shanbhag:1994} we now infer that there are product-measurable processes $(\xi_j(x))_{x \geq 0}$, $j=1,-1$ with
\begin{itemize}
    \item[(i)]  $H_j^{(k)}(x) = H_j^{(k)}(0) \E [\xi_j(x)] < \infty$, $x \geq 0$;
    \vspace{-0.2cm}
    \item[(ii)] $\xi_j(x+y) = \xi_j(x) \xi_j(y)$ for all $x, y \geq 0$;
    \vspace{-0.2cm}
    \item[(iii)]    $\int \xi_j(x) \, \mu_{\alpha}(\dx) = 1$ (pathwise).
\end{itemize}
(ii) together with the product-measurability of $\xi_j$ implies that $\xi_j(x) = e^{\alpha_j x}$ for all $x \geq 0$ for some random variable $\alpha_j$.
Then condition (iii) becomes $\int e^{\alpha_j x} \mu_{\alpha}(\dx) = 1$ (pathwise) which can be rewritten as $\varphi_{\mu_{\alpha}}(\alpha_j)=0$ (pathwise)
for the Laplace transform $\varphi_{\mu_{\alpha}}$ of $\mu_{\alpha}$. By \eqref{eq:A6}, $\varphi_{\mu_{\alpha}}$ is strictly decreasing and hence $\alpha_j=0$ (pathwise).
From (i) we therefore conclude $H_j^{(k)}(x) = H_j^{(k)}(0)$, $j=1,-1$.
Since $h_j$ is locally bounded and has only countably many discontinuities, $H_j^{(k)}(x) \to h_j(x)$ for (Lebesgue-)almost all $x$ in $[0,\infty)$. From the fact that the $H_j^{(k)}$ are constant, we infer that the $h_j$ are constant (Lebesgue-)a.e.
This in combination with the fact that $e^{-\alpha x} h_j(x) = g(je^{-x})$ is monotone implies that $h_j$ is constant on $(0,\infty)$,
$h_j(x) = c_j$ for all $x \geq 0$, say, $j=1,-1$. From $H_j^{(k)}>0$ on $[0,y)$ for some $y$ we further conclude $c_j > 0$, $j=1,-1$.
Now \eqref{eq:CDE1} for $x>0$ can be rewritten as $c_1 = pc_1+qc_{-1}$.
Since $0<p,q<1$ by assumption, we conclude $c_{-1} = c_1 =: c$.
Finally, \eqref{eq:CDE1} for $x=0$ yields $1=c$.

By now we have shown that for any sequence $(t_n')_{n \geq 0}$ in $\R\setminus\{0\}$ with $t_n' \to 0$
there is a subsequence $(t_{n_k}')_{k \geq 0}$
such that
\begin{equation}    \label{eq:limiting relation with c=1}
\frac{1-f(zt_k)}{1-f(t_k)}
~\underset{k \to \infty}{\longrightarrow}~
|z|^{\alpha} \quad \text{for } |z| \leq 1
\end{equation}
for $(t_k)_{k \geq 1} = (t_{n_k}')_{k \geq 1}$ or $(t_k)_{k \geq 1} = (-t_{n_k}')_{k \geq 1}$.
Replacing $z$ by $-z$ in the formula above, we see that the same limiting relation holds for the sequence $(-t_k)_{k \geq 1}$
so that every sequence tending to $0$ has a subsequence along which \eqref{eq:limiting relation with c=1} holds.
This implies \eqref{eq:reg_var_2sided}.
\end{proof}

\begin{proof}[Proof of Theorem
\ref{Thm:d=1,2sided_FE_disintegrated}]
\underline{Case III:} Let $f \in \SM$. If $f(t)=1$ for some $t \not =0$, then
$f(u)=1$ for all $u\in\R$ by Lemma \ref{Lem:f(t)<1}. In this case $M(t)=1$ for all $t\in\R$, and
\eqref{eq:d=1_2sided_FE_disintegrated} holds with $c_1=c_{-1}=0$.
Assume now that $f(t) \not = 1$ for all $t \not = 0$.
Using Lemma \ref{Lem:BK's_trick} and arguing as in the proof
of \cite[Lemma 8.8]{Alsmeyer+Biggins+Meiners:2012} we conclude
\begin{eqnarray*}
-\log (M(t))
%& = &
%\lim_{n \to \infty} \sum_{|v|=n} 1-f(L(v)t)    \\
%& = &
%\lim_{n \to \infty} \sum_{|v|=n} \frac{1-f(L(v)t)}{1-f(L(v))} (1-f(L(v)))
~=~ |t|^{\alpha} (-\log (M(1)))
\end{eqnarray*}
for any $t\neq 0$ and
\begin{eqnarray*}
-\log (M(1))
%& = &
%\lim_{k \to \infty} \sum_{|v|=n+k} 1-f(L(v))   \\
%& = &
%\sum_{|v|=n} \lim_{k \to \infty} \sum_{|w|=k} \frac{1-f(L(vw))}{1-f([L(w)]_v)} (1-f([L(w)]_v)) \\
%& = &
%\sum_{|v|=n} |L(v)|^{\alpha} \lim_{k \to \infty} \sum_{|w|=k} (1-f([L(w)]_v))
~=~ \sum_{|v|=n} |L(v)|^{\alpha} [-\log (M(1))]_v       \quad   \text{a.s.~for all } n \in \N_0.
\end{eqnarray*}
Since $f$ takes values in $[0,1]$, we have $0 \leq M(1) \leq 1$ a.s.
Moreover, by the dominated convergence theorem,
\begin{equation*}
1   =   f(0)    =   \lim_{t \to 0} f(t) =   \lim_{t \to 0} \E [M(t)]    =   \E \big[\lim_{t \to 0} M(1)^{|t|^{\alpha}}\big] =   \Prob(M(1) > 0).
\end{equation*}
Consequently, $0 < M(1)\leq 1$ a.s. Since $f(t) \neq 1$ for $t \neq 0$ we infer $\Prob(M(1)=1)<1$.
Therefore,
$-\log (M(1))$ is a nonnegative, non-null endogenous fixed point
of the smoothing transform with weights $|T_j|^{\alpha}$. From
Proposition \ref{Prop:endogeny}(a), we infer the existence of a
constant $c > 0$ such that $-\log (M(1)) = cW$. Consequently,
$M(t) = \exp(-Wc|t|^{\alpha})$ a.s.\ for all $t \in \R$.
\end{proof}

\subsection{Determining $\nu$ and $\mathbf{\Sigma}$}

\begin{Lemma}   \label{Lem:nu_evaluated}
Suppose that \eqref{eq:A1}-\eqref{eq:A4} hold.
Let $(\W,\mathbf{\Sigma},\nu)$ be the random L\'{e}vy triplet which appears in \eqref{eq:char_exponent}.
\begin{itemize}
    \item[(a)]
        There exists a finite measure $\sigma$ on the Borel subsets of $\Sd$ such that
        \begin{equation}    \label{eq:Levy_measure}
        \nu(A)  ~=~ W  \iint_{\Sd \times (0,\infty)} \1_{A}(rs) r^{-(1+\alpha)} \, \sigma(\ds) \, \dr
        \end{equation}
        for all Borel sets $A \subseteq \R^d$ a.s. Furthermore,
        $\sigma$ is symmetric, \textit{i.e.}, $\sigma(B)=\sigma(-B)$ for all Borel sets $B \subseteq \Sd$ if $\mathbb{G}(T) = \R^*$.
        $\alpha \geq 2$ implies $\sigma = 0$ a.s.\ (and, thus, $\nu = 0$ a.s.).
    \item[(b)]
        If $\alpha \not = 2$, then $\mathbf{\Sigma} = 0$ a.s.
        If $\alpha = 2$, then there is a deterministic symmetric positive semi-definite (possibly zero) matrix $\Sigma$ with $\mathbf{\Sigma} = W\Sigma$ a.s.
\end{itemize}
\end{Lemma}
\begin{proof}
By \eqref{eq:M's equation} and \eqref{eq:char_exponent},
\begin{align*}
\imag \langle \W, \bt \rangle & - \frac{\bt \mathbf{\Sigma} \bt^{\!\textsf{T}}}{2}
+ \int \left(e^{\imag \langle \bt, \bx \rangle} - 1 - \frac{\imag \langle \bt, \bx \rangle}{1+| \bx |^2} \right) \, \nu(\dbx)   \\
& ~=~
\imag \sum_{|v|=n} L(v) \langle [\W]_v, \bt \rangle ~-~ \frac{\sum_{|v|=n} L(v)^2  \bt [\mathbf{\Sigma}]_v \bt^{\!\textsf{T}}}{2}   \\
& \hphantom{~=}~ + \sum_{|v|=n} \int \left(e^{\imag L(v) \langle \bt, \bx \rangle} - 1 - \frac{\imag L(v)\langle \bt,\bx \rangle}{1+| \bx |^2} \right) \, [\nu]_v(\dbx)  \\
& ~=~
\imag \sum_{|v|=n} L(v)  \left(\langle[\W]_v, \bt\rangle + \int \left[\frac{\langle \bt, \bx \rangle}{1+L(v)^2| \bx |^2} - \frac{\langle \bt, \bx \rangle}{1+| \bx |^2} \right] \, [\nu]_v(\dx) \right) \\
&   \hphantom{~=}~ - \frac{\sum_{|v|=n} L(v)^2 \bt [\mathbf{\Sigma}]_v \bt^{\!\textsf{T}}}{2}    \\
& \hphantom{~=}~ + \sum_{|v|=n} \int \left(e^{\imag L(v)\langle \bt,\bx \rangle} - 1 - \frac{\imag L(v) \langle \bt,\bx \rangle}{1+L(v)^2 | \bx |^2} \right) \, [\nu]_v(\dbx),
\quad   \bt \in \R^d.
\end{align*}
The uniqueness of the L\'evy triplet implies
\begin{align}
\mathbf{\Sigma}  ~&=~   \sum_{|v|=n} L(v)^2 [\mathbf{\Sigma}]_v,    \label{eq:Sigma_endogenous} \\
\int g(\bx) \, \nu(\dbx)    ~&=~    \sum_{|v|=n} \int g(L(v) \bx) \, [\nu]_v(\dbx)  \label{eq:nu}
\end{align}
a.s.~for all $n \in\N_0$ and all non-negative
Borel-measurable functions $g$ on $\R^d$.

\noindent
(a)
Let $I_{r}(B) = \{\bx \in \R^d: r \leq |\bx|, \bx/|\bx| \in B\}$ where $r>0$ and $B$ is a Borel subset of $\Sd$.
Define $I_r(B)$ for $r < 0$ as $I_{|r|}(-B)$.
Since $\nu$ is a (random) L\'evy measure, $\nu(I_r(B)) < \infty$ a.s.\ for all $r \not = 0$ and all $B$.
When choosing $g = \1_{I_{r}(B)}$, \eqref{eq:nu} becomes
\begin{equation}    \label{eq:nu_I_t}
\nu(I_{r}(B))   ~=~ \sum_{|v|=n} [\nu]_v(I_{L(v)^{-1} r}(B))
\end{equation}
a.s.\ for all $n\in\N_0$. For fixed $B$ define
\begin{equation*}
f_{B}(r)    ~:=~    \begin{cases}
                        1                                       &   \text{if }  r = 0,          \\
                        \E [\exp(-\nu(I_{r^{-1}}(B)))]              &   \text{if }  r \not = 0.
                        \end{cases}
\end{equation*}
Then, by \eqref{eq:nu_I_t} and the independence of $(L(v))_{|v|=n}$ and
$([\nu]_v)_{|v|=n}$,
\begin{equation*}
f_{B}(r)    ~=~ \E \bigg[ \exp\bigg(-\sum_{|v|=n} [\nu]_v(I_{L(v)^{-1} r^{-1}}(B))\bigg)\bigg]
~=~ \E \bigg[\prod_{|v|=n} f_{B}(L(v)r)\bigg]
\end{equation*}
for all $r \in \R$. Further, $f_{B}$ is nondecreasing on $(-\infty,0]$ and nonincreasing on $[0,\infty)$.
Since $I_{r^{-1}}(B) \downarrow \emptyset$ as $r \uparrow 0$ or $r \downarrow 0$, $f_{B}$ is continuous at $0$, and we conclude
that $f_{B} \in \SM$.
%Moreover, $f_{B}$
%is continuous at $0$, for $I_{r^{-1}}(B) \downarrow \emptyset$ as
%$r \uparrow 0$ or $r \downarrow 0$. Consequently, $f_{B} \in \SM$.
From Theorem \ref{Thm:d=1,2sided_FE_disintegrated}, we infer that
the limit $M_{B}$ of the multiplicative martingales associated with
$f_{B}$ is of the form
\begin{equation*}
M_{B}(r)    ~=~ \begin{cases}
                    \exp(-W \sigma(B) \alpha^{-1} r^{\alpha})   &   \text{for } r \geq 0,       \\
                    \exp(-W \sigma(-B) \alpha^{-1} |r|^{\alpha})    &   \text{for } r \leq 0
                    \end{cases}
\quad   \text{a.s.},
\end{equation*}
where $\sigma(B)$ and $\sigma(-B)$ are nonnegative constants (depending on $B$ but not on $r$) and $\sigma(B)=\sigma(-B)$ if $\mathbb{G}(T) = \R^*$.
On the other hand, by an argument as in \cite[Lemma 4.8]{Alsmeyer+Meiners:2013},
we infer that $M_{B}(r) = \exp(-\nu(I_{r^{-1}}(B)))$ for all $r \in \R$ a.s.\ and thus
\begin{equation}    \label{eq:nu(I_r(B))}
\nu( I_r(B))    ~=~ \begin{cases}
                    W \sigma(B) \alpha^{-1} r^{-\alpha} &   \text{for } r > 0,      \\
                    W \sigma(B) \alpha^{-1} |r|^{-\alpha}   &   \text{for } r < 0
                    \end{cases}
\quad   \text{a.s.}
\end{equation}
Let $\mathcal{D} := \{[\ba,\bb) \cap \Sd: \ba,\bb \in \Q^d\}$ where $[\ba,\bb)=\{\bx \in \R^d: a_k \leq x_k < b_k \text{ for } k=1,\ldots,d\}$.
$\mathcal{D}$ is a countable generator of the Borel sets on $\Sd$.
\eqref{eq:nu(I_r(B))} holds for all $B \in \mathcal{D}$ and all $r \in \Q$ simultaneously,
a.s. From \eqref{eq:nu(I_r(B))} one infers (since $\Prob(W > 0) > 0$)
that $\sigma$ is a content on $\mathcal{D}$ (that is, $\sigma$ is nonnegative, finitely additive and $\sigma(\emptyset)=0$).
It is even $\sigma$-additive on $\mathcal{D}$ (whenever the countable union
of disjoint sets from $\mathcal{D}$ is again in $\mathcal{D}$).
Thus, there is a unique continuation of $\sigma$ to a measure on
the Borel sets on $\Sd$. For ease of notation, this measure will
again be denoted by $\sigma$. By uniqueness,
\eqref{eq:nu(I_r(B))} holds for all Borel sets $B \subseteq \Sd$ and $r \in \Q$
simultaneously, a.s.\ and, by standard arguments, extends to all
$r \in \R$ as well. Lemma 2.1 in \cite{Kuelbs:1973} now yields
that \eqref{eq:Levy_measure} holds for all Borel sets $A \subseteq
\R^d$ a.s.

\noindent (b)
Turning to $\mathbf{\Sigma}$, we write $\mathbf{\Sigma} = (\mathbf{\Sigma}_{k\ell})_{k,\ell=1,\ldots,d}$.
For every $k,\ell =1,\ldots,d$, \eqref{eq:Sigma_endogenous} implies
\begin{equation}    \label{eq:Sigma_kl endogenous}
\mathbf{\Sigma}_{k\ell} ~=~ \sum_{|v|=n} L(v)^2 [\mathbf{\Sigma}_{k\ell}]   \quad   \text{a.s.~for all } n \in \N_0,
\end{equation}
that is, $\mathbf{\Sigma}_{k\ell}$ is an endogenous fixed point w.r.t.~$(T_j^2)_{j \geq 1}$.
If $\alpha \not = 2$, then $\mathbf{\Sigma}_{k\ell} = 0$ a.s.~by Proposition \ref{Prop:endogeny}(c).
Suppose $\alpha=2$. Since $\mathbf{\Sigma}_{kk}\geq 0$ a.s., we infer $\mathbf{\Sigma}_{kk} = W \Sigma_{kk}$ for some
$\Sigma_{kk} \geq 0$ by Proposition \ref{Prop:endogeny}(a).
Further, the Cauchy-Schwarz inequality implies $-\mathbf{\Sigma}_{k\ell}+W \sqrt{\Sigma_{kk} \Sigma_{\ell\ell}}\geq 0$ a.s.
Hence, \eqref{eq:Sigma_kl endogenous} and Lemma 4.16 in \cite{Alsmeyer+Meiners:2013} imply
$\mathbf{\Sigma}_{k\ell} = W\Sigma_{k\ell}$ for some
$\Sigma_{k\ell} \in \R$. Consequently, $\mathbf{\Sigma} = W
\Sigma$ for $\Sigma = (\Sigma_{k\ell})_{k,\ell=1,\ldots,d}$. Since
$\mathbf{\Sigma}$ is symmetric and positive semi-definite, so is $\Sigma$.
\end{proof}

\subsection{The proofs of Theorems \ref{Thm:Z} and \ref{Thm:endogeny}}  \label{subsec:endogeny proofs}

The key ingredient to the proofs of Theorems \ref{Thm:Z} and \ref{Thm:endogeny}
is a bound on the tails of fixed points. This bound is provided by the following lemma.

\begin{Lemma}   \label{Lem:tail bounds}
Let $d=1$ and assume that \eqref{eq:A1}-\eqref{eq:A4} hold.
Let $X$ be a solution to \eqref{eq:generalized stable}.
Then
\begin{itemize}
	\item[(a)]
		$\Prob(|X| > t) = O(1-\varphi(t^{-\alpha}))$ as $t \to \infty$.
	\vspace{-0.2cm}
	\item[(b)]
		$\Prob(|X| > t) = o(1-\varphi(t^{-\alpha}))$ as $t \to \infty$ if $X$ is an endogenous fixed point.
\end{itemize}
\end{Lemma}
\begin{proof}
By \eqref{eq:nu_ladder} and \eqref{eq:Levy_measure}, with $F$ denoting the distribution of $X$,
\begin{equation}    \label{eq:tail bounds}
\sum_{v \in \mathcal{T}_u} F(A/L(v))
~\to~   W  \iint \1_{A}(rs) r^{-(1+\alpha)} \, \sigma(\ds) \, \dr      \quad
\text{as $u \to \infty$ a.s.}
\end{equation}
for every Borel set $A \subset \R$ that has a positive distance
from $0$ (since $\nu$ is continuous).
Use the above formula for $A = \{|x| > 1\}$ and rewrite it in terms of $G(t) := t^{-\alpha} \Prob(|X| > t^{-1})$ for $t > 0$.
This gives
\begin{equation}
\sum_{v \in \mathcal{T}_u} e^{-\alpha S(v)} G(e^{-S(v)}) ~\to~   W \frac{\sigma(\{1,-1\})}{\alpha}
\quad   \text{as $u \to \infty$ a.s.}
\end{equation}
Now one can follow the arguments given in the proof of Lemma 4.9 in \cite{Alsmeyer+Meiners:2013} to conclude (a).

Finally, assume that $X$ is endogenous, $X = X^{(\varnothing)}$, say, for the root value of an endogenous recursive tree process $(X^{(v)})_{v \in \V}$.
Denote by $\Phi$ the limit of the multiplicative martingales associated with the Fourier transform $\phi$ of $X$.
Then it is implicit in the proof of Lemma \ref{Lem:endogeny} that $\Phi(t) = \exp(\imag t X^{(\varnothing)})$ a.s.,
that is, the L\'evy measure in the random L\'evy triplet of $\Phi$ vanishes a.s. Hence, the right-hand side of \eqref{eq:tail bounds} vanishes a.s.
(b) now follows by the same arguments as assertion (b) of Lemma 4.9 in \cite{Alsmeyer+Meiners:2013}.
\end{proof}

\begin{proof}[Proof of Theorem \ref{Thm:endogeny}]
%First assume $X$ is an endogenous fixed point w.r.t.~$T$.
%Lemma \ref{Lem:tail bounds} gives
%\begin{equation}    \label{eq:tails of X}
%\Prob(|X|>t)    ~=~ o(1-\varphi(t^{-\alpha}))   \quad   \text{as } t \to \infty.
%\end{equation}
%Further recall that Proposition \ref{Prop:W}(c) implies that $1-\varphi(t^{-\alpha})$ is regularly varying of index $-\alpha$ as $t \to \infty$.

\noindent
(a) Assume that $\alpha < 1$.
Then one can argue as in the proof of Theorem 4.12 in \cite{Alsmeyer+Meiners:2013}
(with $L(v)$ there replaced by $|L(v)|$ here) to infer that $X = 0$ a.s.

\noindent (b)
Here $\alpha=1$.

\noindent
\underline{Case I} in which $T_j \geq 0$ a.s., $j \in\N$. The result follows from Proposition \ref{Prop:endogeny}(d).

\noindent
\underline{Case II} in which $T_j \leq 0$ a.s., $j \in\N$.
If $X$ is an endogenous fixed point w.r.t.~$T$, then $X$ is an endogenous fixed point w.r.t.~$(L(v))_{|v|=2}$.
We use Lemma \ref{Lem:WBP_even} to reduce the problem to the already settled Case I where $T_j$, $j\in\N$
have to be replaced with the nonnegative $L(v)$, $|v|=2$. This allows us to conclude that $X=cW$ for some
$c \in \R$. However, since
\begin{equation*}
cW  ~=~ X   ~=~ \sum_{j \geq 0} T_j [X]_j   ~=~ \sum_{j \geq 0}
(-|T_j|) c [W]_j    ~=~ -c W    \quad   \text{a.s.}
\end{equation*}
we necessarily have $c=0$.

\noindent
\underline{Case III}.
Using the embedding technique of Section \ref{subsec:embedded},
we can assume w.l.o.g.~that \eqref{eq:A6} holds in addition to \eqref{eq:A1}--\eqref{eq:A5}.
If after the embedding, we are in Case II rather than Case III, then $X = 0$ a.s., by what we have already shown.
Therefore, the remaining problem is to conclude that $X=0$ a.s.~in Case III under the assumptions \eqref{eq:A1}--\eqref{eq:A6}.

Let $\Phi$ denote the limit of the multiplicative
martingales associated with the Fourier transform of $X$. From the
proof of Lemma \ref{Lem:endogeny} we conclude that
$\Phi(t)=\exp({\rm i}tX)$ a.s.~which
together with \eqref{eq:W(h)} and \eqref{eq:bW ladder} implies
\begin{equation}    \label{eq:X representation with I}
X   ~=~ \lim_{t \to \infty} \sum_{v \in \mathcal{T}_t} L(v) I(|L(v)|^{-1})
\quad   \text{a.s. as } t \to \infty
\end{equation}
where $I(t): = \int_{\{|x| \leq t\}} x \, F(\dx)$.
Integration by parts gives
\begin{equation}    \label{eq:I integrated by parts}
I(t)    ~=~ \int_0^t \Prob(X>s)-\Prob(X<-s) \, \ds - t(\Prob(X > t) - \Prob(X<-t)).
\end{equation}
The contribution of the second term to \eqref{eq:X representation with I} is negligible, for
\begin{equation}\label{eq:relat}
t|\Prob(X > t) - \Prob(X<-t)| = o(D(t^{-1})) \quad	\text{as } t \to \infty
\end{equation}
by Lemma \ref{Lem:tail bounds}(b) and
\begin{equation*}
\sum_{v \in \mathcal{T}_t}|L(v)| D(|L(v)|^{-1})  ~\to~
W \quad \text{a.s.}
\end{equation*}
by \eqref{eq:twolimits}.
Hence,
\begin{equation}    \label{eq:X representation with int_0^L(v)^-1}
X   ~=~ \lim_{t \to \infty} \sum_{v \in \mathcal{T}_t} L(v) \int_0^{|L(v)|^{-1}} \!\!\!\!\! \big(\Prob(X\!>\!s)-\Prob(X\!<\!-\!s)\big) \, \ds
\quad   \text{a.s. as } t \to \infty.
\end{equation}
One can replace the integral from $0$ to $|L(v)|^{-1}$ above by
the corresponding integral with $|L(v)|^{-1}$ replaced by $e^t$.
To justify this, in view of \eqref{eq:relat}, it is enough to
check that
\begin{equation}    \label{eq:crucial limit}
\limsup_{t \to \infty} \sum_{v \in \mathcal{T}_t} |L(v)| \int_{e^t}^{|L(v)|^{-1}} \!\!\!\!\! (1-\varphi(s^{-1})) \, \ds ~\leq~  cW      \quad   \text{a.s.}
\end{equation}
for some constant $c \geq 0$. This statement is derived in the
proof of Theorem 4.13 in \cite{Alsmeyer+Meiners:2013} under the
assumptions \eqref{eq:A1}--\eqref{eq:A3}, \eqref{eq:A4a} and $\E [\sum_{j \geq 1} |T_j| (\log^-(|T_j|))^2]<\infty$;
see (4.41), (4.42) and the subsequent lines in the cited reference.
Consequently, we arrive at the following representation of $X$
\begin{equation}    \label{eq:X representation with int_0^e^t}
X   ~=~ \lim_{t \to \infty} \sum_{v \in \mathcal{T}_t} L(v) \int_0^{e^t} \! \big(\Prob(X\!>\!s)-\Prob(X\!<\!-\!s)\big) \ds
\quad   \text{a.s.}
\end{equation}
Now observe that \eqref{eq:A5} implies \eqref{eq:A4a}, and that,
under \eqref{eq:A1}-\eqref{eq:A4a}, $\lim_{t\to\infty} t (1-\varphi(t^{-1})) = 1$
because $W$ is the a.s.~limit of uniformly integrable martingale
$(\sum_{|v|=n}|L(v)|)_{n\geq 0}$. In combination with \eqref{eq:relat} this yields
$|\Prob(X > s) - \Prob(X<-s)|=o(s^{-1})$ as $s \to \infty$
and so, for every $\varepsilon > 0$ there is a $t_0$ such that
\begin{align*}
\bigg|\int_{e^{t_0}}^{e^t} & \big(\Prob(X\!>\!s)-\Prob(X\!<\!-\!s)\big) \ds \bigg|
~\leq~
\varepsilon \int_{e^{t_0}}^{e^t} s^{-1} \ds
~=~ \varepsilon (t-t_0)
\end{align*}
for all $t\geq t_0$.
Lemma \ref{Lem:sumL(v)->0} thus implies $X = 0$ a.s.~as claimed.

\noindent
(c) Assume that $\alpha > 1$.

\noindent
(i)$\Rightarrow$(iii): Let $X$ be a non-null endogenous fixed point w.r.t.~$T$.
Then using Lemma \ref{Lem:tail bounds} (b) and recalling that according to Proposition
\ref{Prop:W}(c) $1-\varphi(t^{-\alpha})$ is regularly varying of index $-\alpha$ at $\infty$
we conclude that $\E[|X|^\beta]<\infty$ for all $\beta\in (0,\alpha)$.
In particular, $\E [X|\A_n]$ is an $\mathcal{L}^\beta$-bounded martingale with limit (a.s.\ and in $\mathcal{L}^\beta$) $X$.
Further, using that $\{|v|=n: L(v) \not = 0\}$ is a.s.\ finite for each
$n\in\N_0$, we obtain
\begin{equation*}
\E [X|\A_n] ~=~ \E \bigg[ \sum_{|v|=n} L(v) [X]_v \Big| \A_n
\bigg] ~=~ Z_n \E [X]  \quad   \text{a.s.}
\end{equation*}
Hence $\E [Z_n] = 1$ for all $n\in\N_0$ and $X = \E [X] \cdot \lim_{n \to \infty} Z_n$ a.s.\ and in
$\mathcal{L}^{\beta}$ which, among others, proves the uniqueness of the endogenous fixed point up to a real scaling factor.

\noindent
(iii)$\Rightarrow$(ii): The implication follows because $(Z_n)_{n\geq 0}$ is a martingale,
and convergence in $\mathcal{L}^{\beta}$ for some $\beta > 1$ implies uniform integrability.

\noindent
(ii)$\Rightarrow$(i): For every $n \in \N_0$,
\begin{eqnarray}    \notag
Z & = &
\lim_{k \to \infty} Z_{n+k} ~=~ \lim_{k \to \infty} \sum_{|v|=n} \sum_{|w|=k} L(v) [L(w)]_v     \\
& = & \sum_{|v|=n} L(v) \lim_{k \to \infty}  \sum_{|w|=k} [L(w)]_v
~=~ \sum_{|v|=n} L(v) [Z]_v \quad   \text{a.s.} \label{eq:Z endogenous}
\end{eqnarray}
This means that $Z$ is an endogenous fixed point w.r.t.~$T$ which
is non-null because $\Prob(Z=0)<1$ by assumption.
\end{proof}

\begin{proof}[Proof of Theorem \ref{Thm:Z}]

\noindent (a) is Lemma 4.14(a) in \cite{Alsmeyer+Meiners:2013};
(b) is Theorem \ref{Thm:endogeny}(c) of the present paper.

\noindent (c) Let $\alpha \geq 2$. If $Z_1 = 1$ a.s., then $Z_n =
1$ a.s.\ and the a.s.~ convergence to $Z=1$ is trivial.
Conversely, assume that $\Prob(Z_1 = 1) < 1$ and that $Z_n \to Z$ a.s. as $n \to \infty$.
According to part (b) of the theorem,
$(Z_n)_{n\geq 0}$ is a martingale, and $Z_n \to Z$ in $\mathcal{L}^{\beta}$ for all $\beta \in (1,\alpha)$.
By an approach that is close to the one taken in \cite[Proof of Theorem 1.2]{AIPR2009} we shall show that
this produces a contradiction. Pick some $\beta \in (1,2)$ if $\alpha=2$ and take $\beta=2$ if $\alpha>2$.
For $Z_n \to Z$ in $\mathcal{L}^{\beta}$ to hold true it is
necessary that $\E [|Z_1-1|^{\beta}] < \infty$. Then, using the
lower bound in the Burkholder-Davis-Gundy inequality \cite[Theorem
11.3.1]{CT1997}, we infer that for some constant $c_{\beta}>0$ we
have
\begin{eqnarray}
\E [|Z-1|^\beta]
& \geq &
c_{\beta} \E \bigg[\bigg( \sum_{n \geq 0} (Z_{n+1}-Z_{n})^2 \bigg)^{\!\! \beta/2}\bigg]     \notag  \\
& = &
c_{\beta} \E \bigg[ \bigg(\sum_{n \geq 0} \bigg(\sum_{|v|=n} L(v)([Z_1]_v-1)\bigg)^{\!\! 2} \bigg)^{\!\! \beta/2}\bigg] \notag  \\
& \geq &
c_{\beta} \E \bigg[ \bigg( \sum_{n=0}^{m-1} \bigg(\sum_{|v|=n} L(v)([Z_1]_v-1)\bigg)^{\!\! 2} \bigg)^{\!\! \beta/2}\bigg]       \label{eq:L^2BDG_lower}
\end{eqnarray}
for every $m \in \N$. Since $\beta \in (1,2]$, the function $x
\mapsto x^{\beta/2}$ ($x \geq 0$) is concave which implies
$(x_1 + \ldots + x_m)^{\beta/2} \geq m^{\beta/2-1}
(x_1^{\beta/2}+\ldots+x_m^{\beta/2})$ for any $x_1, \ldots, x_m
\geq 0$. Plugging this estimate into \eqref{eq:L^2BDG_lower} gives
\begin{eqnarray*}
\E [|Z-1|^\beta]
& \geq &
c_{\beta} m^{\beta/2-1} \sum_{n=0}^{m-1} \E \bigg[ \bigg|\sum_{|v|=n} L(v)([Z_1]_v-1)\bigg|^\beta\bigg].
\end{eqnarray*}
Given $\A_n$, $\sum_{|v|=n} L(v)([Z_1]_v-1)$ is a weighted sum of i.i.d.\ centered random variables
and hence the terminal value of a martingale.
%On the right hand side, for $n=0,\ldots,m$, the integrand in the $n$th summand given $\A_n$ is a weighted sum of i.i.d.\ centered random variables
%and hence, again, a martingale.
Thus, we can again use the lower bound of the Burkholder-Davis-Gundy inequality \cite[Theorem 11.3.1]{CT1997}
and then Jensen's inequality on $\{W_n(2) > 0\}$ where $W_n(\gamma) = \sum_{|v|=n} |L(v)|^{\gamma}$
to infer
\begin{eqnarray*}
\E [|Z-1|^\beta]
& \geq &
c_{\beta}^2 m^{\beta/2-1} \sum_{n=0}^{m-1} \E \bigg[\bigg(\sum_{|v|=n} L(v)^2([Z_1]_v-1)^2\bigg)^{\!\! \beta/2}\bigg]   \\
& = &
c_{\beta}^2 m^{\beta/2-1} \sum_{n=0}^{m-1} \E \bigg[W_n(2)^{\beta/2} \bigg(\sum_{|v|=n} \frac{L(v)^2}{W_n(2)}([Z_1]_v-1)^2\bigg)^{\!\! \beta/2}\bigg]   \\
& \geq &
c_{\beta}^2 m^{\beta/2-1} \sum_{n=0}^{m-1} \E \bigg[W_n(2)^{\beta/2} \sum_{|v|=n} \frac{L(v)^2}{W_n(2)}([Z_1]_v-1)^\beta \bigg)\bigg]   \\
& = &
c_{\beta}^2 m^{\beta/2-1} \E [|Z_1-1|^{\beta}] \sum_{n=0}^{m-1} \E [W_n(2)^{\beta/2}].
\end{eqnarray*}
To complete the proof it suffices to show
\begin{equation}\label{eq:diver}
\underset{m\to\infty}{\lim}\,m^{\beta/2-1} \sum_{n=0}^{m-1} \E
[W_n(2)^{\beta/2}]=\infty
\end{equation}
for the latter contradicts $\E [|Z-1|^\beta] < \infty$.

\noindent {\em Case $\alpha > 2$}:
Recalling that $m(2)>1$ in view of \eqref{eq:A3} and that $\beta=2$ we have
$\E[W_n(2)^{\beta/2}]=\E [W_n(2)] = m(2)^n \to \infty$ and thereupon \eqref{eq:diver}.

\noindent {\em Case $\alpha=2$}:
Since \eqref{eq:A4a} is assumed, we infer $W_n(2) \to W$ a.s.\ and in $\mathcal{L}^1$.
In particular $\E [W_n(2)^{\beta/2}] \to \E [W^{\beta/2}] > 0$ which entails \eqref{eq:diver}.
\end{proof}

\begin{Rem} \label{Rem:Hu&Shi}
In Theorem \ref{Thm:Z}(c) the case when $\alpha=2$ and \eqref{eq:A4a} fails remains a challenge.
Here, some progress can be achieved once the asymptotics of $\E [W_n(2)^\gamma]$ as $n\to\infty$ has been understood,
where $\gamma\in (0,1)$. The last problem, which is nontrivial because $\lim_{n\to\infty} W_n(2)= 0$ a.s.,
was investigated in \cite[Theorem 1.5]{Hu+Shi:2009} under
assumptions which are too restrictive for our purposes.
\end{Rem}

\subsection{Solving the homogeneous equation in $\R^d$} \label{subsec:solving homogeneous}

\begin{Lemma}   \label{Lem:W(1)}
Assume that \eqref{eq:A1}--\eqref{eq:A4} hold,
that $\alpha = 1$, and that $\mathbb{G}(T)=\Rp$.
Further, assume that $\E [\sum_{j \geq 1} T_j (\log^-(T_j))^2] < \infty$.
Let $\bX = (X_1,\ldots,X_d)$ be a solution to \eqref{eq:generalized stable} with distribution function $F$
and let $\W(1) = (W(1)_1,\ldots,W(1)_d)$ be defined by
\begin{equation*}
\W(1)   ~=~ \lim_{t \to \infty} \sum_{v \in \mathcal{T}_t} L(v)
\int_{\{|\bx| \leq L(v)^{-1}\}} \bx \, F(\dbx)  \quad
\text{a.s.},
\end{equation*}
\textit{i.e.}, as in \eqref{eq:W(h)}.
Then there exists a finite constant $K>0$ such that $\max_{j=1,\ldots,d} |W(1)_j| \leq KW$ a.s.
\end{Lemma}
\begin{proof}
First of all, the existence of the limit that defines $\W(1)$ follows
from Lemma \ref{Lem:explicit_representation_along_ladder_lines}.
Fix $j \in \{1,\ldots,d\}$. Then, with $F_j$ denoting the distribution of $X_j$,
\begin{align}   \label{eq:W(1)_j}
W(1)_j
~&=~
\lim_{t \to \infty} \sum_{v \in \mathcal{T}_t} L(v) \int_{\{|x| \leq L(v)^{-1}\}} x \, F_j(\dx) \notag  \\
& =~
\lim_{t \to \infty} \sum_{v \in \mathcal{T}_t} L(v) \bigg(\int_0^{L(v)^{-1}} \!\!\!\!\!\!\! (\Prob(X_j>x)-\Prob(X_j<-x)) \, \dx   \notag  \\
&\hphantom{=~ \lim_{t \to \infty} \sum_{v \in \mathcal{T}_t} L(v) \bigg(}
- L(v)^{-1}(\Prob(X_j > L(v)^{-1}) - \Prob(X<-L(v)^{-1}))\bigg) \quad   \text{a.s.}
\end{align}
By Lemma \ref{Lem:tail bounds}, there is a finite constant $K_1 > 0$ such that
\begin{equation}    \label{eq:Prob(X_j>t)}
\Prob(|X_j| > t)    ~\leq~  K_1 (1-\varphi(t^{-1})  ~=~ K_1 t^{-1} D(t^{-1}) \quad   \text{for all sufficiently large } t.
\end{equation}
Therefore, by \eqref{eq:twolimits},
\begin{equation*}
\limsup_{t \to \infty} \bigg|\sum_{v \in \mathcal{T}_t} (\Prob(X_j > L(v)^{-1}) - \Prob(X<-L(v)^{-1}))\bigg| ~\leq~
K_1 W \quad   \text{a.s.}
\end{equation*}
It thus suffices to show that, for $I_j(t) := \int_0^t \Prob(X_j>x)-\Prob(X_j<-x) \, \dx$, $t \geq 0$,
\begin{align}   \label{eq:suffices}
\limsup_{t \to \infty} \bigg|\sum_{v \in \mathcal{T}_t} L(v) I_j\big(L(v)^{-1}\big) \bigg|  ~\leq~  K_2 W   \quad   \text{a.s.}
\end{align}
for some finite constant $K_2 > 0$.
We write $I_j\big(L(v)^{-1}\big) = I_j\big(L(v)^{-1}\big) - I_j(e^t) + I_j(e^t)$ and observe that by
\eqref{eq:Prob(X_j>t)} and \eqref{eq:crucial limit},
\begin{align}   \label{eq:suffices part I}
\limsup_{t \to \infty} \bigg|\sum_{v \in \mathcal{T}_t} L(v)
\big(I_j\big(L(v)^{-1}\big)-I_j(e^t)\big) \bigg| ~\leq~ K_3
W \quad   \text{a.s.}
\end{align}
for some finite constant $K_3 > 0$.
It thus suffices to show that
\begin{equation}    \label{eq:suffices part II}
\limsup_{t \to \infty} |I_j(e^t)| \sum_{v \in \mathcal{T}_t} L(v)   ~\leq~  K_4 W   \quad   \text{a.s.}
\end{equation}
If $\limsup_{t \to \infty} |I_j(e^{t})|/D(e^{-t}) =
\infty$, then using \eqref{eq:twolimits} we infer $\lim_{t
\to \infty} |I_j(e^t)| \sum_{v \in \mathcal{T}_t} L(v) = \infty$
a.s.~on the survival set $S$. This implies $|W(1)_j| = \infty$
a.s.~on $S$, thereby leading to a contradiction, for the absolute
value of any other term that contributes to $W(1)_j$ is bounded by
a constant times $W$ a.s. Therefore, $\limsup_{t \to \infty} |I_j(e^{t})|/D(e^{-t}) < \infty$ a.s.~which together with
\eqref{eq:twolimits} proves \eqref{eq:suffices part II}.
\end{proof}

For the next theorem, recall that $Z := \lim_{n \to \infty} Z_n = \lim_{n \to \infty} \sum_{|v|=n} L(v)$
whenever the limit exists in the a.s.\ sense, and $Z=0$, otherwise.

\begin{Thm} \label{Thm:Phi evaluated}
Assume that \eqref{eq:A1}-\eqref{eq:A4} hold. Let $\phi$ be
the Fourier transform of a probability distribution on $\R^d$
solving \eqref{eq:FE generalized stable}, and let $\Phi=\exp(\Psi)$
be the limit of the multiplicative martingale corresponding to $\phi$.
\begin{itemize}
    \item[(a)]
        Let $0 < \alpha < 1$.
        Then there exists a finite measure $\sigma$ on $\Sd$ such that
        \begin{equation}    \label{eq:Psi alpha<1}
        \Psi(\bt) = -W \int |\langle \bt, \bs \rangle|^{\alpha} \, \sigma(\dbs)
        + \imag W \tan\big(\frac{\pi \alpha}{2}\big) \int \langle \bt, \bs \rangle |\langle \bt, \bs \rangle|^{\alpha-1} \, \sigma(\dbs)
        \end{equation}
        a.s.\ for all $\bt \in \R^d$.
        If $\mathbb{G}(T)=\R^*$, then $\sigma$ is symmetric and the second integral in \eqref{eq:Psi alpha<1} vanishes.
    \item[(b)]
        Let $\alpha = 1$.

        \noindent
        (b1) Assume that Case I prevails and that $\E [\sum_{j \geq 1} |T_j| (\log^-(|T_j|))^2] < \infty$.
        Then there exist an $\ba \in \R^d$ and a finite measure $\sigma$ on $\Sd$ with $\int s_j \, \sigma(\dbs) = 0$ for $j=1,\ldots,d$
        such that
        \begin{equation}    \label{eq:Psi alpha=12}
        \Psi(\bt) ~=~   \imag W \langle \ba, \bt \rangle - W \int |\langle \bt, \bs \rangle| \, \sigma(\dbs)-\imag W \frac{2}{\pi}
        \int \langle \bt, \bs \rangle \log (|\langle \bt, \bs \rangle|) \, \sigma(\dbs)    \quad   \text{a.s.}
        \end{equation}
        for all $\bt \in \R^d$.

        \noindent
        (b2)
        Assume that Case II or III prevails and that $\E [\sum_{j \geq 1} |T_j| (\log^-(|T_j|))^2] < \infty$ in Case
        II and that \eqref{eq:A5} holds in Case III. Then there exist a finite symmetric measure $\sigma$ on
        $\Sd$ such that
        \begin{equation}    \label{eq:Psi alpha=11}
        \Psi(\bt) ~=~  - W \int |\langle \bt, \bs \rangle| \, \sigma(\dbs)    \quad   \text{a.s.\ for all } \bt \in \R^d.
        \end{equation}

    \item[(c)]
        Let $1 < \alpha < 2$.
        Then there exist an $\ba \in \R^d$ and a finite measure $\sigma$ on $\Sd$ such that
        \begin{equation}    \label{eq:Psi 1<alpha<2}
        \Psi(\bt)   =
        \imag Z \langle \ba, \bt \rangle
        - W \! \int \!\! |\! \langle \bt, \bs \rangle \! |^{\alpha} \, \sigma(\dbs)
        + \imag W \tan \big(\frac{\pi \alpha}{2}\big) \! \int \!\! \langle \bt, \bs \rangle |\!\langle \bt, \bs \rangle \! |^{\alpha-1} \, \sigma(\dbs)
        \end{equation}
        a.s.\ for all $\bt \in \R^d$.
        If $\mathbb{G}(T)=\R^*$, then $\sigma$ is symmetric and the second integral in \eqref{eq:Psi 1<alpha<2} vanishes.
    \item[(d)]
        Let $\alpha = 2$. If $\E[Z_1] = 1$, then additionally assume that \eqref{eq:A4a} holds.
        Then there exist an $\ba \in \R^d$ and a symmetric positive semi-definite (possibly zero) $d \times d$ matrix $\Sigma$ such that
        \begin{equation}    \label{eq:Psi alpha=2}
        \Psi(\bt) ~=~   \imag \langle \mathbf{a}, \bt \rangle - W \frac{\bt \Sigma \bt^{\!\textsf{T}}}{2}   \quad   \text{a.s.\ for all } \bt \in \R^d.
        \end{equation}
        If $\Prob(Z_1=1) < 1$, then $\ba = \bnull$.
    \item[(e)]
        Let $\alpha > 2$. Then there is an $\ba \in \R^d$ such that
        \begin{equation}    \label{eq:Psi alpha>2}
        \Psi(\bt) ~=~   \imag \langle \ba, \bt \rangle  \quad   \text{a.s.\ for all } \bt \in \R^d.
        \end{equation}
        If $\Prob(Z_1=1) < 1$, then $\ba = \bnull$.
\end{itemize}
\end{Thm}
\begin{proof}
We start by recalling that $\Psi$ satisfies \eqref{eq:char_exponent} by Proposition \ref{Prop:Disintegration}.

\noindent (e) If $\alpha > 2$, then, according to Lemma
\ref{Lem:nu_evaluated}, $\mathbf{\Sigma} = 0$ a.s.~and $\nu=0$
a.s. \eqref{eq:char_exponent} then simplifies to $\Psi(\bt) =
\imag \langle \W, \bt \rangle$ and \eqref{eq:M's equation} implies
that $\W = \sum_{|v|=n} L(v) [\W]_v$ a.s.~for all $n \in \N_0$.
Hence, each component of $\W$ is an endogenous fixed point
w.r.t.~$T$. From Theorem \ref{Thm:endogeny}(c) we know that
non-null endogenous fixed points w.r.t.~$T$ exist only if $\E
[Z_1] = 1$ and $Z_n$ converges a.s.~and in mean to $Z$ and then each endogenous fixed point is a multiple of $Z$.
Now if $Z_1=1$ a.s., then $Z=1$ a.s.~and we arrive at $\W = \ba$ for some
$\ba \in \R^d$ which is equivalent to \eqref{eq:Psi alpha>2}. If
$\Prob(Z_1=1) \not = 1$, we conclude from Theorem \ref{Thm:Z}(c)
that $Z=0$ a.s. Hence, \eqref{eq:Psi alpha>2} holds with $\ba = \bnull$.

\noindent (d) Let $\alpha = 2$.
By Lemma \ref{Lem:nu_evaluated}, $\Psi$ is of the form
\begin{equation*}
\Psi(\bt)   ~=~ \imag \langle \W, \bt \rangle - W \frac{\bt \Sigma \bt^{\!\textsf{T}}}{2}   \quad   \text{a.s.}
\end{equation*}
for a deterministic symmetric positive semi-definite matrix
$\Sigma$. Since $\imag$ and $1$ are linearly independent, one
again concludes from \eqref{eq:M's equation} that $\W$ is an
endogenous fixed point w.r.t.~$T$. The proof of the remaining part
of assertion (d) proceeds as the corresponding part of the proof
of part (e) with the exception that if $\E[Z_1]=1$, \eqref{eq:A4a}
has to be assumed in order to apply Theorem \ref{Thm:Z}(c).

We now turn to the case $0 < \alpha < 2$. %Then,
By Lemma \ref{Lem:nu_evaluated}, $\mathbf{\Sigma} = 0$ a.s.~and there exists a finite measure $\tilde{\sigma}$ on
the Borel subsets of $\Sd$ such that
\begin{equation}    \tag{\ref{eq:Levy_measure}}
\nu(A)  ~=~ W  \iint_{\Sd \times (0,\infty)} \1_{A}(rs)
r^{-(1+\alpha)} \, \dr \, \tilde{\sigma}(\ds)
\end{equation}
for all Borel sets $A \subseteq \R^d$ a.s.
Plugging this into \eqref{eq:char_exponent}, we conclude that $\Psi$ is of the form
\begin{eqnarray}
\Psi(\bt)
& = &
\imag \langle \W, \bt \rangle + W \iint \left(e^{\imag r \langle \bt,\bs \rangle} - 1 - \frac{\imag r \langle \bt, \bs \rangle}{1+r^2} \right)
\, \frac{\dr}{r^{1+\alpha}} \, \tilde{\sigma}(\dbs) \notag  \\
& = & \imag \langle \W, \bt \rangle + W \int I(\langle \bt, \bs
\rangle) \, \tilde{\sigma}(\ds), \quad	\bt\in \R^d,
\label{eq:Psi almost known}
\end{eqnarray}
where
\begin{equation*}
I(t)    ~:=~    \int_0^{\infty} \! \Big(e^{\imag tr} \!-\! 1 \!-\!
\frac{\imag tr}{1\!+\!r^2} \Big) \frac{\dr}{r^{1+\alpha}}, \quad	t \in \R.
\end{equation*}
The value of $I(t)$ is known (see \textit{e.g.}\ \cite[pp.\,168]{Gnedenko+Kolmogorov:1968})
\begin{eqnarray*}
I(t)
~=~
\begin{cases}
\imag c t - t^{\alpha} e^{-\frac{\pi \imag \alpha}{2}}\frac{\Gamma(1-\alpha)}{\alpha},                  &   0 < \alpha < 1, \\
\imag c t - (\pi/2) t - \imag t \log (t),                                                                   &   \alpha = 1, \\
\imag c t + t^{\alpha} e^{-\frac{\pi \imag}{2} \alpha}\frac{\Gamma(2-\alpha)}{\alpha(\alpha-1)},        &   1 < \alpha < 2,
\end{cases}
\end{eqnarray*}
for $t > 0$, where $\Gamma$ denotes Euler's Gamma function
and $c \in \R$ is a constant that depends only on $\alpha$.
Further, $I(-t) = \overline{I(t)}$, the complex conjugate of
$I(t)$.
Finally, observe that $\bs_0 := \int \bs \, \tilde{\sigma}(\dbs)\in \R^d$ since it is the integral of a
function which is bounded on $\Sd$ w.r.t.~to a finite measure.

\noindent (a) Let $0 < \alpha < 1$. Plugging in the corresponding
value of $I(t)$ in \eqref{eq:Psi almost known} gives
\begin{eqnarray*}
\Psi(\bt)
& = &
\imag \langle \W + Wc\bs_0, \bt \rangle \\
& &
- W \frac{\Gamma(1\!-\!\alpha)}{\alpha}
\bigg( e^{-\frac{\pi \imag}{2}\alpha} \int_{\{\langle \bt,\bs \rangle > 0\}} \langle \bt, \bs \rangle^{\alpha} \, \tilde{\sigma}(\dbs)
+e^{\frac{\pi \imag}{2}\alpha} \int_{\{\langle \bt,\bs \rangle < 0\}} |\langle \bt, \bs \rangle|^{\alpha} \, \tilde{\sigma}(\dbs)\bigg) \\
& = &
\imag \langle \W + Wc\bs_0, \bt \rangle - \frac{W}{\alpha} \Gamma(1\!-\!\alpha)
\cos \Big(\frac{\pi \alpha}{2}\Big) \int |\langle \bt, \bs \rangle|^{\alpha} \, \tilde{\sigma}(\dbs)    \\
& &
+ \imag \frac{W}{\alpha} \Gamma(1\!-\!\alpha) \sin \Big(\frac{\pi \alpha}{2}\Big) \int \langle \bt, \bs \rangle |\langle \bt, \bs \rangle|^{\alpha-1} \, \tilde{\sigma}(\dbs).
\end{eqnarray*}
Now define $\sigma := \frac{\Gamma(1\!-\!\alpha)}{\alpha} \cos
\big(\frac{\pi \alpha}{2}\big) \tilde{\sigma}$ and notice that
$\cos (\frac{\pi \alpha}{2}) > 0$ since $0<\alpha<1$. Then we get
\begin{eqnarray*}
\Psi(\bt) & = & \imag \langle \tilde\W, \bt \rangle - W \int
|\langle \bt, \bs \rangle|^{\alpha} \, \sigma(\dbs) + \imag W
\tan\Big(\frac{\pi \alpha}{2}\Big) \int \langle \bt, \bs \rangle
|\langle \bt, \bs \rangle|^{\alpha-1} \, \sigma(\dbs),
\end{eqnarray*}
where $\tilde{\W} := \W + Wc\bs_0$. Now, using linear
independence of $1$ and $\imag$ and \eqref{eq:M's equation}, we
conclude that
\begin{align*}
\langle \tilde{\W}&, \bt \rangle + W \tan\Big(\frac{\pi \alpha}{2}\Big) \int \langle \bt, \bs \rangle |\langle \bt, \bs \rangle|^{\alpha-1} \, \sigma(\dbs) \\
& =~ \sum_{|v|=n} L(v) \langle [\tilde{\W}]_v, \bt \rangle + \sum_{|v|=n} L(v) |L(v)|^{\alpha-1} [W]_v \tan\Big(\frac{\pi \alpha}{2}\Big) \int \langle \bt, \bs \rangle |\langle \bt, \bs \rangle|^{\alpha-1} \, \sigma(\dbs)
\end{align*}
for all $\bt \in \R^d$ and all $n \in \N_0$ a.s.
For each $j = 1,\ldots,d$, evaluate the above equation at $\bt = t
\mathbf{e}_j$ for some arbitrary $t > 0$ and with $\mathbf{e}_j$
denoting the $j$th unit vector. Then divide by $t$ and let $t \to \infty$.
This gives that each coordinate of $\tilde{\W}$ is an endogenous fixed point w.r.t.~$T$
which, therefore, must vanish by Theorem \ref{Thm:endogeny}(a).
If $\mathbb{G}(T) = \R^*$, then $\sigma$ is symmetric by Lemma
\ref{Lem:nu_evaluated} and the integral $\int \langle \bt, \bs
\rangle |\langle \bt, \bs \rangle|^{\alpha-1} \,
\sigma(\dbs)$ is equal to zero.
The proof of (a) is thus complete.

\noindent
(b) Let $\alpha=1$.
Again, we plug the corresponding value of $I(t)$ in \eqref{eq:Psi almost known}.
With $\tilde{\W} := \W + W c \bs_0$ and $\sigma := \frac{\pi}{2} \tilde{\sigma}$, this yields
\begin{equation}    \label{eq:Psi when alpha=1}
\Psi(\bt)
~=~ \imag \langle \tilde{\W}, \bt \rangle
- W \int |\langle \bt,\bs \rangle| \, \sigma(\ds) - \imag W \frac{2}{\pi} \int \langle \bt, \bs \rangle \log (|\langle \bt, \bs \rangle|) \, \sigma(\dbs)
\end{equation}
for all $\bt \in \R^d$ a.s.

\noindent
(b2)
The measure $\sigma$ is symmetric by Lemma \ref{Lem:nu_evaluated} and the last
integral in \eqref{eq:Psi when alpha=1} vanishes because the
integrand is odd. Now combine \eqref{eq:M's equation} and
\eqref{eq:Psi when alpha=1} and use the linear independence of $1$
and $\imag$ to conclude
\begin{equation}    \label{eq:Psi when alpha=1 and G(T)=R^*}
\langle \tilde{\W}, \bt \rangle
~=~ \sum_{|v|=n} L(v) \langle [\tilde{\W}]_v, \bt \rangle
\end{equation}
for all $\bt \in \R^d$ a.s. Choosing $\bt = \mathbf{e}_j$ for
$j=1,\ldots,d$, we see that each coordinate of $\tilde{\W}$ is an
endogenous fixed point w.r.t.~$T$ which must vanish a.s.~by Theorem \ref{Thm:endogeny}(b).

\noindent
(b1)
%Now assume that $\mathbb{G}(T) = \Rp$ and
%that $\E[\sum_{j \geq 1} T_j (\log^-(T_j))^2] < \infty$.
We show that $\int s_j \, \sigma(\dbs) = 0$ for $j=1,\ldots,d$ or,
equivalently, $\mathbf{s}_0 = \bnull$. To this end, use
\eqref{eq:M's equation} and the linear independence of $1$ and
$\imag$ to obtain that
\begin{align}
\langle \tilde{\W}, \bt \rangle - & W \frac{2}{\pi} \int \langle \bt, \bs \rangle \log (|\langle \bt, \bs \rangle|) \, \sigma(\dbs) \notag  \\
%& =~ \sum_{|v|=n} L(v) \langle [\tilde{\W}]_v, \bt \rangle
%- \sum_{|v|=n} L(v) [W]_v \frac{2}{\pi} \int \langle \bt, \bs \rangle \log (|\langle L(v) \bt, \bs \rangle|) \, \sigma(\dbs)   \notag  \\
& =~ \sum_{|v|=n} L(v) \langle [\tilde{\W}]_v, \bt \rangle
- \sum_{|v|=n} L(v) \log (L(v)) [W]_v \frac{2}{\pi} \int \langle \bt, \bs \rangle \, \sigma(\dbs)   \notag  \\
& \hphantom{=~ \sum_{|v|=n} L(v) \langle [\tilde{\W}]_v, \bt \rangle\ }
- \sum_{|v|=n} L(v) [W]_v \frac{2}{\pi} \int \langle \bt, \bs \rangle \log (|\langle \bt, \bs \rangle|) \,\sigma(\dbs)  \label{eq:evaluate at J(t)=0}
\end{align}
a.s.\ for all $n \in \N_0$.
Assume for a contradiction that for some $j \in \{1,\ldots,d\}$, we have $\int s_j \sigma(\dbs) \not = 0$.
Then put $J(\bt) := \int \langle \bt, \bs \rangle \log (|\langle \bt, \bs \rangle|) \, \sigma(\dbs)$, $\bt \in \R^d$.
For $u \not = 0$, one has
\begin{eqnarray*}
J(u \mathbf{e}_j)
& = &
\int u s_j \log (|u s_j|) \, \sigma(\dbs)
~=~ u  \log (|u|) \int s_j \, \sigma(\dbs) + u \int s_j \log (|s_j|) \, \sigma(\dbs).
\end{eqnarray*}
Thus, $J(u \mathbf{e}_j) = 0$ iff
\begin{equation*}
\log (|u|)  ~=~ - \frac{\int s_j \log (|s_j|) \, \sigma(\dbs)}{\int s_j \, \sigma(\dbs)}.
\end{equation*}
Since we assume $\int s_j \sigma(\dbs) \not = 0$, one can
choose $u\not=0$ such that $J(u \mathbf{e}_j)=0$. Evaluating
\eqref{eq:evaluate at J(t)=0} at $u \mathbf{e}_j$ and then
dividing by $u\not=0$ gives
\begin{equation*}   \label{eq:at J(ue_j)}
\tilde{W}_j ~=~ \sum_{|v|=n} L(v) [\tilde{W}_j]_v   - \sum_{|v|=n}
L(v) \log(L(v)) [W]_v \frac{2}{\pi} \int s_j \, \sigma(\dbs),
\end{equation*}
where $\tilde{W}_j$ is the $j$th coordinate of $\tilde{\W}$.
Using \eqref{eq:bW ladder} we infer
\begin{eqnarray*}
\tilde{\W} & = & \W + Wc \bs_0 ~=~ \W(1) + \!\!\!
\underset{\{|\bx| > 1\}}{\int} \frac{\bx}{1+|\bx|^2} \nu(\dbx) -
\!\!\! \underset{\{|\bx| \leq 1\}}{\int} \frac{\bx
|\bx|^2}{1+|\bx|^2} \nu(\dbx) + W c \bs_0 \quad \text{a.s.,}
\end{eqnarray*}
where $\W(1) = \lim_{t \to \infty} \sum_{v \in \mathcal{T}_t} L(v)
\int_{\{|\bx| \leq |L(v)|^{-1}\}}  \bx \, F(\dbx)$. Since we know
from Lemma \ref{Lem:nu_evaluated} that all randomness in $\nu$
comes from a scalar factor $W$, we conclude that $\tilde{\W} =
\W(1)+\tilde{\mathbf{c}} W$ a.s.~for some
$\tilde{\mathbf{c}} \in \R^d$. From Lemma \ref{Lem:W(1)}, we know
that $|\tilde{W}_j| \leq K W$ a.s.\ for some $K \geq 0$.
In the case $\int s_j \, \sigma(\dbs) < 0$ we use these observations to conclude
\begin{align*}
-KW~\leq~   \tilde{W}_j &~=~    \sum_{|v|=n} L(v) [\tilde{W}_j]_v   -  \sum_{|v|=n} L(v) \log (L(v)) [W]_v \frac{2}{\pi} \int s_j \, \sigma(\dbs)   \\
&~\leq~ KW - \sum_{|v|=n} L(v) \log (L(v)) [W]_v \frac{2}{\pi} \int s_j \,
\sigma(\dbs) ~\to~   -\infty
\end{align*}
a.s.~on $S$, the set of survival since $\sum_{|v|=n} L(v) [W]_v = W > 0$ a.s.~on $S$ and $\sup_{|v|=n} L(v) \to 0$ a.s.~by Lemma \ref{Lem:sup->0}.
This is a contradiction.
Analogously, one can produce a contradiction when $\int s_j \, \sigma(\dbs) > 0$.
Consequently, $\int s_j \sigma(\dbs) = 0$ for $j=1,\ldots,d$.
Using this and the equation \eqref{eq:purely_tree-based} for $W$ in \eqref{eq:evaluate at J(t)=0}, we conclude that
\begin{align*}
\langle \tilde{\W}, \bt \rangle  ~=~ \sum_{|v|=n} L(v) \langle [\tilde{\W}]_v, \bt \rangle
\end{align*}
a.s.\ for all $n \in \N_0$. Evaluating this equation at $\bt =
\mathbf{e}_j$ shows that $\tilde{W}_j$ is an endogenous fixed
point w.r.t.~$T$, hence $\tilde{W}_j = W a_j$ a.s.~by Theorem
\ref{Thm:endogeny}(b), $j=1,\ldots,d$. Hence, $\tilde{\W} = W \ba$
for $\ba = (a_1,\ldots,a_d)$. The proof of (b) is complete.

\noindent (c) 
Let $1 < \alpha < 2$. Plugging the corresponding value of $I(t)$ in \eqref{eq:Psi almost known}
and arguing as in the case $0<\alpha<1$ we infer
\begin{eqnarray*}
\Psi(\bt)
& = &
\imag \langle \tilde{\W}, \bt \rangle
- W
\bigg(\int |\langle \bt, \bs \rangle|^{\alpha} \, \sigma(\dbs)
- \imag \tan \big(\frac{\pi \alpha}{2}\big) \int \langle \bt, \bs \rangle^{\alpha} |\langle \bt, \bs \rangle|^{\alpha-1} \, \sigma(\dbs)\bigg)
\end{eqnarray*}
where $\sigma := - \alpha^{-1}(\alpha-1)^{-1} \Gamma(2-\alpha) \cos(\frac{\pi \alpha}{2}) \tilde{\sigma}$
(notice that $\cos(\frac{\pi \alpha}{2}) < 0$) and, as before, $\tilde{\W} := \W + c\bs_0$. The equality
$\tilde{\W} = \ba Z$ for some $\ba \in \R^d$ can be checked
as in the proof of the corresponding assertion in the case $\alpha=2$.
If $\mathbb{G}(T) = \R^*$, then $\sigma$ is symmetric by Lemma \ref{Lem:nu_evaluated}
and the integral $\int \langle \bt, \bs \rangle |\langle \bt, \bs \rangle|^{\alpha-1} \, \sigma(\dbs)$ vanishes.
\end{proof}

\subsection{Proof of the converse part of Theorem \ref{Thm:SF}} \label{subsec:solving inhomogeneous}

From what we have already derived in the preceding sections, there
is only a small step to go in order to prove the converse part of
Theorem \ref{Thm:SF}. The techniques needed to do this final step
have been developed in \cite{Alsmeyer+Meiners:2012}. Thus we
shall only give a sketch of the proof.

\begin{proof}[Proof of Theorem \ref{Thm:SF} (converse part)]
Fix any $\phi \in \SF$.
Then define the corresponding multiplicative martingale by setting
\begin{equation}    \label{eq:M_n inhomogeneous}
M_n(\bt) :=~ \exp\bigg(\imag \sum_{|v|<n} L(v) \langle \bC(v), \bt  \rangle\bigg) \prod_{|v|=n} \phi(L(v) \bt)
~=:~    \exp(\imag \langle \W_n^*, \bt \rangle) \Phi_n(\bt),
\end{equation}
$\bt \in \R^d$. From \eqref{eq:FE generalized stable inhom} one
can deduce just as in the homogeneous case that, for fixed $\bt
\in \R^d$, $(M_n(\bt))_{n \in \N_0}$ is a bounded martingale
w.r.t.~$(\A_n)_{n \in \N_0}$ and, thus, converges a.s.~and in mean
to a limit $M(\bt)$ with $\E[M(\bt)] = \phi(\bt)$. On the other
hand, $\W_n^* \to \W^*$ in probability implies
$\Phi_n(\bt) \to \Phi(\bt):= M(\bt)/\exp(\imag \langle \W^*, \bt
\rangle)$ in probability as $n \to \infty$. Mimicking the proof of Theorem 4.2 in
\cite{Alsmeyer+Meiners:2012}, one can show that $\psi(\bt) = \E[\Phi(\bt)]$ is a solution to \eqref{eq:FE generalized stable}
and that the $\Phi(\bt)$, $\bt \in \R^d$ are the limits of the
multiplicative martingales associated with $\psi$. Hence
$\Phi(\bt) = \exp(\Psi(\bt))$ for some $\Psi$ as in Theorem \ref{Thm:Phi evaluated}.
Finally, $\phi(\bt) = \E[M(\bt)] = \E[\exp(\imag \langle \W^*, \bt \rangle + \Psi(\bt))]$,
$\bt \in \R^d$. The proof is complete.
\end{proof}

\subsubsection*{Acknowledgement}
\footnotesize
The authors thank Gerold Alsmeyer for helpful discussions and Svante Janson for pointing out an important reference.
The research of M.\;Meiners was supported by DFG SFB 878 ``Geometry, Groups and Actions''.

\bibliographystyle{abbrv}
\bibliography{Fixed_points}

\def\cprime{$'$}
\begin{thebibliography}{10}

\bibitem{AB2005}
D.~J. Aldous and A.~Bandyopadhyay.
\newblock A survey of max-type recursive distributional equations.
\newblock {\em Ann. Appl. Probab.}, 15(2):1047--1110, 2005.

\bibitem{Alsmeyer:2002}
G.~Alsmeyer.
\newblock The minimal subgroup of a random walk.
\newblock {\em J. Theoret. Probab.}, 15(2):259--283, 2002.

\bibitem{Alsmeyer+Biggins+Meiners:2012}
G.~Alsmeyer, J.~D. Biggins, and M.~Meiners.
\newblock The functional equation of the smoothing transform.
\newblock {\em Ann. Probab.}, 40(5):2069--2105, 2012.

\bibitem{Alsmeyer+Damek+Mentemeier:2013}
G.~Alsmeyer, E.~Damek, and S.~Mentemeier.
\newblock Precise tail index of fixed points of the two-sided smoothing
  transform.
\newblock In {\em Random Matrices and Iterated Random Functions}, volume~53 of
  {\em Springer Proceedings in Mathematics \& Statistics}, pages 229--251,
  2013.

\bibitem{AIPR2009}
G.~Alsmeyer, A.~Iksanov, S.~Polotskiy, and U.~R{\"o}sler.
\newblock Exponential rate of {$L_p$}-convergence of intrinsic martingales in
  supercritical branching random walks.
\newblock {\em Theory Stoch. Process.}, 15(2):1--18, 2009.

\bibitem{Alsmeyer+Iksanov:2009}
G.~Alsmeyer and A.~M. Iksanov.
\newblock A log-type moment result for perpetuities and its application to
  martingales in supercritical branching random walks.
\newblock {\em Electron. J. Probab.}, 14:289--313 (electronic), 2009.

\bibitem{Alsmeyer+Kuhlbusch:2010}
G.~Alsmeyer and D.~Kuhlbusch.
\newblock Double martingale structure and existence of $\phi$-moments for
  weighted branching processes.
\newblock {\em M\"unster J. Math.}, 3:163--212, 2010.

\bibitem{Alsmeyer+Meiners:2012}
G.~Alsmeyer and M.~Meiners.
\newblock Fixed points of inhomogeneous smoothing transforms.
\newblock {\em J. Difference Equ. Appl.}, 18(8):1287--1304, 2012.

\bibitem{Alsmeyer+Meiners:2013}
G.~Alsmeyer and M.~Meiners.
\newblock Fixed points of the smoothing transform: two-sided solutions.
\newblock {\em Probab. Theory Related Fields}, 155(1-2):165--199, 2013.

\bibitem{Alsmeyer+Roesler:2006}
G.~Alsmeyer and U.~R{\"o}sler.
\newblock A stochastic fixed point equation related to weighted branching with
  deterministic weights.
\newblock {\em Electron. J. Probab.}, 11:no. 2, 27--56 (electronic), 2006.

\bibitem{Araman+Glynn:2006}
V.~F. Araman and P.~W. Glynn.
\newblock Tail asymptotics for the maximum of perturbed random walk.
\newblock {\em Ann. Appl. Probab.}, 16(3):1411--1431, 2006.

\bibitem{Bassetti+Ladelli:2012}
F.~Bassetti and L.~Ladelli.
\newblock Self-similar solutions in one-dimensional kinetic models: a
  probabilistic view.
\newblock {\em Ann. Appl. Probab.}, 22(5):1928--1961, 2012.

\bibitem{Bassetti+Ladelli+Matthes:2011}
F.~Bassetti, L.~Ladelli, and D.~Matthes.
\newblock Central limit theorem for a class of one-dimensional kinetic
  equations.
\newblock {\em Probab. Theory Related Fields}, 150(1-2):77--109, 2011.

\bibitem{Bassetti+Matthes:2014}
F.~Bassetti and D.~Matthes.
\newblock Multi-dimensional smoothing transformations: {E}xistence, regularity
  and stability of fixed points.
\newblock {\em Stochastic Process. Appl.}, 124(1):154--198, 2014.

\bibitem{Biggins:1977}
J.~D. Biggins.
\newblock Martingale convergence in the branching random walk.
\newblock {\em J. Appl. Probability}, 14(1):25--37, 1977.

\bibitem{Biggins:1979}
J.~D. Biggins.
\newblock Growth rates in the branching random walk.
\newblock {\em Z. Wahrsch. Verw. Gebiete}, 48(1):17--34, 1979.

\bibitem{Biggins:1998}
J.~D. Biggins.
\newblock Lindley-type equations in the branching random walk.
\newblock {\em Stoch. Proc. Appl.}, 75:105--133, 1998.

\bibitem{Biggins+Kyprianou:1997}
J.~D. Biggins and A.~E. Kyprianou.
\newblock {S}eneta-{H}eyde norming in the branching random walk.
\newblock {\em The Annals of Probability}, 25(1):337--360, 1997.

\bibitem{Biggins+Kyprianou:2005}
J.~D. Biggins and A.~E. Kyprianou.
\newblock Fixed points of the smoothing transform: the boundary case.
\newblock {\em Electron. J. Probab.}, 10:no. 17, 609--631 (electronic), 2005.

\bibitem{Bingham+Goldie+Teugels:1989}
N.~H. Bingham, C.~M. Goldie, and J.~L. Teugels.
\newblock {\em Regular variation}, volume~27 of {\em Encyclopedia of
  Mathematics and its Applications}.
\newblock Cambridge University Press, Cambridge, 1989.

\bibitem{Bobylev+Cercignani:2003}
A.~V. Bobylev and C.~Cercignani.
\newblock Self-similar asymptotics for the {B}oltzmann equation with inelastic
  and elastic interactions.
\newblock {\em J. Statist. Phys.}, 110(1-2):333--375, 2003.

\bibitem{Buraczewski:2009}
D.~Buraczewski.
\newblock On tails of fixed points of the smoothing transform in the boundary
  case.
\newblock {\em Stochastic Process. Appl.}, 119(11):3955--3961, 2009.

\bibitem{Buraczewski+Kolesko:2014}
D.~Buraczewski and K.~Kolesko.
\newblock Linear stochastic equations in the critical case.
\newblock {\em J. Difference Equ. Appl.}, 20(2):188--209, 2014.

\bibitem{Caliebe:2003}
A.~Caliebe.
\newblock Symmetric fixed points of a smoothing transformation.
\newblock {\em Advances in Applied Probability}, 35(2):377--394, 2003.

\bibitem{CT1997}
Y.~S. Chow and H.~Teicher.
\newblock {\em {P}robability {T}heory}.
\newblock Springer texts in Statistics. Springer, New York, third edition,
  1997.

\bibitem{Durrett+Liggett:1983}
R.~Durrett and T.~M. Liggett.
\newblock Fixed points of the smoothing transformation.
\newblock {\em Z. Wahrsch. Verw. Gebiete}, 64(3):275--301, 1983.

\bibitem{Gnedenko+Kolmogorov:1968}
B.~V. Gnedenko and A.~N. Kolmogorov.
\newblock {\em Limit distributions for sums of independent random variables}.
\newblock Translated from the Russian, annotated, and revised by K. L. Chung.
  With appendices by J. L. Doob and P. L. Hsu. Revised edition. Addison-Wesley
  Publishing Co., Reading, Mass.-London-Don Mills., Ont., 1968.

\bibitem{Hu+Shi:2009}
Y.~Hu and Z.~Shi.
\newblock Minimal position and critical martingale convergence in branching
  random walks, and directed polymers on disordered trees.
\newblock {\em Ann. Probab.}, 37(2):742--789, 2009.

\bibitem{Iksanov+Polotskiy:2006}
A.~Iksanov and S.~Polotskiy.
\newblock Regular variation in the branching random walk.
\newblock {\em Theory Stoch. Process.}, 12(1-2):38--54, 2006.

\bibitem{Iksanov:2004}
A.~M. Iksanov.
\newblock Elementary fixed points of the {BRW} smoothing transforms with
  infinite number of summands.
\newblock {\em Stochastic Process. Appl.}, 114(1):27--50, 2004.

\bibitem{Iksanov+Meiners:2014}
A.~M. Iksanov and M.~Meiners.
\newblock Rate of convergence in the law of large numbers for supercritical
  general multi-type branching processes, 2014.
\newblock arXiv:1401.1368v1.

\bibitem{Iksanov+Roesler:2006}
A.~M. Iksanov and U.~R{\"o}sler.
\newblock Some moment results about the limit of a martingale related to the
  supercritical branching random walk and perpetuities.
\newblock {\em Ukra\"\i n. Mat. Zh.}, 58(4):451--471, 2006.

\bibitem{Jelenkovic+Olvera:2012b}
P.~R. Jelenkovi{\'c} and M.~Olvera-Cravioto.
\newblock Implicit renewal theorem for trees with general weights.
\newblock {\em Stochastic Process. Appl.}, 122(9):3209--3238, 2012.

\bibitem{Jelenkovic+Olvera:2012a}
P.~R. Jelenkovi{\'c} and M.~Olvera-Cravioto.
\newblock Implicit renewal theory and power tails on trees.
\newblock {\em Adv. in Appl. Probab.}, 44(2):528--561, 2012.

\bibitem{Kac:1956}
M.~Kac.
\newblock Foundations of kinetic theory.
\newblock In {\em Proceedings of the {T}hird {B}erkeley {S}ymposium on
  {M}athematical {S}tatistics and {P}robability, 1954--1955, vol. {III}}, pages
  171--197, Berkeley and Los Angeles, 1956. University of California Press.

\bibitem{Kallenberg:2002}
O.~Kallenberg.
\newblock {\em Foundations of Modern Probability}.
\newblock Probability and its Applications (New York). Springer-Verlag, New
  York, second edition, 2002.

\bibitem{Kuelbs:1973}
J.~Kuelbs.
\newblock A representation theorem for symmetric stable processes and stable
  measures on {$H$}.
\newblock {\em Z. Wahrscheinlichkeitstheorie und Verw. Gebiete}, 26:259--271,
  1973.

\bibitem{Liang+Liu:2011}
X.~Liang and Q.~Liu.
\newblock Tail behavior of laws stable by random weighted mean.
\newblock {\em C. R. Math. Acad. Sci. Paris}, 349(5-6):347--352, 2011.

\bibitem{Liu:1998}
Q.~Liu.
\newblock Fixed points of a generalized smoothing transformation and
  applications to the branching random walk.
\newblock {\em Advances in Applied Probability}, 30(1):85--112, 1998.

\bibitem{Liu:2000}
Q.~Liu.
\newblock On generalized multiplicative cascades.
\newblock {\em Stochastic Process. Appl.}, 86(2):263--286, 2000.

\bibitem{Liu:2001}
Q.~Liu.
\newblock Asymptotic properties and absolute continuity of laws stable by
  random weighted mean.
\newblock {\em Stochastic Process. Appl.}, 95(1):83--107, 2001.

\bibitem{Lyons:1997}
R.~Lyons.
\newblock A simple path to {B}iggins' martingale convergence for branching
  random walk.
\newblock In {\em Classical and modern branching processes ({M}inneapolis,
  {MN}, 1994)}, volume~84 of {\em IMA Vol. Math. Appl.}, pages 217--221.
  Springer, New York, 1997.

\bibitem{Matthes+Toscani:2008}
D.~Matthes and G.~Toscani.
\newblock Analysis of a model for wealth redistribution.
\newblock {\em Kinet. Relat. Models}, 1(1):1--27, 2008.

\bibitem{Mentemeier:2013}
S.~Mentemeier.
\newblock The fixed points of the multivariate smoothing transform, 2013.
\newblock arXiv:1309.0733v2.

\bibitem{Neininger+Rueschendorf:2004}
R.~Neininger and L.~R{\"u}schendorf.
\newblock A general limit theorem for recursive algorithms and combinatorial
  structures.
\newblock {\em Ann. Appl. Probab.}, 14(1):378--418, 2004.

\bibitem{Pulvirenti+Toscani:2004}
A.~Pulvirenti and G.~Toscani.
\newblock Asymptotic properties of the inelastic {K}ac model.
\newblock {\em J. Statist. Phys.}, 114(5-6):1453--1480, 2004.

\bibitem{Ramachandran+Lau+Gu:1988}
B.~Ramachandran, K.-S. Lau, and H.~M. Gu.
\newblock On characteristic functions satisfying a functional equation and
  related classes of simultaneous integral equations.
\newblock {\em Sankhy\=a Ser. A}, 50(2):190--198, 1988.

\bibitem{Rao+Shanbhag:1994}
C.~R. Rao and D.~N. Shanbhag.
\newblock {\em Choquet-{D}eny type functional equations with applications to
  stochastic models}.
\newblock Wiley Series in Probability and Mathematical Statistics: Probability
  and Mathematical Statistics. John Wiley \& Sons Ltd., Chichester, 1994.

\bibitem{Roesler:1991}
U.~R{\"o}sler.
\newblock A limit theorem for ``{Q}uicksort''.
\newblock {\em RAIRO Inform. Th\'eor. Appl.}, 25(1):85--100, 1991.

\bibitem{Roesler+Topchii+Vatutin:2000}
U.~R{\"o}sler, V.~A. Topchi{\u\i}, and V.~A. Vatutin.
\newblock Convergence conditions for branching processes with particles having
  weight.
\newblock {\em Diskret. Mat.}, 12(1):7--23, 2000.

\bibitem{Samorodnitsky+Taqqu:1994}
G.~Samorodnitsky and M.~S. Taqqu.
\newblock {\em Stable non-{G}aussian Random Processes}.
\newblock Stochastic Modeling. Chapman \& Hall, New York, 1994.
\newblock Stochastic Models with Infinite Variance.

\bibitem{Volkovich+Litvak:2010}
Y.~Volkovich and N.~Litvak.
\newblock Asymptotic analysis for personalized web search.
\newblock {\em Adv. Appl. Probab.}, 42:577--604, 2010.

\end{thebibliography}

\vspace{0.1cm}
\noindent
\textsc{Alexander Iksanov   \\
Faculty of Cybernetics  \\
National T.\ Shevchenko University of Kyiv,\\
01601 Kyiv, Ukraine,    \\
Email: iksan@univ.kiev.ua}   \\
\vspace{0.05cm}

\noindent
\textsc{Matthias Meiners    \\
Fachbereich Mathematik  \\
Technische Universit\"at Darmstadt  \\
64289 Darmstadt, Germany    \\
Email: meiners@mathematik.tu-darmstadt.de}
\end{document}